\documentclass[letterpaper,11pt,reqno]{amsart}

\usepackage{amsmath,amssymb,amscd,amsthm,amsxtra}
\usepackage[implicit=true]{hyperref}
\usepackage{bbm}
\usepackage{cite}
\usepackage{tikz}
\usetikzlibrary{matrix} 
\usepackage{enumerate}

\allowdisplaybreaks[2]

\usepackage{color}

\renewcommand{\H}{\mathcal{H}}
\newtheorem{theorem}{Theorem}[section]

\newtheorem{lemma}[theorem]{Lemma}
\newtheorem{proposition}[theorem]{Proposition}
\newtheorem{remark}[theorem]{Remark}

\newtheorem{definition}[theorem]{Definition}
\newtheorem{corollary}[theorem]{Corollary}

\newtheorem*{ack}{Acknowledgments}

\DeclareMathOperator*{\intt}{\int}

\newcommand{\I}{\hspace{0.5mm}\text{I}\hspace{0.5mm}}
\newcommand{\II}{\text{I \hspace{-2.8mm} I} }
\newcommand{\noi}{\noindent}
\newcommand{\Z}{\mathbb{Z}}
\newcommand{\R}{\mathbb{R}}
\newcommand{\C}{\mathbb{C}}
\newcommand{\T}{\mathbb{T}}

\newcommand{\1}{\mathbbm{1}}

\newcommand{\D}{\mathcal{D}}
\newcommand{\G}{\mathbf{G}}
\newcommand{\B}{\mathbf{B}}

\let\Im=\undefined\DeclareMathOperator*{\Im}{Im}

\let\P= \undefined
\newcommand{\P}{\mathbf{P}}

\newcommand{\NB}{\mathbb{N}}

\newcommand{\FL}{\mathcal{F}L} %%%%%%%%%%%% Fourier-Lebesgue spaces

\renewcommand{\S}{\mathcal{S}}

\def\norm#1{\|#1\|}

\newcommand{\al}{\alpha}
\newcommand{\be}{\beta}
\newcommand{\dl}{\delta}

\newcommand{\X}{\mathbb{X}}
\newcommand{\eps}{\varepsilon}

\newcommand{\ld}{\lambda}

\newcommand{\s}{\sigma}

\newcommand{\ft}{\widehat}
\newcommand{\Ft}{{\mathcal{F}}}

\newcommand{\embeds}{\hookrightarrow}

\newcommand{\nbar}{\overline{n}}

\newcommand{\conj}[1]{\overline{#1}}

\renewcommand{\l}{\ell}

\newcommand{\les}{\lesssim}
\newcommand{\ges}{\gtrsim}

\newcommand{\jb}[1]
{\langle #1 \rangle}

\usepackage{bbm}

\DeclareMathOperator{\sgn}{sgn}

\numberwithin{equation}{section}
\numberwithin{theorem}{section}

\begin{document}
\baselineskip = 14pt

\title[A refined well-posedness result for mKdV in Fourier-Lebesgue spaces]
{A refined well-posedness result for the modified KdV equation in the Fourier-Lebesgue spaces}

\author[A.~Chapouto]
{Andreia Chapouto}

\address{
 Andreia Chapouto\\ Maxwell Institute for Mathematical Sciences
 and 
School of Mathematics\\
The University of Edinburgh\\
and The Maxwell Institute for the Mathematical Sciences\\
James Clerk Maxwell Building\\
The King's Buildings\\
 Peter Guthrie Tait Road\\
Edinburgh\\ 
EH9 3FD\\United Kingdom} 

\email{andreia.chapouto@ed.ac.uk}



%\date{\today}

% ---------------------------------------ABSTRACT-----------------------------------

% ---------------------------------------ABSTRACT-----------------------------------

\begin{abstract}
	
	We study the well-posedness of the complex-valued modified Korteweg-de Vries equation (mKdV) on the circle at low regularity.
	 In our previous work (2019), we introduced the second renormalized mKdV equation, based on the conservation of momentum, which we proposed as the correct model to study the complex-valued mKdV outside of $H^\frac12(\T)$.
	  Here, we employ the method introduced by Deng-Nahmod-Yue (2019) to prove local well-posedness of the second renormalized mKdV equation in the Fourier-Lebesgue spaces $\FL^{s,p}(\T)$ for $s\geq \frac12$ and $1\leq p <\infty$.
	  As a byproduct of this well-posedness result, we show ill-posedness of the complex-valued mKdV without the second renormalization for initial data in these Fourier-Lebesgue spaces with infinite momentum.
	
\end{abstract}

\vspace*{-10mm}

\maketitle

\tableofcontents

\section{Introduction}
\subsection{Modified Korteweg-de Vries equation}
We consider the Cauchy problem for the complex-valued modified Korteweg-de Vries equation (mKdV) on the one-dimensional torus~$\T=\R / (2\pi\Z)$:
\begin{align}
\begin{cases}
\partial_t u + \partial_x^3 u = \pm |u|^2 \partial_x u, \\
u\vert_{t=0} = u_0,
\end{cases} \quad (t,x) \in \R\times \T,
\label{mkdv}
\end{align}
where $u$ is a complex-valued function. The complex-valued mKdV equation \eqref{mkdv} appears as a model for the dynamical evolution of nonlinear lattices, fluid dynamics and plasma physics (see \cite{RRECFM,HWLPE}, for example). This equation is a completely integrable complex-valued generalization of the usual mKdV equation
\begin{equation}
\partial_t u + \partial_x^3 u = \pm u^2 \partial_x u,\label{mkdv_real}
\end{equation}
also referred to as the mKdV equation of Hirota \cite{Hir}.
Indeed, real-valued solutions of \eqref{mkdv} are also solutions of \eqref{mkdv_real}. From the completely integrable structure, it follows that \eqref{mkdv} has an infinite number of conservation laws. In particular, the mass and the momentum play an important role:
\begin{align*}
\text{Mass:} \quad \mu\big(u(t)\big) & = \frac{1}{2\pi} \intt_\T |u(t)|^2 \, dx ,\\
\text{Momentum:} \quad P\big(u(t) \big) & = \frac{1}{2\pi} \Im \intt_\T u(t) \partial_x \conj{u}(t) \, dx.
\end{align*}
Exploiting the conservation of mass $\mu\big(u(t)\big) = \mu(u_0)$ for solutions $u\in C\big(\R;L^2(\T)\big)$ of \eqref{mkdv}, we can consider the first renormalized mKdV equation (mKdV1):
\begin{align}
\partial_t u + \partial_x^3 u = \pm \Big( |u|^2 - \mu(u) \Big)\partial_x u. \label{renorm}
\end{align}
The mKdV1 equation \eqref{renorm} is equivalent to mKdV \eqref{mkdv} in $L^2(\T)$ in the following sense: $u\in C\big(\R;L^2(\T)\big)$ is a solution of \eqref{mkdv} if and only if $\mathcal{G}_1(u)(t,x) := u \big(t,x\mp\mu(u)t \big)$ is a solution of \eqref{renorm}.
In \cite{BO94}, Bourgain showed that \eqref{renorm} is locally well-posed in $H^s(\T)$ for $s\geq \frac12$, using the Fourier restriction norm method. 

In \cite{A}, we showed that the momentum plays an important role in the well-posedness of the complex-valued mKdV equation \eqref{mkdv}. In fact, outside of $H^\frac12(\T)$, the momentum is no longer conserved nor finite, preventing us from making sense of the nonlinearity. There, we established $H^\frac12(\T)$ as the limit for the well-posedness theory of the complex-valued mKdV equation \eqref{mkdv}. In particular, we proved the non-existence of solution for initial data with infinite momentum. Our analysis was based on the Fourier-Lebesgue spaces $\FL^{s,p}(\T)$, defined through the following norm
\begin{equation}
\| u_0\|_{\FL^{s,p}} = \big\| \jb{n}^s \ft{u}_0 \|_{\l^p_n},
\end{equation} 
where $\jb{\cdot} = (1+|\cdot|^2)^\frac12$. More specifically, we are interested in $\FL^{s,p}(\T)$ for $s\geq \frac12$ and $2\leq p <\infty$, since $H^\frac12(\T) \subsetneq \FL^{\frac12,p}(\T)$ for these choices of parameters.

In this paper, we extend the previously mentioned ill-posedness result on the complex-valued mKdV equation \eqref{mkdv} as follows.
\begin{theorem}
	Let $s\geq \frac12$ and $2\leq p <\infty$. Suppose that $u_0\in\FL^{s,p}(\T)$ has infinite momentum in the sense that
	\begin{align}
	|P(\P_{\leq N} u_0 )| \to \infty \text{ as } N\to \infty,\label{ill-posed}
	\end{align}
	where $\P_{\leq N}$ denotes the Dirichlet projection onto the spatial frequencies $\{|n|\leq N\}$. Then, for any $T>0$, there exists no distributional solution $u\in C([-T,T];\FL^{s,p}(\T))$ to the complex-valued mKdV equation \eqref{mkdv} satisfying the following conditions:
	\begin{enumerate}[\normalfont (i)]
		\item $u\vert_{t=0} = u_0$,
		\item The smooth global solutions $\{u_N\}_{N\in\NB}$ of mKdV \eqref{mkdv}, with $u_N\vert_{t=0} = \P_{\leq N}u_0$, satisfy $u_N \to u$ in $C([-T,T];\FL^{s,p}(\T))$.
	\end{enumerate}
	\label{th:nonexistence}
\end{theorem}
This ill-posedness result motivated establishing an alternative model for the complex-valued mKdV equation \eqref{mkdv} at low regularity. Similarly to the first gauge transform $\mathcal{G}_1$ which exploited the conservation of mass, we introduced in \cite{A} a second renormalization of the equation through the following gauge transformation depending on the momentum
\begin{equation}
\mathcal{G}_2(u)(t,x) = e^{\mp iP(u)t} u(t,x). \label{g2}
\end{equation}
Note that if $u \in C\big(\R; H^\frac12(\T)\big)$, the momentum is finite and conserved $P\big(u(t)\big) = P(u_0)$, thus the gauge transformation $\mathcal{G}_2$ is invertible and $u$ solves \eqref{mkdv} if and only if $\mathcal{G}_2 \circ \mathcal{G}_1 ( u )$ solves the second renormalized mKdV equation (mKdV2):
\begin{equation}
\partial_t u + \partial_x^3 u = \pm \bigg(|u|^2 \partial_x u - \mu(u) \partial_x u - i P(u) u \bigg). \label{renorm2}
\end{equation}
The effect of the gauge transformation $\mathcal{G}_2$ is to remove certain resonant frequency interactions in the nonlinearity which are responsible for the ill-posedness result in Theorem~\ref{th:nonexistence}. This supports our choice of gauge transformation $\mathcal{G}_2$. Hence, we propose mKdV2 \eqref{renorm2} as the correct model to study the complex-valued mKdV equation \eqref{mkdv} outside of $H^\frac12(\T)$. Indeed, in \cite{A}, we proved that, unlike mKdV1 \eqref{renorm}, mKdV2 \eqref{renorm2} is locally well-posed in $\FL^{s,p}(\T)$ for $s\geq \frac12$ and $1\leq p <4$.
Our goal in this paper is to improve this well-posedness result without exploiting the complete integrability of the equation.

\begin{theorem} \label{th:lwp}
	The mKdV2 equation \eqref{renorm2} is locally well-posed in $\FL^{s,p}(\T)$ for any $s\geq \frac12$ and $1\leq p <\infty$. Moreover, the data-to-solution map is locally Lipschitz continuous.
\end{theorem}

\begin{remark}\rm
	\noi(i) The solutions constructed in Theorem~\ref{th:lwp} satisfy the Duhamel formulation:
	\begin{equation*}
	u(t) = S(t) u_0 \pm \int_0^t S(t-t')  \mathcal{N}(u, \conj{u},u)(t') \, dt',
	\end{equation*}
	for $t\in[-T,T]$, where $S(t)$ denotes the linear propagator, $ \mathcal{N}(u, \conj{u}, u)$ corresponds to the right-hand side of \eqref{renorm2}, for some $T>0$ depending only on $\|u_0\|_{\FL^{\frac12,p}}(\T)$. See Section~\ref{sec:preliminaries} for further details.
	
	\noi (ii) Theorem~\ref{th:lwp} is sharp with respect to the method, since local uniform continuity of the data-to-solution map is known to fail in $\FL^{s,p}(\T)$ for any $s<\frac12$ and $1\leq p <\infty$ \cite{A}. 
	Without imposing uniform continuous dependence on the initial data, we expect it to be possible to lower $s$ in Theorem~\ref{th:lwp}. In a forthcoming work, we intend to pursue the question of local well-posedness of mKdV2 \eqref{renorm2} in $\FL^{s,p}(\T)$ for $s<\frac12$ by combining the method introduced by Deng-Nahmod-Yue \cite{DNY} and the energy method in \cite{NTT} and \cite{MPV}.
	
	\noi (iii) To show Theorem~\ref{th:lwp}, we apply the method introduced by Deng-Nahmod-Yue \cite{DNY}, which is based on constructing solutions $u$ with a particular structure. As a consequence, uniqueness holds conditionally in a sub-manifold of $X^{\frac12,\frac12}_{p,2-}$ (see Definition~\ref{def:xsb}) determined by the structure imposed on $u$.  In Sobolev spaces, unconditional uniqueness holds in $H^s(\T)$ for $s\geq \frac13$ (see \cite{KO, MPV}). It would be of interest to consider the problem of unconditional uniqueness of mKdV2 \eqref{renorm2} in the Fourier-Lebesgue spaces.
\end{remark}

Using the a priori bounds established by Oh-Wang \cite{OH1}, which generalize the result by Killip-Vi\c{s}an-Zhang \cite{KVZ} to the Fourier-Lebesgue setting, we extend the solutions in Theorem~\ref{th:lwp} globally-in-time.
\begin{theorem}\label{th:gwp}
	The mKdV2 equation \eqref{renorm2} is globally well-posed in $\FL^{s,p}(\T)$ for $s\geq \frac12$ and $1\leq p <\infty$.
\end{theorem}

For real-valued solutions $u$, the momentum $P(u) \equiv 0$ which implies that $\mathcal{G}_2(u) \equiv u$. Consequently, the previous results on the complex-valued mKdV2 equation \eqref{renorm} in Theorems~\ref{th:lwp} and \ref{th:gwp} also apply to the real-valued setting.
\begin{corollary}\label{cor:lwp_real}
	The real-valued mKdV1 equation \eqref{renorm} is globally well-posed in $\FL^{s, p}(\T)$ for $s\geq \frac12$ and $1\leq p <\infty$.
\end{corollary}
\begin{remark}\rm
Corollary~\ref{cor:lwp_real} (restricted to $s\geq \frac12$) extends the result by Kappeler-Molnar to the defocusing case and also to the large data setting. Furthermore, our solutions satisfy the Duhamel formulation.
\end{remark}

The real-valued mKdV equation \eqref{mkdv_real} has garnered more attention than its complex-valued counterpart in the periodic setting. For context, we review some of the known well-posedness results of the former equation. Local well-posedness in $H^s(\T)$ for $s\geq \frac12$ is due to Bourgain \cite{BO94}. The result was shown through the Fourier restriction norm method and it is sharp if one requires Lipschitz continuity of the solution map \cite{BO97,CCT03}. Colliander-Keel-Staffilani-Takaoka-Tao \cite{CKSTT03} showed that these solutions can be extended globally in time, by applying the $I$-method.
By weakening the resonant contribution in the nonlinearity, Takaoka-Tsutsumi \cite{TT04} and Takaoka-Nakanishi-Tsutsumi \cite{NTT} extended the local well-posedness to $H^s(\T)$, $s>\frac13$. The end-point result was shown by Molinet-Pilod-Vento in \cite{MPV}. These solutions are global-in-time due to the a priori bounds by Killip-Vi\c{s}an-Zhang \cite{KVZ}. Exploiting the complete integrability of the equation, Kappeler-Topalov \cite{KT} showed global existence and uniqueness of solutions of the real-valued defocusing mKdV equation (with `$+$' sign) in $H^s(\T)$ for $s\geq 0$. These solutions are the unique limit of smooth solutions and are not known to satisfy the equation in the distributional sense (see \cite{KT,KOY,OH2} for more details). In \cite{Mo}, Molinet showed that these solutions are indeed distributional solutions and proved ill-posedness of the mKdV equation below $L^2(\T)$, in the sense of failure of continuity of the solution map. See also the work of Schippa \cite{Schippa}.

This ill-posedness result, motivated the study of the mKdV equation outside of Sobolev spaces, in particular, in spaces with similar scaling. To clarify, we briefly recall the scaling heuristics. The mKdV equation \eqref{mkdv} on the real line enjoys the following scaling symmetry 
\begin{equation}
u^\ld (t,x) = \ld u(\ld^3 t , \ld x), \label{scaling}
\end{equation}
which induces the scaling critical regularity $s_{\text{crit}} = - \frac12$ in the sense that the $\dot{H}^{-\frac12}(\R)$-norm is invariant under the scaling \eqref{scaling}. In the subcritical regime, in $H^s(\T)$ for $s>-\frac12$, one might expect well-posedness of \eqref{mkdv} from the scaling heuristics. However, Molinet's result shows that there is no hope of covering this regime within the Sobolev scale. An alternative approach is to study \eqref{mkdv} in spaces with similar scaling, namely the Fourier-Lebesgue spaces. In \cite{KM}, Kappeler-Molnar showed local well-posedness of the real-valued defocusing mKdV equation in $\FL^{s,p}(\T)$, for $s \geq 0$, $1 \leq p < \infty$ (see also \cite{N}). Given that $\dot{\FL}^{s,p}(\R)$ scales like $\dot{H}^\s(\R)$ for $\s = s + \frac1p - \frac12$, this result covers the full subcritical regime, although the solutions are not yet known to be distributional solutions. 

In the euclidean setting, the well-posedness of the complex-valued mKdV equation \eqref{mkdv} has been a long standing topic of interest. Kenig-Ponce-Vega \cite{KPV93} showed local well-posedness of \eqref{mkdv} in $H^s(\R)$ for $s\geq \frac14$. This result is sharp with respect to the method, due to the failure of uniform continuity of the data-to-solution map in $H^s(\R)$, $s<\frac14$. In the Fourier-Lebesgue setting, Gr\"unrock-Vega \cite{GV09} proved local well-posedness of \eqref{mkdv} in $\FL^{s,p}(\R)$ for $s> \frac{1}{2p}$ and $1\leq p <\infty$; see also \cite{G04}. These solutions can be extended globally in time due to the a priori bounds by Oh-Wang in \cite{OH1}.
In a recent paper \cite{HGKV}, Harrop-Griffiths-Killip-Vi\c{s}an showed optimal global well-posedness of the complex-valued mKdV equation in $H^s(\R)$ for $s> -\frac12$ by exploiting the complete integrability of the equation, giving a definite answer to the problem of local well-posedness of mKdV \eqref{mkdv} on the real line, in Sobolev spaces. 

We conclude this section by stating some further remarks.
\begin{remark}\rm
	\noi (i) The ill-posedness result in Theorem~\ref{th:nonexistence} follows an argument by Guo-Oh \cite{GO}. The proof combines the local well-posedness of mKdV2 \eqref{renorm2} in Theorem~\ref{th:lwp} and the rapid oscillation of the phase in the gauge transformation $\mathcal{G}_2$ \eqref{g2}, due to the assumption on the momentum \eqref{ill-posed}. Thus, since Theorem~\ref{th:nonexistence} follows from \cite{A,GO}, we omit the details.
	
	\noi (ii) The scaling heuristics described above can be transported to the Fourier-Lebesgue setting. The critical regularity is given by $s_{\text{crit}} (p)= - \frac{1}{p}$. We can also compare the scaling of the two families of spaces to conclude that $\dot{\FL}^{s,p}(\R)$ scales like $\dot{H}^\s(\R)$ for $\s = s - \frac12 + \frac1p$. From these heuristics, we see that the results in Theorem~\ref{th:lwp} and Corollary~\ref{cor:lwp_real} are at the scale of $L^2(\T)$.

\end{remark}

\begin{remark}\rm
	In \cite{A}, we showed that it is possible to recover solutions of mKdV \eqref{mkdv} from solutions of mKdV2 \eqref{renorm2} outside of $H^\frac12(\T)$ by imposing the following notion of finite momentum at low regularity.
	\begin{definition}
		Suppose that 
		$$P(\P_{\leq N} f) \text{ converges as } N\to\infty.$$ 
		Then, we say that $f$ has finite momentum and denote the limit by $P(f)$.
		\label{def:P}
	\end{definition}
	\noi By imposing this notion of finite momentum to initial data in $\FL^{s,p}(\T)$ for $s\geq \frac12$ and $1\leq p <3$, we showed that the corresponding solutions of mKdV2 \eqref{renorm2} have finite and conserved momentum. The restriction follows from an energy estimate, which we do not know how to improve at the moment. As a consequence of conservation of momentum, we proved existence of distributional solutions of mKdV \eqref{mkdv} by using the gauge transformation $\mathcal{G}_2$ and a limiting argument. We expect a similar result to hold for the full range of well-posedness, $s\geq \frac12$ and $1\leq p<\infty$, if the momentum of the initial data is finite. Since the main focus of this paper is on improving the previous well-posedness result of mKdV2 \eqref{renorm2}, we will not discuss further how to recover solutions of the complex-valued mKdV equation \eqref{mkdv} from those of mKdV2 \eqref{renorm2}.
	
\end{remark}

\subsection{Outline of the strategy}
In our previous work on the complex-valued mKdV equation \cite{A}, we proved local well-posedness of mKdV2 \eqref{renorm2} in $\FL^{s,p}(\T)$ for $s\geq \frac12$ and $1\leq p <4$ by using the Fourier restriction norm method. The solutions are constructed through a contraction mapping argument on the Duhamel formulation
\begin{equation}
u(t) = S(t) u_0 \pm \int_0^t S(t-t')  \mathcal{N}(u, \conj{u}, u) (t') \, dt' =: S(t)u_0 \pm D  \mathcal{N}(u, \conj{u}, u)(t),\label{intro_duhamel}
\end{equation}
where $S(t)$ denotes the linear propagator, $ \mathcal{NR}(u)$ the nonlinearity of the mKdV2 equation and $D$ the Duhamel operator. In particular, we look for solutions in the $X^{s,b}$ spaces adapted to the Fourier-Lebesgue setting defined by the following norm
\begin{align}
\|u\|_{X^{s,b}_{p,q}} = \big\| \jb{n}^s \jb{\tau-n^3}^b \ft{u}(\tau,n) \big\|_{\l^p_n L^q_\tau}, \label{intro_xsb}
\end{align}
see Definition~\ref{xsb} for more details. A key ingredient in proving local well-posedness is a nonlinear estimate of the form 
\begin{equation}
\big\| D  \mathcal{NR}(u_1,u_2,u_3) \big\|_{X^{s,b}_{p,q}  } \les \prod_{j=1}^3 \|u_j\|_{X^{s,b}_{p,q}}, \label{intro_nonlinear}
\end{equation}
for $b=\frac12$ and $q=2$.
 The main difficulty resides in controlling the derivative in the nonlinearity. Since the Duhamel operator has smoothing in time but not space, we want to exploit the multilinear dispersion by using the modulations, i.e., the weights $\jb{\tau - n^3}$ in the norm \eqref{intro_xsb}. The need to use the highest modulation to help control the derivative imposes the restriction $1\leq p <4$ and we do not know how to overcome it within this framework.

In this paper, we apply the method introduced by Deng-Nahmod-Yue \cite{DNY} to the mKdV2 equation \eqref{renorm2} extending our previous local well-posedness result to $\FL^{s,p}(\T)$ for $s\geq \frac12$ and $4\leq p <\infty$. Instead of running a contraction mapping argument on the Duhamel formulation \eqref{intro_duhamel}, we want to establish a system of equations which imposes structure on the solution $u$ but guarantees that it still satisfies the Duhamel formulation \eqref{intro_duhamel}. The structure is appropriately chosen to avoid the nonlinear estimates \eqref{intro_nonlinear} that we could not prove in our previous work. In particular, we want to avoid using the modulations to gain smoothing in space until strictly necessary.

The method was introduced to improve the local well-posedness results of Gr\"unrock-Herr \cite{GH} for the derivative nonlinear Schr\"odinger equation (DNLS). Here, they showed DNLS is locally well-posed in $\FL^{s,p}(\T)$ for $s\geq \frac12$ and $1\leq p <4$, using the Fourier restriction norm method on the gauged equation introduced by Herr \cite{Herr}. Moreover, they proved that the main nonlinear estimate fails for $p\geq 4$, preventing the improvement of the well-posedness result within this framework. The new method by Deng-Nahmod-Yue is designed to bypass the problematic cases responsible for the failure of the nonlinear estimate. To that end, motivated by the probabilistic setting, one looks for solutions which are centered around a suitable object.

The idea of looking for solutions with a particular structure is not new in the context of probabilistic PDEs (with random initial data or stochastic forcing). Centering the solution around a suitably chosen object was seen in the works of Bourgain \cite{BO96}, Da Prato-Debussche \cite{DPD02} and Gubinelli-Imkeller-Perkowski \cite{GIP}, for example. In these works, the randomness of the centers introduces smoothing in space on the remainder nonlinear pieces, making it possible to improve purely deterministic results.
The lack of randomness in our setting leads us to considering a moving `center' $w$ and the solution can be thought of as parameterized by $w$. The equation for $w$ is chosen so that $u$ with the imposed structure solves the Duhamel formulation \eqref{intro_duhamel}. Unlike the previously discussed works, $w$ is not smoother in space but in \emph{time}.

Choosing the correct structure for $u$ is the main difficulty in the method by Deng-Nahmod-Yue and it is strongly motivated by the difficulties in the nonlinear estimate \eqref{intro_nonlinear}. We want to find an appropriate functional $F(u,w)$ such that we can solve the following system
\begin{equation}\label{intro_system}
\begin{cases}
u = w + F(u,w), \\
w = S(t)u_0 \pm D \mathcal{N}(u) - F(u,w).
\end{cases}
\end{equation}
There are three main points in establishing and solving the system \eqref{intro_system}: (i) choosing the frequency regions of the nonlinear terms in $F(u,w)$; (ii) modifying the Duhamel operator to induce smoothing in space; (iii) using second iteration to solve the equation for $w$.

We start by choosing the spaces for $u$ and $w$. In the following discussion, we will focus on the endpoint $s=\frac12$. We want to construct solutions $u\in X^{\frac12,\frac12}_{p,2-} \embeds C(\R; \FL^{\frac12,p}(\T))$. In order to choose the space for $w$, we look at the nonlinear estimate \eqref{intro_nonlinear}. In certain regions of frequency space, the estimate \eqref{intro_nonlinear} holds for any fixed $2\leq p<\infty$ and some $b=1-$ and $q=\infty-$ (see Remark~\ref{rm:w} for more details). This motivates the choice of $X^{\frac12, 1-}_{p,\infty-} \subset X^{\frac12, \frac12}_{p,2-}$ as the space where $w$ lives.

Now, we focus on (i). Our first guess for $F(u,w)$ corresponds to the cubic terms in the Duhamel formulation for which we cannot show the trilinear estimate \eqref{intro_nonlinear} in any $X^{\frac12,b}_{p,q}\subset C(\R; \FL^{\frac12,p}(\T))$ space, regardless of the choice of $b$ and $q$. These contributions should appear in the equation for $u$, in order to be estimated in a weaker norm $X^{\frac12, \frac12}_{p,2-}$. In particular, the contributions essentially look like the following
\begin{equation}\label{paracontrolled}
\begin{aligned}
&D \Big( \sum_N \P_{ N} \partial_x w \cdot  \P_{\ll N} \conj{u} \cdot \P_{\ll N} u \Big), \\
&D \Big(\sum_N \P_{ N} \partial_x w \cdot \P_N \conj{w} \cdot \P_{\ll N} u \Big), \\
&D \Big( \sum_N \P_{N} \partial_x w \cdot \P_N w \cdot  \P_{\ll N} \conj{u} \Big),
\end{aligned}
\end{equation}
where $\P_N$ and $\P_{\ll N}$ denote the Dirichlet projections onto the spatial frequencies $\{|n|\sim N\}$ and $\{|n| \ll N\}$, respectively. These terms can roughly be seen as `paracontrolled' by $w$ (see \cite{GIP} for details on paracontrolled distributions). 

Unfortunately, this first ansatz for $F(u,w)$ is not sufficient. When estimating the above contributions in $X^{\frac12,\frac12}_{p,2-}$, the lack of smoothing in space of the Duhamel operator $D$ forces us to use the modulations. This leads us to (ii) and to the introduction of a modified Duhamel operator which not only has smoothing in time but also in space. The modification is introduced through a convolution in time with a smooth time cutoff parameterized by the resonance relation (see Section~\ref{sec:kernel} for more details). Choosing $F(u,w)$ as the contributions in \eqref{paracontrolled} with $D$ replaced by the modified Duhamel operator establishes an equation for $u$ which can be solved through a fixed point argument, given a fixed $w\in X^{\frac12,1-}_{p,\infty-}$.

From the first equation, we obtain a function $u$ parameterized by $w$ which is not yet a solution to the mKdV2 equation. This will only follow after we have found the correct center $w\in X^{\frac12, 1-}_{p,\infty-}$. The crucial idea behind obtaining $w$ is to modify the second equation in \eqref{intro_system} by substituting $u=u[w]$ by the first equation when relevant, using a \emph{partial} second iteration (iii). This strategy resembles the second iteration method used by Bourgain \cite{BO97}, Oh \cite{OH09} and Richards \cite{Richards16}, for example. In these works, by iterating the Duhamel formulation, it becomes possible to estimate problematic regions in frequency space at the cost of increasing the multilinearity of the analysis. In our case, we have a partial second iteration, as we are exploiting the structure imposed to $u$, not the full Duhamel formulation. This will lead to new cubic and quintic terms for which we can establish estimates (see Sections~\ref{sec:equations} and \ref{sec:w}).

Recall that our underlying motivation to establish the system of equations \eqref{intro_system} is to avoid using the modulations until strictly necessary. On the one hand, the nonlinear terms in the equation for $u$ have smoothing in space induced by the modified Duhamel operator, which is sufficient to control the derivative. On the other hand, in order to solve the equation for $w$, we can use second iteration which introduces smoother in time terms or increases the multilinearity of the contributions. In both cases, the modification allows us to show the estimates for any $2\leq p <\infty$, even when the modulations need to be used.

\subsection{Outline of the paper} In Section \ref{sec:preliminaries} we introduce the relevant function spaces and some auxiliary results. In Section \ref{sec:kernel} we introduce the modified Duhamel operator and show relevant kernel estimates. In Section~\ref{sec:equations} we establish equation for $u$. Moreover, we explain how to use second iteration to obtain the equation for $w$, which is included in detail in Appendix~\ref{app:w}. The relevant estimates for the nonlinear terms in the equation for $u$ are shown in Section~\ref{sec:u}. Lastly, Section \ref{sec:w} is dedicated to solving the equation for $w$.

\section{Preliminaries}\label{sec:preliminaries}
\subsection{Notations and function spaces}
We start by introducing some useful notations. We will use $\varphi: \R \to \R$ to denote a smooth time cutoff function, equal to 1 on $[-1,1]$ and 0 outside of $[-2,2]$ and $\varphi_T (t) := \varphi(T^{-1}t)$, for $0\leq T \leq 1$. 
Lastly, let  $A\les B$ denote an estimate of the form $A\leq CB$ for some constant $C>0$. Similarly, $A\sim B$ will denote $A\les B$ and $B\les A$, while $A\ll B$ will denote $A\leq \eps B$, for some small constant $ 0<\eps\ll 1$.
The notations $a+$ and $a-$ represent $a+\eps$ and $a-\eps$ for arbitrarily small $\eps>0$, respectively.
 
Our conventions for the Fourier transform are as follows. The Fourier transform of $u: \R\times\T \to \C$ with respect to the space variable is given by
$$\Ft_x u(t,n) = \frac{1}{2\pi} \int_\T u(t,x) e^{- inx} \ dx.$$
The Fourier transform of $u$ with respect to the time variable is given by
$$\Ft_t u(\tau,x) = \frac{1}{2\pi}\int_\R u(t,x) e^{-it\tau} \ dt.$$
The space-time Fourier transform is denoted by $\Ft = \Ft_t \Ft_x$. We will also use $\ft{u}$ to denote $\Ft_x u$, $\Ft_t u$ and $\Ft_{t,x} u$, but it will become clear which one it refers to from context, namely from the use of the spatial and time Fourier variables $n$ and $\tau$, respectively.

Now, we focus on the relevant spaces of functions. Let $\S (\R\times\T)$  denote the space of functions $u:\R\times\R\to\C$, with $u\in C^\infty(\R\times\T)$ which satisfy
\begin{align*}
u(t,x+1) = u(t,x), \quad  \sup_{(t,x) \in \R\times\T} | t^\al \partial_t^\be \partial_x^\gamma u(t,x)| < \infty, \quad \al,\be,\gamma\in\mathbb{Z}_{\geq 0}.
\end{align*}
In \cite{BO2}, Bourgain introduced the $X^{s,b}$-spaces defined by the norm
\begin{align}
\|u\|_{X^{s,b}} = \big\| \jb{n}^s \jb{\tau - n^3}^b \ft{u}(\tau,n) \big\|_{\l^2_n L^2_\tau}.
\label{xsb}
\end{align}
In the following, we define the $X^{s,b}$-spaces adapted to the Fourier-Lebesgue setting (see Gr\"{u}nrock-Herr \cite{GH}).

\begin{definition}\label{def:xsb}
	Let $s,b\in\R$, $1\leq p,q\leq \infty$. The space $X^{s,b}_{p,q}(\R\times\T)$, abbreviated $X^{s,b}_{p,q}$, is defined as the completion of $\S(\R\times\T)$ with respect to the norm
	\begin{align*}
	\norm{u}_{X^{s,b}_{p,q}} = \big\| \jb{n}^s \jb{\tau - n^3}^b \ft{u}(\tau, n) \big\|_{\l^p_n L^q_\tau}.
	\end{align*}
	When $p=q=2$, the $X^{s,b}_{p,q}$-spaces defined above reduce to the standard $X^{s,b}$-spaces defined in \eqref{xsb}.
\end{definition}

Before proceeding, we introduce the relevant spaces of functions and associated parameters. Let $0< \dl \ll 1$ a small parameter to be chosen later, depending on $2< p <\infty$. We introduce the following parameters
\begin{align*}
b_0 &= 1 - 2 \dl, & b_1 &= 1 - \dl, \\
q_0 &= \frac{1}{4\dl}, & q_1 &= \frac{1}{4.5\dl}, \\
\frac{1}{r_0} &= \frac12 + \dl, & \frac{1}{r_1} &= \frac12 + 2\dl ,  & \frac{1}{r_2} &= \frac12 + 3\dl.
\end{align*}
Note that $b_0 < b_1$, $q_1 < q_0$ and $r_2 < r_1 < r_0$. 

We will focus on showing the result for the endpoint $s=\frac12$, see Remark~\ref{rm:persistence} for more details. Consequently, we will conduct our analysis in the following $X^{s,b}_{p,q}$ spaces:
\begin{align*}
Y_0 &= X^{\frac12, \frac12}_{p, r_0} (\R\times \T),  \\
Y_1 &= X^{\frac12, \frac12}_{p, r_1} (\R\times\T), \\
Z_0 & = X^{\frac12, b_0}_{p, q_0} (\R\times\T),  \\
Z_1 & = X^{\frac12, b_1}_{p,q_0} (\R\times\T).
\end{align*}
Recall the following embedding: for any $1\leq p <\infty$, 
\begin{align*}
X^{s,b}_{p,q} (\R\times \T) &\embeds C (\R; \FL^{s,p}(\T)) \quad \text{for } b>\frac{1}{q'} = 1 - \frac1q.
\end{align*}
In particular, it holds that $Z_0 \subset Y_0 \subset C(\R; \FL^{\frac12,p}(\T))$.

\subsection{Nonlinearity and notion of solution}
The nonlinearity of \eqref{renorm2} has the following spatial Fourier transform
\begin{align}
 \sum_{\substack{n=n_1+n_2+n_3,\\ \Phi(\nbar_{123}) \neq 0}} in_1 \ft{u}(n_1) \ft{\conj{u}}(n_2) \ft{u} (n_3) - in |\ft{u}(n)|^2 \ft{u}(n), \label{nonlinearity1}
\end{align}
where $\nbar_{123} = (n_1,n_2,n_3)$ and $\Phi$ denotes the resonance relation 
$$\Phi(\nbar_{123}) = n^3 - n_1^3 - n_2^3 - n_3^3 = 3(n_1+n_2)(n_1+n_3)(n_2+n_3),$$
where the factorization holds if $n=n_1+n_2+n_3$.
Consider the following trilinear operators
\begin{align}
\Ft_x\big( \mathcal{NR}_{\geq}(u_1,u_2,u_3) \big)(n) &= \sum_{\substack{n=n_1+n_2+n_3,\\ \Phi(\nbar_{123}) \neq 0, \\ |n_2| \geq |n_3|}}  in_1 \ft{u}_1(n_1) \ft{u}_2 (n_2) \ft{u}_3(n_3), \label{nonresonant1}\\
\Ft_x\big( \mathcal{NR}_{>}(u_1,u_2,u_3) \big)(n) &= \sum_{\substack{n=n_1+n_2+n_3,\\ \Phi(\nbar_{123}) \neq 0, \\ |n_2| > |n_3|}}  in_1 \ft{u}_1(n_1) \ft{u}_2 (n_2) \ft{u}_3(n_3), \label{nonresonant2}\\
\Ft_x \big( \mathcal{R}(u_1,u_2,u_3) \big) (n) &= - in \ft{u}_1 (n) \conj{\ft{u}}_2 ( n) \ft{u}_3 (n). \nonumber
\end{align}
Consequently, we can decompose the nonlinearity \eqref{nonlinearity1} into non-resonant and resonant contributions
$$ \mathcal{N}(u, \conj{u}, u) = \mathcal{NR}_{\geq}(u, \conj{u}, u) +  \mathcal{NR}_>(u,u,\conj{u})+ \mathcal{R}(u,u,u).$$
Note that if $n_j$ is the spatial frequency corresponding to $u_j$, $j=1,2,3$ in \eqref{nonresonant1} and \eqref{nonresonant2}, then $|n_2| \geq |n_3|$. Consequently, we want to consider the following subregions:
\begin{align*}
\mathbb{X}_{A}(n) & = \big\{ (n_1,n_2,n_3) \in \Z^3:  \ \Phi(\nbar_{123}) \neq 0, \ |n_2| \ll |n_1| \big\}, \\
\mathbb{X}_{B}(n) & = \big\{ (n_1,n_2,n_3)\in \Z^3: \ \Phi(\nbar_{123}) \neq 0, \ |n_3| \ll \min(|n|, |n_1|) \leq \max(|n|, |n_1|) \sim |n_2| \big\}, \\
\mathbb{X}_{C}(n) & =  \big\{ (n_1,n_2,n_3) \in \Z^3: \ \Phi(\nbar_{123}) \neq 0,  \  |n| \les |n_3| \ll |n_1| \big\}, \\
\mathbb{X}_{D}(n) & = \big\{ (n_1,n_2,n_3) \in \Z^3: \ \Phi(\nbar_{123}) \neq 0,  \ |n_1| \les |n_3| \big\}.
\end{align*}
For $*\in\{A, B, C, D\}$, let $ \mathcal{NR}_{*, \geq},  \mathcal{NR}_{*, >}$ denote the operators defined in \eqref{nonresonant1} and \eqref{nonresonant2}, respectively, with sums localized to $\X_*(n)$.
We can write the non-resonant contributions of the nonlinearity as $$ \mathcal{NR}_{\geq} =  \mathcal{NR}_{A, \geq} +  \mathcal{NR}_{B, \geq} +  \mathcal{NR}_{C, \geq} +  \mathcal{NR}_{D, \geq}$$ and equivalently for $ \mathcal{NR}_{>}$.
We also introduce the following notation, which will be useful in Sections~\ref{sec:u} and \ref{sec:w}:
\begin{equation}
\X_*^\mu(n) := \big\{ \nbar_{123} \in \X_*(n): \ \Phi(\nbar_{123} ) = \mu \big\}. \label{localizedX}
\end{equation}

The following lemma clarifies the relation between the frequencies in the subregions introduced.
\begin{lemma}
	If $|n_2| \geq |n_3|$, we have the following properties for the sets $\X_*(n)$, where $*\in\{A, B, C, D\}$:
	
	{\rm(1)} $(n_1, n_2, n_3)\in\X_A(n) \implies |n_3| \leq |n_2| \ll |n_1|\sim |n| $;
	
	{\rm(2)} $(n_1, n_2, n_3)\in\X_B(n) \implies |n_3| \ll |n| \les |n_1| \sim |n_2| $ or $|n_3| \ll |n_1| \ll |n| \sim |n_2|$;
		
	{\rm(3)} $(n_1, n_2, n_3)\in\X_C(n) \implies  |n| \les |n_3| \ll |n_2| \sim |n_1| $;
	
	{\rm(4)} $(n_1, n_2, n_3)\in\X_D(n) \implies|n_1| \les |n_3| \leq |n_2|  $;  
	
	{\rm(5)} $(n_1, n_2, n_3)\in\X_A(n) \cup \X_B(n) \cup \X_C(n) \implies |\Phi(\nbar_{123})| \ges \max(|n_1|, |n_2|)^2$;
	
	{\rm(6)} $(n_1, n_2, n_3) \in \X_D(n) \implies |\Phi(\nbar_{123})| \ges |n_2|$ and $\jb{n}^\frac12 |n_1| \les (\jb{n_1}\jb{n_2}\jb{n_3})^\frac12$.

\end{lemma}
\begin{remark}\rm
Note that the nearly-resonant case when $|n_1|\sim|n_2|\sim|n_3|$ and $|\Phi(\nbar_{123})|\ges \max(|n_1|, |n_2|, |n_3|)$ is included in $\X_D(n)$. There are other frequency interactions allowed in this region which are fully non-resonant, i.e., $|\Phi(\nbar_{123})| \ges \max(|n_1| , |n_2|, |n_3|)^2$ holds. However, due to the condition in $(6)$, the resonance relation will not play a crucial role when estimating this contribution.
\end{remark}

Lastly, we recall our notion of solution to \eqref{renorm2}. 
Define the linear propagator $S(t)$ as follows
\begin{align*}
\Ft_x \big( S(t) F(t,x) \big) (n) = e^{itn^3} \ft{F}(t,x).
\end{align*}
Also, define the Duhamel operator, $D$, and its truncated version, $\D$:
\begin{align*}
D F(t,x) & = \int_0^t S(t-t') F(t',x) \ dt' ,\\
\D F(t,x) & = \varphi(t) \cdot D \big( \varphi(t') \cdot F(t',x) \big) (t) = \varphi(t) \int_0^t S(t-t') \varphi (t') F(t',x) \ dt'.
\end{align*}
We say that $u\in C(\R; \FL^{\frac12,p}(\T))$ is a solution of \eqref{renorm2} with initial data $u_0\in\FL^{\frac12, p}(\T)$ if it satisfies the following integral equation, the Duhamel formulation\footnote{In the following, we choose to take the `$+$' sign in front of the nonlinearity in \eqref{renorm2} for simplicity, as the sign will not play a role in the analysis.}
\begin{align}
u(t) = S(t) u_0 + D  \mathcal{NR}(u, \conj{u}, u) (t) + D \mathcal{R}(u,u,u)(t). \label{duhamel}
\end{align}
Since we are only concerned with local well-posedness, let $0<  T \leq 1$ and consider the truncated Duhamel formulation:
\begin{align}
u(t) = \varphi \cdot S(t) u_0 + \varphi_T \cdot \D  \mathcal{NR} (u, \conj{u}, u) (t) + \D \mathcal{R}(u,u,u)(t). \label{duhamel_truncated}
\end{align}
Since solving \eqref{duhamel_truncated} for $t\in\R$ is equivalent to solving \eqref{duhamel} for $t\in[-T,T]$, we will focus on the truncated equation.

\subsection{Auxiliary results}\
The following proposition allows us to gain a small power of the time of existence $T$, needed to close the contraction mapping argument.
\begin{lemma}\label{lm:T}
	Suppose that $F$ is a smooth function such that $F(0)=0$. Then, we have the following estimates
	\begin{align*}
	\| \varphi_T\cdot F \|_{Y_0} &\les T^\theta \|F\|_{Y_1}, \\
	\|\varphi_T \cdot F \|_{Z_0} & \les T^\theta \|F\|_{Z_1}, 
	\end{align*}
	for any $0<\theta\leq \frac{\dl}{2}$ and $0<T\leq 1$.
\end{lemma}
\begin{proof}
We want to estimate the following quantity
\begin{align*}
\| \varphi_T \cdot F \|_{X^{s,b}_{p,q}} 
& = \bigg\| \jb{n}^s \jb{\tau}^b \big( \ft{\varphi}_T \ast_\tau \ft{F}(\cdot+n^3,n) \big)(\tau)  \bigg\|_{\l^p_n L^{q}_\tau}.
\end{align*}
Both estimates follow once we show
\begin{align*}
\|\jb{\tau}^b \ft{\varphi}_T \ast f (\tau) \|_{L^q_\tau} \les T^{\frac{1}{\tilde{q}} - \frac1q} \| \jb{\tau}^b f (\tau) \|_{L^{\tilde{q}}_\tau}, 
\end{align*}
for $f$ satisfying $\int_\R f(\tau) \, d\tau = 0$, $1<\tilde{q}<q<\infty$ and $b<1<b+\frac{1}{\tilde{q}}$. 
The estimate above follows once we prove the two following estimates:
\begin{align}
\big\| \jb{\tau}^b \big(  \ft{\varphi}_T \ast(\1_{|\tau| \geq T^{-1} } f) \big) (\tau) \big\|_{L^q_\tau} & \les T^{\frac{1}{\tilde{q}} - \frac1q} \| \jb{\tau}^b f(\tau) \|_{L^{\tilde{q}}_\tau}, \nonumber\\
\big\| \jb{\tau}^b \big( \ft{\varphi}_T \ast (\1_{|\tau| < T^{-1} }  f ) \big)(\tau) \big\|_{L^q_\tau} & \les T^{\frac{1}{\tilde{q}} - \frac1q} \| \jb{\tau}^b f(\tau) \|_{L^{\tilde{q}}_\tau}. \label{linear_show2}
\end{align}
Note that the first inequality is equivalent to the following result
\begin{align}
\bigg\| \intt_\R \frac{\jb{\tau}^b}{\jb{\ld}^b} \1_{|\ld| \geq T^{-1}} \ft{\varphi}_T(\tau - \ld) f (\ld) \, d\ld \bigg\|_{L^q_{\tau}} \les T^{\frac{1}{\tilde{q}} - \frac1q} \|f\|_{L^{\tilde{q}}_\tau}. \label{linear_show1_aux}
\end{align}
Using Young's inequality with $1 + \frac1q = \frac{1}{\tilde{q}} + \frac1r$ gives
\begin{align*}
\text{LHS of }\eqref{linear_show1_aux} & \les \bigg\| \intt_\R \jb{T(\tau-\ld)}^b T \ft{\varphi}(T(\tau-\ld)) f (\ld) \ d \ld \bigg\|_{L^q_\tau} 
 \les T \|  \jb{T\tau}^b \ft{\varphi}(T\tau) \|_{L^r_{\tau}} \| f\|_{L^{\tilde{q}}_\tau}.
\end{align*}
The estimate follows from $T \| \jb{T\tau}^b \ft{\varphi}(T\tau) \|_{L^r_\tau} \les_\varphi T^{1-\frac1r} = T^{\frac{1}{\tilde{q}} - \frac1q}$.

To prove \eqref{linear_show2}, using the fact that $\int_\R f(\tau) \, d\tau = 0$, we note that
\begin{align}
\ft{\varphi}_T \ast \big( \1_{|\tau| < T^{-1}} f \big)  (\tau) 
=\!\!\!\!\! \intt_{|\ld| < T^{-1}} \!\!\!\!\! f(\ld) T \big[ \ft{\varphi}(T(\tau - \ld)) - \ft{\varphi}(T\tau) \big] \, d\ld -\!\!\!\! \intt_{|\ld|\geq T^{-1}} \!\!\!\!\! f(\ld) T\ft{\varphi}(T\tau) \, d\ld. \label{linear_show2_aux}
\end{align}
For the first contribution in \eqref{linear_show2_aux}, we distinguish between the regions $\{\tau: \, |\tau| \les T^{-1}\}$ and $\{\tau: \, |\tau| \gg T^{-1}\}$, and use mean value theorem to obtain
\begin{align*}
\intt_{|\ld| < T^{-1}} T |f(\ld)| |\ft{\varphi} (T(\tau-\ld)) - \ft{\varphi}(T\tau)| \, d\ld \les_{\varphi} \frac{T}{\jb{T\tau}^\al} \intt_\R |T\ld| \, |f(\ld)| \, d\ld,
\end{align*}
for any $\al>0$. 
For the second contribution in \eqref{linear_show2_aux}, we have
\begin{align*}
\intt_{|\ld| \geq T^{-1}} |f(\ld)| T|\ft{\varphi} (T\tau)| d\ld & \les_{\varphi} \frac{T}{\jb{T\tau}^\al} \intt_\R  |f(\ld)| \ d\ld.
\end{align*}
Combining the estimates for the two contributions in \eqref{linear_show2_aux}, we obtain
\begin{align*}
\bigg\| \jb{\tau}^b \ft{\varphi}_T \ast \big( \1_{|\tau| < T^{-1}} f \big)  (\tau) \bigg\|_{L^q_\tau}
& \les \bigg\| \frac{T\jb{\tau}^b}{\jb{T\tau}^\al} \bigg\|_{L^q_\tau} \| \min(1, |T\ld|) \jb{\ld}^{-b} \|_{L^r_\ld} \|  \jb{\tau}^b f(\tau) \|_{L^{\tilde{q}}_\tau}, 
\end{align*}
by using H\"older's inequality for the last step with $1 = \frac{1}{r} + \frac{1}{\tilde{q}}$. 
Thus, we have 
\begin{align*}
\bigg\| \frac{T\jb{\tau}^b}{\jb{T\tau}^\al} \bigg\|_{L^q_\tau}  
& \les T^{1-b-\frac1q}, \\
\| \min(1, |T\ld|) \jb{\ld}^{-b} \|_{L^r_\ld} & \les T^{b-\frac1r},
\end{align*}
given that $b>1 - \frac{1}{\tilde{q}}$ and $b<1$. Combining the two bounds, we obtain the intended power of $T$.

\end{proof}

\begin{remark}\rm
	Lemma~\ref{lm:T} will only be applied to functions of the form $F(t) = \int_0^t G(t') \, dt'$ which satisfy the assumption $F(0)=0$, namely the Duhamel operator $\D$ and the operators $\G, \B$ defined in Section~\ref{sec:kernel}.
\end{remark}

The following lemma estimates the number of divisors of a given natural number. The result was adapted from \cite{DNY}.

\begin{lemma}\label{lm:divisor}
	Fix $0<\eps<1$, $\rho\geq 1$ and let $k,q\in\Z$ such that $|q| \ges |k|^\eps >0$. Then, the number of divisors $r\in\Z$ of $k$ that satisfy $|r-q| \les \rho$ is at most $\mathcal{O}_\eps (\rho^\eps)$.
	
\end{lemma}

Lastly, recall the following well-known tool (see \cite[Lemma 4.2]{GTV}).
\begin{lemma}\label{lm:convolution}
	Let $0\leq \al\leq \be$ such that $\al + \be>1$ and $\eps>0$. Then, we have
	\begin{align*}
	\int_\R \frac{1}{\jb{x-a}^{\al} \jb{x-b}^{\be} } dx \les \frac{1}{ \jb{a-b}^{\gamma} },
	\end{align*}
	where 
	\begin{align*}
	\quad \gamma =
	\begin{cases}
	\al + \be -1 , & \be<1, \\
	\al - \eps, & \be =1, \\
	\al, & \be>1.
	\end{cases}
	\end{align*}
\end{lemma}

\section{Splitting the Duhamel operator}\label{sec:kernel}
In this section, we make the smoothing-in-time of the Duhamel operator explicit by estimating its kernel. The lack of smoothing in space motivates us to introduce a modified version of the Duhamel operator localized to $\X_A(n), \X_B(n)$, in order to control the derivative from the nonlinearity. Moreover, we show kernel estimates for this modified Duhamel operator as well as the remainder operator.

\begin{proposition}\label{prop:KI} \
	The truncated Duhamel operator $\D$ has the following space-time Fourier transform
	\begin{align*}
	\Ft_{t,x} \big( \D F \big)(\tau,n) = \intt_\R K(\tau - n^3, \ld - n^3) \ft{F}(\ld, n) \, d\ld
	\end{align*}
	where the kernel $K$ is given by the following expression
	\begin{align}
	K(\tau , \ld) = -i \intt_\R \ft{\varphi}(\mu - \ld) \frac{\ft{\varphi}(\tau - \mu) - \ft{\varphi}(\tau)}{\mu} \, d\mu \label{K}
	\end{align}
	and satisfies the following estimates
	\begin{align}
	| K (\tau, \ld) | & \les_{\al, \varphi} \bigg( \frac{1}{\jb{\tau - \ld}^\al} + \frac{1}{\jb{\tau}^\al} \bigg)\frac{1}{\jb{\ld} } \nonumber\\
	& \les \frac{1}{\jb{\tau} \jb{\tau - \ld}}  \label{estimateK_new}
	\end{align}
	where $\al$ is a large enough positive number.
\end{proposition}

\begin{proof}
	We start by calculating the space-time Fourier transform of $\D F$,
	\begin{align*}
	\Ft_{t,x} \big( \D F \big)(\tau,n) & = \intt_\R \ft{\varphi}(\tau - \mu) \Ft_t \bigg( \int_0^t e^{i(t-t')n^3} \varphi(t') \ft{F}(t',n) \ dt' \bigg)(\mu) \, d\mu.
	\end{align*}
	Using the fact that $\int_0^t f(t') \ dt' = \tfrac12 \intt_\R f(t') \big( \sgn(t-t') + \sgn(t') \big) \, dt'$, we have
	\begin{multline*}
	\Ft_t \bigg( \int_0^t e^{i(t-t')n^3} \varphi(t') \ft{F}(t',n) \ dt' \bigg)(\mu) 
	 = \frac{1}{4\pi} \intt_\R e^{-it'\mu} \varphi(t') \ft{F}(t',n) \, dt' \intt_\R e^{-it(\mu-n^3)} \sgn(t)  dt \\
	+ \frac{1}{4\pi} \intt_\R e^{-it'n^3} \varphi(t') \ft{F}(t',n) \sgn(t') \,dt' \intt_\R e^{-it(\mu-n^3)} \, dt.
	\end{multline*}
	Consequently, 
	\begin{multline*}
	\Ft_t \bigg( \int_0^t e^{i(t-t')n^3} \varphi(t') \ft{F}(t',n) \ dt' \bigg)(\mu)   =  \frac{-i}{\mu-n^3} \intt_\R \ft{\varphi}(\mu-\ld) \ft{F}(\ld, n) \, d\ld \\ -i\dl(\mu-n^3) \intt_\R \intt_\R  \ft{\varphi}(n^3 - \ld - \mu') \frac{1}{\mu'} \ft{F}(\ld, n) \, d\mu' \, d\ld.
	\end{multline*}
	Calculating the convolution with $\ft{\varphi}$, we get the intended expression.

	It remains to show the estimate on the kernel. Consider the 2 regions of integration $\{\mu: \ |\mu| > 1\}$ and $\{\mu: \ |\mu| \leq 1\}$. Then, for the first region, using Cauchy-Schwarz inequality and Lemma~\ref{lm:convolution}, we have

\begin{align*}
\intt_{|\mu|>1} \frac{|\ft{\varphi}(\tau - \mu) \ft{\varphi}(\mu - \ld) |}{|\mu|} \, d\mu 
&\les_\varphi \bigg(\intt_\R \frac{d\mu}{\jb{\tau-\mu}^{2\al} \jb{\mu-\ld}^{1+2\al}} \bigg)^\frac12 \bigg( \intt_R \frac{d\mu}{\jb{\mu}^2 \jb{\mu-\ld}^2} \bigg)^\frac12 \\
&\les \frac{1}{\jb{\ld} \jb{\tau - \ld}^\al},\\
 \intt_{|\mu|>1} \frac{|\ft{\varphi}(\tau) \ft{\varphi}(\mu - \ld)|}{ |\mu| } d\mu &\les_{\varphi} \frac{1}{\jb{\tau}^\al}\intt_\R \frac{1}{\jb{\mu-\ld}^\al\jb{\mu}} \, d\mu 
\les \frac{1}{\jb{\ld} \jb{\tau}^\al },
\end{align*}
for any $\al>0$.
	For the second region, we distinguish two cases for $\tau$: $|\tau|\ges 1$ or $|\tau| \gg 1$. In both cases, using mean value theorem, we obtain
	\begin{align*}
	\intt_{|\mu|\leq 1} |\ft{\varphi}(\mu - \ld)| \frac{|\ft{\varphi}(\tau - \mu) - \ft{\varphi}(\tau)|}{|\mu|} d\mu 
	& \les_{\varphi} \frac{1}{\jb{\ld} \jb{\tau}^\al }, 
	\end{align*}
	for any $\al>0$.

	Combining the estimates, we get
	\begin{align*}
	|K(\tau, \ld)| \les_{\varphi} \bigg( \frac{1}{\jb{\tau - \ld}^\al} + \frac{1}{\jb{\tau}^\al} \bigg)\frac{1}{\jb{\ld} }.
	\end{align*}
	To show \eqref{estimateK_new}, note that $\jb{\tau} \les \jb{\tau-\ld} \jb{\ld}$ and $\jb{\tau-\ld} \les \jb{\tau} \jb{\ld}$. Thus,
	\begin{align*}
	\bigg( \frac{1}{\jb{\tau - \ld}^\al} + \frac{1}{\jb{\tau}^\al} \bigg)\frac{1}{\jb{\ld} }  \les \frac{1}{\jb{\tau} \jb{\tau-\ld}^{\al-1}} + \frac{1}{\jb{\tau-\ld} \jb{\tau}^{\al-1}} 
	\les \frac{1}{\jb{\tau} \jb{\tau-\ld}},
	\end{align*}
	for $\al\geq 2$.
\end{proof}

We want to split each of the non-resonant contributions $\D  \mathcal{NR}_{*, \geq}, \D  \mathcal{NR}_{*,>}$ for $*\in\{A,B\}$ into two components: $\G_{*\geq}, \B_{*,\geq}$ and $\G_{*, >}, \B_{*,>}$, respectively. The $\G$ contributions will have sufficient smoothing in space to allow us to control the derivative from the nonlinearity. This follows by introducing a convolution with a smooth function $\eta$ parameterized by the resonance relation $\Phi(\nbar_{123})$.

Consider a Schwartz function $\eta$ satisfying
\begin{align}
\ft{\eta}(-1) = 0, \quad \H \ft{\eta}(-1) =-1, \label{eta}
\end{align}
where $\mathcal{H}$ denotes the Hilbert transform, i.e., principal value convolution with the function~$\frac{1}{\tau}$. We define the operators $\G_{*, \geq}, \B_{*,\geq}$ through their spatial Fourier transform as follows
\begin{align*}
& \Ft_x\big( \G_{*, \geq}  (u_1,u_2,u_3) \big)(t,n) \\
& \phantom{XXX}= \varphi(t) \sum_{\substack{n=n_1+n_2+n_3,\\\nbar_{123}\in\X_*(n),\\|n_2| \geq |n_3|}}i n_1 \int_0^t e^{i(t-t')n^3}  \eta\big(\Phi(\nbar_{123})(t-t')\big) \varphi(t') \prod_{j=1}^3 \ft{u}_j(t',n_j) \, dt',\\
& \Ft_x\big( \B_{*, \geq}  (u_1,u_2,u_3)\big)(t,n) \\
& \phantom{XXX}= \varphi(t) \sum_{\substack{n=n_1+n_2+n_3,\\\nbar_{123}\in\X_*(n), \\ |n_2| \geq |n_3|}} i n_1 \int_0^t e^{i(t-t')n^3} \big[1 - \eta\big(\Phi(\nbar_{123})(t-t')\big)\big]  \varphi(t')\prod_{j=1}^3 \ft{u}_j(t',n_j) \, dt',
\end{align*}
with equivalent definitions for $\G_{*,>}, \B_{*,>}$ where we impose the condition $|n_2| > |n_3|$ to the sum. 

In order to control the contributions that depend on $\G_*$, for $*\in\{A,B\}$, we want to calculate and estimate the kernel of this convolution operator.
\begin{proposition}\label{prop:KY}
	Let $*\in\{A,B\}$. Then, the convolution operators $\G_{*, \geq}, \G_{*,>}$ have the following space-time Fourier transform
	\begin{multline*}
	\Ft_{t,x} \big( \G_{*,\underset{(>)}{\geq}} (u_1,u_2,u_3) \big)(\tau,n)  \\ =  \sum_{\substack{n=n_1+n_2+n_3,\\ \nbar_{123}\in\X_*(n), \\ |n_2| \underset{ (>)}{\geq} |n_3| } } n_1 \intt_\R K_G\big(\tau-n^3, \ld - n^3, \Phi(\nbar_{123})\big) 
	 \intt_{\ld = \tau_1 + \tau_2 + \tau_3} \prod_{j=1}^3 \ft{u}_j (\tau_j, n_j) \, d\tau_1 \, d\tau_2 \,d\ld,
	\end{multline*}
	where the kernel $K_G$ is given by the following expression
	\begin{align}
	K_G (\tau, \ld, \Phi) = \intt_\R \bigg(\ft{\varphi}(\tau-\mu) \ft{\varphi}(\mu - \ld) \frac{1}{\Phi} \H\ft{\eta}\Big(\frac{\mu}{\Phi}\Big) + \ft{\varphi}(\tau - \mu) \H \ft{\varphi} (\mu - \ld ) \frac{1}{\Phi} \ft{\eta} \Big(\frac{\mu}{\Phi}\Big) \bigg)\, d\mu , \label{KY}
	\end{align}
	and satisfies the following estimates
	\begin{align}
	| K_G (\tau, \ld, \Phi) | & \les_{\al, \varphi, \eta} \frac{1}{\jb{\tau-\ld}^\al} \min\bigg(\frac{1}{\jb{\Phi}}, \frac{1}{\jb{\ld}}\bigg) + \frac{1}{\jb{\tau-\ld}} \min\bigg( \frac{1}{\jb{\Phi}}, \frac{1}{\jb{\tau}} \bigg), \nonumber\\
	& \les \frac{1}{\jb{\tau-\ld}} \min\bigg( \frac{1}{\jb{\Phi}}, \frac{1}{\jb{\tau}} \bigg), \label{estimateKY_new}
	\end{align}
	where $\al$ is a large enough positive number.
\end{proposition}

\begin{proof}
	Since the relation between $|n_2|$ and $|n_3|$ will not play an important role in the proof, we will use $\G_{*}$ to denote both $\G_{*, \geq},\G_{*,>}$.
	We want to calculate the following
	\begin{multline*}
	\Ft_{t,x} \big( \G_*(u_1, u_2,u_3) \big) (\tau, n) \\
	= \sum_{\substack{n=n_1+n_2+n_3,\\ \nbar_{123} \in \X_*(n)}} in_1 \intt_\R \ft{\varphi} (\tau-\mu)  \Ft_t \bigg( \int_0^t e^{i(t-t')n^3} \eta\big( \Phi(\nbar_{123})(t-t')\big) F(t') \, dt' \bigg) (\mu) \, d\mu,
	\end{multline*}
	where $F(t) = \varphi(t) \prod_{j=1}^3 \ft{u}_j(t, n_j)$.
	Denote $\Phi(\nbar_{123})$ by $\Phi$, for simplicity. Proceeding as in the proof of Proposition~\ref{prop:KI}, we have
	\begin{align*}
	&\Ft_{t,x} \big(\G_* (u_1, u_2,u_3) \big)(\tau, n)  \\
	& = \sum_{\substack{n=n_1+n_2+n_3,\\ \nbar_{123} \in \X_*(n)}} n_1 \intt_\R \ft{\varphi}(\tau-\mu) \bigg[ \ft{F}(\mu) \frac{1}{\Phi} \H \ft{\eta} \bigg(\frac{\mu-n^3}{\Phi}\bigg) + \H \ft{F}(\mu) \frac{1}{\Phi} \ft{\eta} \bigg(\frac{\mu-n^3}{\Phi} \bigg) \bigg] d\mu
	\end{align*}
	Substituting $\ft{F}(\mu)$ and $\H\ft{F}(\mu)$, we obtain the intended expression.

	It remains to show the kernel estimate. First, note that for a Schwartz function $f$, we have 
	\begin{align*}
	|\H f(\xi)|\leq \lim_{\eps\to 0} \intt_{\eps<|\mu|<1}  \bigg| \frac{f(\xi-\mu) - f(\xi)}{\mu} \bigg| d \mu  + \bigg|  \lim_{\eps\to 0} \intt_{\eps<|\mu| <1} \frac{f(\xi)}{\mu} d\mu   \bigg| + \intt_{|\mu| \geq 1} \bigg|  \frac{f(\xi-\mu)}{\mu}  \Bigg| d\mu.
	\end{align*}
	Considering the first contribution, using mean value theorem and distinguishing the cases $|\xi|\les 1$ and $|\xi|\gg 1$ gives
	\begin{align*}
	\intt_{\eps<|\mu|<1} \frac{|f(\xi-\mu) - f(\xi)|}{|\mu|} d\mu 
	\les_{f,\al} \frac{1}{\jb{\xi}^\al},
	\end{align*}
	for any $\al>0$.
	The second contribution is equal to zero, so it only remains to control the third one. Using Lemma~\ref{lm:convolution}, it follows that
	\begin{align*}
	\intt_{|\mu| \geq 1}  \frac{|f(\xi-\mu)|}{|\mu|}  d\mu 
	  \les_{f, \al} \frac{1}{\jb{\xi}}. 
	\end{align*}
	Consequently, 
	\begin{equation}
	|\H f(\xi)| \les \frac{1}{\jb{\xi}}. \label{hilbert}
	\end{equation}
	Since $\ft{\varphi}$ is a Schwartz function, using \eqref{hilbert},
	\begin{align*}
	\frac{1}{\jb{\Phi}}\bigg| \H\ft{\eta} \bigg(\frac{\mu}{\Phi}\bigg)\bigg|  \les \frac{1}{\jb{\Phi}} \bigg(\1_{|\mu| \leq |\Phi|} + \1_{|\mu| \geq |\Phi|\geq 1} \frac{\jb{\Phi}}{\jb{\mu}} \bigg) \les \min\bigg(\frac{1}{\jb{\Phi}}, \frac{1}{\jb{\mu}} \bigg). 
	\end{align*}
	Now, considering the kernel and the estimates for $\H\ft{\eta}, \H\ft{\varphi}$, we have the following
	\begin{align*}
	|K_G(\tau, \ld, \Phi)| 
	& \les \intt_\R |\ft{\varphi}(\tau-\mu) \ft{\varphi}(\mu - \ld)| \min\bigg(\frac{1}{\jb{\Phi}}, \frac{1}{\jb{\mu}} \bigg) d\mu \\
	&\phantom{xx} + \intt_\R \frac{1}{\jb{\Phi}\jb{\mu-\ld}} \bigg|\ft{\varphi}(\tau-\mu) \ft{\eta}\bigg(\frac{\mu}{\Phi}\bigg)\bigg|  d\mu =: \I + \II.
	\end{align*}
	From the first contribution $\I$, from Lemma~\ref{lm:convolution} we have
	\begin{align*}
	 \I \les_{\varphi, \al} \intt_\R \frac{1}{\jb{\tau-\mu}^{\al+1} \jb{\mu-\ld}^{\al+1}} \min\bigg(\frac{1}{\jb{\Phi}}, \frac{1}{\jb{\mu}} \bigg) d\mu \les \frac{1}{\jb{\tau-\ld}^\al} \min\bigg( \frac{1}{\jb{\Phi}}, \frac{1}{\jb{\ld}}\bigg).
	\end{align*}

	For the second contribution $\II$, applying Lemma~\ref{lm:convolution} or Cauchy-Schwarz inequality give the following estimates
	\begin{align*}
	 \II& \les_{\varphi, \eta}  \intt_\R \frac{1}{\jb{\Phi} \jb{\mu-\ld} \jb{\tau-\mu}^{1+}} d\mu 
	 \les \frac{1}{\jb{\Phi} \jb{\tau-\ld}},\\
	 \II & \les_{ \varphi, \eta} \bigg(\intt_\R \frac{d\mu}{\jb{\tau-\mu}^{2} \jb{\mu}^{2}} \bigg)^\frac12 \bigg(\intt_\R \frac{d\mu}{\jb{\tau-\mu}^{2} \jb{\mu-\ld}^{2}} \bigg)^\frac12  \les \frac{1}{\jb{\tau} \jb{\tau - \ld}}.
	\end{align*}
	Consequently, $\II \les \frac{1}{\jb{\tau-\ld}} \min\bigg(\frac{1}{\jb{\Phi}}, \frac{1}{\jb{\tau}} \bigg)$,
	and the first estimate follows. For \eqref{estimateKY_new}, we consider different cases $\max(\jb{\Phi}, \jb{\ld}) \ges \max(\jb{\Phi}, \jb{\tau})$ or $\max(\jb{\Phi}, \jb{\ld}) \ll \max(\jb{\Phi}, \jb{\tau})$. Note that for the latter, $\max(\jb{\Phi}, \jb{\tau}) = \jb{\tau}$ and $\jb{\tau-\ld}\sim\jb{\tau}$. The estimate follows by choosing $\al \geq 2$.

\end{proof}
\begin{remark}\rm 
	For $*\in\{A,B\}$, consider the operators $\D  \mathcal{NR}_{*, \geq}(u_1,u_2,u_3)$ and $\G_{*,\geq}(u_1,u_2,u_3)$, and the kernel estimates in Propositions~\ref{prop:KI} and \ref{prop:KY}. Then, 
	\begin{align*}
	\big| \Ft_{t,x} \big(\D  \mathcal{NR}_{*, \geq} (u_1,u_2,u_3)\big) (\tau,n) \big| 
	 &\les \sum_{\substack{n=n_1+n_2+n_3, \\ \nbar_{123}\in\X_*(n), \\ |n_2| \geq |n_3| }} \intt_\R \frac{|n_1|}{\jb{\tau-\ld} \jb{\tau-n^3}} \ft{F}(\ld,\nbar_{123}) \, d\ld, \\
	\big| \Ft_{t,x} \big(\G_{*, \geq} (u_1,u_2,u_3) \big) (\tau,n) \big| 
	 &\les \sum_{\substack{n=n_1+n_2+n_3 ,\\ \nbar_{123}\in\X_*(n), \\ |n_2| \geq |n_3| }} \intt_\R \frac{|n_1|}{\jb{\tau-\ld}} \min\bigg( \frac{1}{\jb{\Phi(\nbar_{123})}}, \frac{1}{\jb{\tau-n^3}} \bigg) \\
	 & \phantom{XXXXXXXX}\times  \ft{F}(\ld,\nbar_{123}) \, d\ld,
	\end{align*}
	where $\ft{F}(\ld,\nbar_{123}) =   \big(|\ft{u}_1(\cdot,n_1)| \ast  |\ft{u}_2(\cdot,n_2)| \ast  |\ft{u}_3(\cdot,n_3)| \big)(\ld)$.
	Thus, for the modified Duhamel operators $\G_{*, \geq}$ we can `exchange' the gain in time derivatives through $\jb{\tau-n^3}$ for spatial derivatives through $\jb{\Phi(\nbar_{123})}$, unlike the usual Duhamel operator.
\end{remark}

Now we estimate the kernel of the remainder operators $\B_{*,\geq}, \B_{*,>}$, $*\in\{A,B\}$. The assumptions on $\eta$ in \eqref{eta} play an important role in establishing the following kernel estimates.
\begin{proposition}\label{prop:KX} \
	Let $*\in\{A,B\}$. Then, the convolution operators $\B_{*, \geq}, \B_{*,>}$ have the following space-time Fourier transform 
	\begin{multline*}
	\Ft_{t,x} \big( \B_{*, \underset{(>)}{\geq}} (u_1,u_2,u_3) \big)(\tau,n) 
	\\
	=  \sum_{\substack{n=n_1+n_2+n_3, \\ \nbar_{123}\in\X_*(n), \\|n_2| \underset{(>)}{\geq} |n_3|} } n_1 \intt_\R K_B(\tau-n^3, \ld - n^3, \Phi(\nbar_{123}))
	 \intt_{\ld = \tau_1 + \tau_2 + \tau_3} \prod_{j=1}^3 \ft{u}_j (\tau_j, n_j) \, d\tau_1  \, d\tau_2 \, d\ld,
	\end{multline*}
	where the kernel $K_B$ is given by 
	\begin{multline*}
	K_B (\tau, \ld, \Phi) = \intt_\R \frac{\ft{\varphi}(\tau - \mu) - \ft{\varphi}(\tau)}{\mu} \ft{\varphi}(\mu-\ld) \, d\mu - \intt_\R \ft{\varphi}(\tau-\mu) \ft{\varphi}(\mu-\ld) \frac1\Phi \H\ft{\eta}\Big(\frac{\mu}{\Phi}\Big)d\mu \\
	+ \intt_\R \ft{\varphi}(\tau - \mu) \H \ft{\varphi}(\mu-\ld) \frac1\Phi \ft{\eta}\bigg(\frac{\mu}{\Phi}\bigg) \, d\mu  ,
	\end{multline*}
	and satisfies the following estimate
	\begin{equation*}
	\begin{aligned}
	| K_B (\tau, \ld, \Phi) |  &\les_{\al, \varphi, \eta}  \frac{1}{\jb{\tau}^\al \jb{\ld}} + \frac{\jb{\ld + \Phi}}{\jb{\tau-\ld}^\al \jb{\ld}} \min\bigg( \frac{1}{\jb{\Phi}}, \frac{1}{\jb{\ld}}\bigg)  \\
	& \phantom{XXX}+ \frac{\jb{\tau+\Phi}}{\jb{\tau-\ld}} \min\bigg( \frac{1}{\jb{\Phi}}, \frac{1}{\jb{\tau}} \bigg)^2, 
	\end{aligned}
	\end{equation*}
	for any  $\al>0$.
	
\end{proposition}
\begin{proof}
	Let $*\in\{A,B\}$ and let $\B_*$ denote both $\B_{*, \geq}$ and $\B_{*,>}$. By definition, we have that $\B_* = \D   \mathcal{NR}_*  - \G_*$. 
	Using Propositions~\ref{prop:KI} and \ref{prop:KY}, we have that $\Ft_{t,x}(\B_*(u_1,u_2,u_3))$ has the intended formula, with the kernel $K_B$ defined as $K_B(\tau, \ld, \Phi)  = -i K(\tau, \ld) - K_G(\tau, \ld, \Phi)$. Using the expressions \eqref{K} and \eqref{KY} for the kernels, we obtain the following
	\begin{align*}
	& K_B(\tau, \ld, \Phi) \\
	 & \phantom{xx}= \intt_{|\mu| \leq 1} \big(\ft{\varphi}(\tau-\mu) - \ft{\varphi}(\tau) \big) \ft{\varphi}(\mu-\ld) \frac1\mu \, d\mu 
	- \intt_{|\mu|\leq1} \ft{\varphi}(\tau - \mu) \ft{\varphi}(\mu-\ld) \frac1\Phi \H\ft{\eta}\bigg(\frac{\mu}{\Phi}\bigg) \, d\mu \\
	& \phantom{xxx}+ \intt_{|\mu|>1} \ft{\varphi}(\tau - \mu) \ft{\varphi}(\mu-\ld) \bigg\{ \frac1\mu - \frac1\Phi \H\ft{\eta}\bigg(\frac{\mu}{\Phi}\bigg) \bigg\}\, d\mu 
	+ \intt_\R \ft{\varphi}(\tau - \mu) \H \ft{\varphi}(\mu-\ld) \frac1\Phi \ft{\eta}\bigg(\frac{\mu}{\Phi}\bigg) \, d\mu \\
	 &  \phantom{xxx}- \intt_{|\mu|>1} \ft{\varphi}(\tau) \ft{\varphi}(\mu-\ld) \frac1\mu \, d\mu \\
	 &  \phantom{xx}=: \I_1 + \I_2 + \I_3 + \I_4 + \I_5.
	\end{align*}
	For $\I_1$, using mean value theorem gives
	\begin{align*}
	|\I_1| & \les \intt_{|\mu| \leq 1} \1_{|\tau| \les 1} \frac{d\mu }{\jb{\tau}^\al \jb{\mu}^\al \jb{\mu-\ld}^{1+\al}}  + \intt_{|\mu| \leq 1} \1_{|\tau| \gg 1} \frac{d\mu }{\jb{\tau}^\al \jb{\mu}^\al \jb{\mu-\ld}^{1+\al}} 
	 \les \frac{1}{\jb{\tau}^\al \jb{\ld}^\al},
	\end{align*} 
	for any $\al>0$.
	
	For the following contribution, using \eqref{hilbert} and Cauchy-Schwarz inequality gives
	\begin{align*}
	|\I_2| 
	& \les \intt_{|\mu|\leq 1} \frac{d\mu }{\jb{\Phi} \jb{\tau-\mu}^{\al} \jb{\mu-\ld}^{\al+1} \jb{\mu}^{\al}} 
	 \les \frac{1}{\jb{\Phi} \jb{\tau-\ld}^\al \jb{\ld}^\al}. 
	\end{align*}
	
	Before estimating $\I_3$, 
	note that since $\H\ft{\eta}(-1) = -1$ and using mean value theorem, we get
	\begin{align*}
	\bigg| \frac1\mu - \frac1\Phi \H\ft{\eta}\bigg(\frac{\mu}{\Phi}\bigg) \bigg| & \sim \frac{1}{\jb{\mu}} \bigg| \H\ft{\eta}(-1) - \frac{\mu}{\Phi} \H \ft{\eta} \bigg(\frac{\mu}{\Phi}\bigg)\bigg| \\
	& \les_{\eta} \1_{\jb{\Phi} \ges \jb{\mu}} \frac{1}{\jb{\mu}} \bigg| -1 - \frac{\mu}{\Phi}\bigg|  
	+  \1_{\jb{\Phi} \ll \jb{\mu}} \frac{\jb{\mu+\Phi}}{\jb{\mu}^2} \\
	& \les \frac{\jb{\mu+\Phi}}{\jb{\mu}} \min\bigg( \frac{1}{\jb{\Phi}}, \frac{1}{\jb{\mu}} \bigg).
	\end{align*}
	Using the above estimate, it follows from previous arguments that 
	\begin{align*}
	|\I_3| 
	& \les \1_{\jb{\Phi}\ges \jb{\tau}} \frac{\jb{\ld + \Phi}}{\jb{\tau-\ld}^\al \jb{\ld} \jb{\Phi}}+  \1_{\jb{\Phi}\ll \jb{\tau}} \frac{\jb{\ld+\Phi}}{\jb{\tau-\ld}^\al \jb{\ld}^2} 
	 \les \frac{\jb{\ld + \Phi}}{\jb{\tau-\ld}^\al \jb{\ld}} \min\bigg( \frac{1}{\jb{\Phi}}, \frac{1}{\jb{\ld}}\bigg),
	\end{align*}
	for any $\al>0$.
	
	In order to estimate $\I_4$, we start by showing a bound for $\frac{1}{\Phi} \ft{\eta} \big(\frac{\mu}{\Phi}\big)$. If $|\Phi| \ges |\mu|$, we use the fact that $\ft{\eta}(-1) =0$ and mean value theorem. Otherwise, $|\Phi| \ll |\mu| \implies \jb{\mu - \Phi} \sim \jb{\mu}$. It follows that
	\begin{align*}
	\bigg| \frac{1}{\Phi} \ft{\eta} \bigg(\frac{\mu}{\Phi}\bigg) \bigg| & \les \1_{|\Phi| \ges |\mu|} \frac{1}{\jb{\Phi}} \bigg| \ft{\eta}\bigg(\frac{\mu}{\Phi}\bigg) - \ft{\eta}(-1) \bigg| + \1_{|\Phi| \ll |\mu|} \frac{1}{\jb{\mu}} \bigg| \frac{\mu}{\Phi} \ft{\eta} \bigg(\frac{\mu}{\Phi} \bigg) \bigg| \\
	& \les _{\eta}\jb{\mu+\Phi}\min \bigg( \frac{1}{\jb{\Phi}}, \frac{1}{\jb{\mu}}\bigg)^2.
	\end{align*}
	Using the above estimate and the fact that $|\H\ft{\varphi}(\mu-\ld)| \les \jb{\mu-\ld}^{-1}$ in \eqref{hilbert}, we have
	\begin{align*}
	|\I_4| 
	& \les \1_{\jb{\Phi}\ges \jb{\tau}} \frac{\jb{\tau + \Phi}}{\jb{\Phi}^2 \jb{\tau-\ld}} + \1_{\jb{\Phi}\ll \jb{\tau}} \frac{\jb{\tau+\Phi}}{\jb{\tau-\ld} \jb{\tau}^2} 
	\les \frac{\jb{\tau+\Phi}}{\jb{\tau-\ld}} \min\bigg( \frac{1}{\jb{\Phi}}, \frac{1}{\jb{\tau}} \bigg)^2.
	\end{align*}
	
	For the last contribution, we have
	\begin{align*}
	|\I_5| & \les  \intt_{|\mu|>1} \frac{\jb{\tau}^{\al} |\ft{\varphi}(\tau)|\jb{\mu-\ld}| \ft{\varphi}(\mu-\ld)|}{ \jb{\tau}^{\al} \jb{\mu-\ld} \jb{\mu}} \ d\mu 
	\les \frac{1}{\jb{\tau}^\al \jb{\ld}},
	\end{align*}
	for any $\al>0$.
	The estimate follows.

\end{proof}

We want to further split the operators $\B_{*,\geq}, \B_{*,>}$, $*\in\{A,B\}$, to obtain better kernel estimates. 
We will split the kernel $K_B$ in two pieces: when we can estimate the multiplier directly, and when we also have to use $\s_{\max} = \max_{j=1,2,3} |\tau_j - n_j^3|$. Let $K_B = K_{0} + K_{+}$ where the kernels are defined below
\begin{align*}
K_{0}(\tau, \ld, \Phi) & = \1_{\jb{\ld} \ges \jb{\Phi}} \bigg( \intt_{|\mu| \leq 1} \big(\ft{\varphi}(\tau-\mu) - \ft{\varphi}(\tau) \big) \ft{\varphi}(\mu-\ld) \frac1\mu \ d\mu - \intt_{|\mu|>1} \ft{\varphi}(\tau) \ft{\varphi}(\mu-\ld) \frac1\mu \ d\mu \bigg) \\
& + \1_{\jb{\ld +\Phi} \les \jb{\tau-\ld}} \bigg(\intt_{|\mu|>1} \ft{\varphi}(\tau - \mu) \ft{\varphi}(\mu-\ld) \bigg\{ \frac1\mu - \frac1\Phi \H\ft{\eta}\bigg(\frac{\mu}{\Phi}\bigg) \bigg\}d\mu\bigg)\\
& + \1_{\jb{\tau+\Phi} \les \jb{\tau-\ld}} \bigg(\intt_\R \ft{\varphi}(\tau - \mu) \H \ft{\varphi}(\mu-\ld) \frac1\Phi \ft{\eta}\bigg(\frac{\mu}{\Phi}\bigg) \ d\mu\bigg), \\
K_{+} (\tau, \ld, \Phi)& = \1_{\jb{\ld} \ll \jb{\Phi}} \bigg( \intt_{|\mu| \leq 1} \big(\ft{\varphi}(\tau-\mu) - \ft{\varphi}(\tau) \big) \ft{\varphi}(\mu-\ld) \frac1\mu \ d\mu - \intt_{|\mu|>1} \ft{\varphi}(\tau) \ft{\varphi}(\mu-\ld) \frac1\mu \ d\mu \bigg) \\
& + \1_{\jb{\ld + \Phi} \gg \jb{\tau-\ld}}\intt_{|\mu|>1} \ft{\varphi}(\tau - \mu) \ft{\varphi}(\mu-\ld) \bigg\{ \frac1\mu - \frac1\Phi \H\ft{\eta}\bigg(\frac{\mu}{\Phi}\bigg) \bigg\}d\mu\\
& + \1_{\jb{\tau + \Phi} \gg \jb{\tau-\ld}} \bigg(\intt_\R \ft{\varphi}(\tau - \mu) \H \ft{\varphi}(\mu-\ld) \frac1\Phi \ft{\eta}\bigg(\frac{\mu}{\Phi}\bigg) \ d\mu\bigg).
\end{align*}
Thus, we have the following estimates for the kernels, for any $0\leq \al \leq 1$,
\begin{align}
|K_{0}(\tau, \ld, \Phi)| & \les \frac{1}{\jb{\tau-n^3}^{1+\al} \jb{\Phi}^{1-\al}} , \label{estimateK0}\\
|K_{+}(\tau, \ld, \Phi)| & \les \1_{\jb{\ld} \ll \jb{\Phi}} \frac{1}{\jb{\tau-\ld} \jb{\tau}} +  \frac{\jb{\ld+\Phi}^{1-\al}}{\jb{\tau-\ld} \jb{\tau}} \min\bigg( \frac{1}{\jb{\Phi}} ,  \frac{1}{\jb{\tau}} \bigg)^{1-\al}.\label{estimateK+}
\end{align}
Considering the contribution corresponding to the kernel $K_0$ in $\B_{*, \geq}, \B_{*, >}$, it can be easily estimated. However, in order to estimate the one corresponding to $K_+$, we will have to use the  largest modulation. In particular, using \eqref{estimateK+}, we see that 
\begin{equation}
\big|K_{+}\big(\tau-n^3, \ld - n^3, \Phi(\nbar_{123})\big) \big| \les \frac{\jb{\s_{\max}}^{1-\beta}}{\jb{\tau-\ld} \jb{\tau-n^3} \jb{\Phi(\nbar_{123})}^{1-\al}}, \label{KX+}
\end{equation}
for any $0\leq \al \leq 1$, since
$\ld = \tau_1 + \tau_2 + \tau_3$ and
$$|\ld - n^3 + \Phi(\nbar_{123})| = |\tau_1 - n_1^3 + \tau_2 - n_2^3 + \tau_3 - n_3^3| \les \max\limits_{j=1,2,3} |\tau_j - n_j^3| = \s_{\max}.$$
Thus, we can use $\s_{\max}$ in order to estimate the numerator of the second contribution in \eqref{estimateK+}, which motivates splitting the operators depending on which modulation is the largest. In particular, we have
\begin{align}
\B_{*, \underset{(>)}{\geq}} &=  \B^{0}_{*,\underset{(>)}{\geq}} + \B^{1}_{*,\underset{(>)}{\geq}} + \B^{2}_{*,\underset{(>)}{\geq}} + \B^{3}_{*, \underset{(>)}{\geq}}, \label{EXJ}
\end{align}
where $ \B^{0}_{*,\underset{(>)}{\geq}}$ has kernel $K_{0}$ and $\B^{j}_{*,\underset{(>)}{\geq}}$ has kernel $K_{+}$ localized to the region where ${\s_{\max} = |\s_j|}$, $j=1,2,3$.

\section{System of equations}\label{sec:equations}
Instead of running a contraction mapping argument on the integral equation \eqref{duhamel}, we will solve an ordered system of equations. The first equation imposes structure on $u$ centered around a smoother-in-time function $w$, while the second guarantees that $u$ solves the Duhamel formulation. In order to obtain a solution of mKdV2 \eqref{renorm2}, we start by solving the equation for $u$ obtaining a function parameterized by $w$, $u=u[w]$. Then, we apply a \emph{partial} second iteration argument, using the  additional structure of $u$, in order to obtain an appropriate center $w$.

In this section, we establish the relevant equations for $u$ and $w$ and the main results needed to show Theorem~\ref{th:lwp}. For a fixed $p$ in $2\leq p <\infty$, we will only focus on showing local well-posedness in $\FL^{\frac12,p}(\T)$. The result for $s>\frac12$ follows from a persistence of regularity argument. See Remark~\ref{rm:persistence} for more details.

For a fixed $w\in Z_0$, we consider the following equation for $u$
\begin{equation}
u = w + \varphi_T \big[ \G_{A, \geq}(w,\conj{u},u) +  \G_{A, >}(w,u,\conj{u}) +   \G_{B, \geq}(w,\conj{w},u) +  \G_{B, >}(w,w, \conj{u}) \big]. \label{eq:u} 
\end{equation}
We start by showing that the equation \eqref{eq:u} is locally well-posed via a fixed point argument, obtaining $u=u[w]$, i.e., $u$ parameterized by $w$.
\begin{proposition}\label{prop:u}
	For any $w\in Z_0$ satisfying $\|w\|_{Z_0} \leq A_2$, there exists a unique $u\in Y_0$ with $\|u\|_{Y_0} \leq A_3$ satisfying \eqref{eq:u}, for some $T=T(A_2)>0$. The mapping $w \mapsto u[w]$ is Lipschitz from the $A_2$-ball of $Z_0$ to the $A_3$-ball of $Y_0$.
\end{proposition}

Now, we want to establish an equation for $w$ such that if $u=u[w]$ satisfies \eqref{eq:u}, then it also satisfies the Duhamel formulation \eqref{duhamel_truncated}. 
We see that $w$ must satisfy the following equation
\begin{equation}
\begin{aligned}
w 
& = \varphi(t) S(t) u_0 + \varphi_T \cdot \mathcal{IR}(u, u, u) \\
&  \phantom{XXXX}+   \varphi_T \cdot  \D  \mathcal{NR}_{C, \geq}(u,\conj{u}, u) +   \varphi_T \cdot  \D  \mathcal{NR}_{C, >}(u,u, \conj{u}) \\
&  \phantom{XXXX}+   \varphi_T \cdot \D  \mathcal{NR}_{D, \geq}(u,\conj{u},u) +   \varphi_T \cdot \D  \mathcal{NR}_{D, >}(u,u, \conj{u})\\
& \phantom{XXXX} +   \varphi_T\big[ \B_{A, \geq}(w,\conj{u},u) + \B_{A, >}(w,u, \conj{u})  + \B_{B, \geq}(w,\conj{w}, u) + \B_{B, >}(w,w,\conj{u}) \big] \\
&  \phantom{XXXX} +   \varphi_T \big[ \D  \mathcal{NR}_{A, \geq} (u,\conj{u},u) - \D  \mathcal{NR}_{A, \geq} (w,\conj{u},u) \big] \\
& \phantom{XXXX}+ \varphi_T \big[ \D  \mathcal{NR}_{A, >} (u,u,\conj{u}) - \D  \mathcal{NR}_{A, >} (w,u,\conj{u}) \big] \\
& \phantom{XXXX} +   \varphi_T \big[ \D  \mathcal{NR}_{B, \geq} (u,\conj{u}, u) - \D  \mathcal{NR}_{B, \geq} (w,\conj{w}, u) \big] \\
& \phantom{XXXX} + \varphi_T \big[ \D  \mathcal{NR}_{B, >} (u,u, \conj{u}) - \D  \mathcal{NR}_{B, >} (w,w, \conj{u})  \big] .
\end{aligned}\label{w_first}
\end{equation}
The equation \eqref{eq:u} imposes structure on $u$, which we want to exploit in order to modify the above equation for $w$. We want to use the idea introduced by Bourgain in \cite{BO97} and apply a partial second iteration by replacing $u$ by its equation \eqref{eq:u} where appropriate. The decomposition on the operators $\D  \mathcal{NR}$ and $\B$, introduced in Sections \ref{sec:preliminaries} and \ref{sec:kernel}, explicitly identify which entries have the largest frequencies and the largest modulations, respectively. This information will guide how we apply second iteration. 

For the terms $\D  \mathcal{NR}_{*, \geq}, \D  \mathcal{NR}_{*,>}$, $*\in\{A,B,C,D\}$, we replace the equation for $u$ \eqref{eq:u} from left to right to obtain only cubic and quintic terms, as in the following example
\begin{align*}
\D  \mathcal{NR}_{C,\geq}(u,\conj{u}, u) & = \D  \mathcal{NR}_{C, \geq} (w,\conj{u}, u) \\
& + \D  \mathcal{NR}_{C, \geq} \big( \varphi_T \cdot \G_{A,\geq} [w,\conj{u}, u] , \conj{u}, u\big) +   \D  \mathcal{NR}_{C, \geq} \big( \varphi_T \cdot \G_{A,>} [w,u,\conj{u}] , \conj{u}, u\big) \\
&+  \D  \mathcal{NR}_{C, \geq} \big( \varphi_T \cdot \G_{B,\geq} [w,\conj{w}, u] , \conj{u}, u\big) + \D  \mathcal{NR}_{C, \geq} \big( \varphi_T \cdot \G_{B,>} [w,w, \conj{u}] , \conj{u}, u\big), \\
\D  \mathcal{NR}_{C, \geq} (w,\conj{u}, u)  & =\D  \mathcal{NR}_{C, \geq} (w,\conj{w}, u)  \\
& + \D  \mathcal{NR}_{C, \geq} \big( w, \conj{\varphi_T \cdot \G_{A,\geq} [w,\conj{u}, u] },u\big) +   \D  \mathcal{NR}_{C, \geq} \big( w, \conj{\varphi_T \cdot \G_{A,>} [w,u,\conj{u}] }, u\big) \\
&+  \D  \mathcal{NR}_{C, \geq} \big( w, \conj{\varphi_T \cdot \G_{B,\geq} [w,\conj{w}, u]} , u\big) + \D  \mathcal{NR}_{C, \geq} \big( w, \conj{\varphi_T \cdot \G_{B,>} [w,w, \conj{u}] }, u\big),\\
\D  \mathcal{NR}_{C, \geq} (w,\conj{w}, u) & = \D  \mathcal{NR}_{C, \geq} (w,\conj{w}, w) \\
&+  \D  \mathcal{NR}_{C, \geq} \big( w, \conj{w}, \varphi_T \cdot \G_{A,\geq} [w,\conj{u}, u] \big) + \D  \mathcal{NR}_{C, \geq} \big( w, \conj{w}, \varphi_T \cdot \G_{A,>} [w,u, \conj{u}] \big), \\
&+  \D  \mathcal{NR}_{C, \geq} \big( w, \conj{w}, \varphi_T \cdot \G_{B,\geq} [w,\conj{w}, u] \big) + \D  \mathcal{NR}_{C, \geq} \big( w, \conj{w}, \varphi_T \cdot \G_{B,>} [w,w, \conj{u}] \big).
\end{align*}
 This strategy prioritizes the entry corresponding to the derivative followed by the one with the largest frequency between the remaining two factors. For $\D  \mathcal{NR}_{*,\geq}, \D  \mathcal{NR}_{*, >}$ with $*\in\{A,B\}$ there will be no cubic terms after second iteration, due to the differences in \eqref{w_first}.

For the terms $\B_{*,\geq}, \B_{*, >}$, with $*\in\{A,B\}$, we split the operators into four pieces $\B^{j}_{*,\geq}, \B^{j}_{*,>}$, $j=0,1,2,3$ as defined in \eqref{EXJ}. The contributions corresponding to $j=0$ are easily estimated, but for $j=1,2,3$ the largest modulation plays an important role in estimating the kernel. If the $j-$th entry corresponds to a $u$ or $\conj{u}$ term, we replace it with the equation for $u$ \eqref{eq:u}.
For example, we have
\begin{align*}
\B_{B, \geq}(w, \conj{w}, u) & = \B^{0}_{B, \geq} (w, \conj{w}, w) + \B^{1}_{B,\geq} (w,\conj{w}, u)  \\
&+ \B^{2}_{B, \geq} (w,\conj{w}, u) + \B^{3}_{B, \geq}(w,\conj{w},w) \\
&  +   \B^{3}_{B, \geq} \big(w, \conj{w}, \varphi_T\cdot \G_{A, \geq}[w, \conj{u},u]\big) +  \B^{3}_{B, \geq} \big(w, \conj{w}, \varphi_T\cdot \G_{A, >}[w, u,\conj{u}]\big)\\
&+   \B^{3}_{B, \geq} \big(w, \conj{u}, \varphi_T\cdot \G_{B, \geq}[w, \conj{w},u]\big) + \B^{3}_{B, \geq} \big(w, \conj{u}, \varphi_T\cdot \G_{B, >}[w, w,\conj{u}]\big).
\end{align*}

Proceeding as detailed above, we obtain a new equation for $w$. Due to its length, we have decided to only include it in Appendix~\ref{app:w}.

\begin{proposition}\label{prop:w}
	For any $u_0\in \FL^{\frac12,p}(\T)$ satisfying $\|u_0\|_{\FL^{\frac12,p}} \leq A_1$, there exists a unique $w\in Z_0$ with $\|w\|_{Z_0} \leq A_2$ satisfying \eqref{w_first}, for some $T=T(A_1,A_2,A_3)>0$. The mapping $u_0 \mapsto w$ is Lipschitz from the $A_1$-ball of $\FL^{\frac12,p}(\T)$ to the $A_2$-ball of $Z_0$.
\end{proposition}
\begin{remark}\rm
	In order to show Proposition~\ref{prop:w}, we will not run a contraction mapping argument for the map defined by the right-hand side of \eqref{w_first} nor the equation \eqref{eq:w} included in Appendix~\ref{app:w}. Some quintic terms in \eqref{eq:w} require the use of the equation for $u$ \eqref{eq:u} once again, introducing new quintic terms but also new septic terms. Given the considerable number of new terms that this additional step introduces, we have decided to omit them when presenting the equation for $w$. The strategy for obtaining the new contributions is described in Section~\ref{sec:quintic} along with the estimates needed for both the quintic and septic terms. 
	
\end{remark}

\begin{remark}\rm \label{rm:persistence}
	In Sections \ref{sec:u} and \ref{sec:w}, we will establish the estimates needed to show Propositions~\ref{prop:u} and \ref{prop:w}, from which Theorem~\ref{th:lwp} follows for $s=\frac12$. In order to extend this result to $s>\frac12$, note that the required estimates in 
	\begin{align*}
	Y'_0 = X^{s,\frac12}_{p,r_0}, \quad Z'_0 = X^{s,b_0}_{p,q_0},
	\end{align*}
	follow easily from the estimates shown by associating the extra weight $\jb{n}^{s-\frac12}$ in the norm to the function with the largest frequency. Consequently, by a persistence of regularity argument, we obtain a unique solution $u\in C([-T,T];\FL^{s,p}(\T))$ where $T = T\big( \| u_0\|_{\FL^{\frac12,p}} \big)>0$.
	
\end{remark}

In the remaining of the paper, we will establish the nonlinear estimates needed to show Propositions~\ref{prop:u} and \ref{prop:w}. The results follow from a contraction mapping argument in $Y_0$ and $Z_0$, respectively. From Lemma~\ref{lm:T}, it suffices to estimate the terms in equations \eqref{eq:u} and \eqref{eq:w} in $Y_1$ and $Z_1$, respectively, dropping the factors of $\varphi_T$.

 For simplicity, we will use the notation $\D  \mathcal{NR}_{*}$ to denote $\D  \mathcal{NR}_{*, \geq}$, $\D  \mathcal{NR}_{*,>}$, for $*\in\{A,B,C,D\}$, $\G_*$ to denote $\G_{*, \geq}, \G_{*,>}$ for $*\in\{A,B\}$, and $\B^{j}_{*}$ for $\B^{j}_{*, \geq}, \B^{j}_{*,>}$ for $*\in\{A,B\}$ and $j=0,1,2,3$. In the estimates, there is no distinction between the frequency regions where $|n_2|\geq |n_3|$ and $|n_2| > |n_3|$, motivating this simplified notation.

\section{Proof of Proposition~\ref{prop:u}}\ \label{sec:u}
In this section, we establish the estimates needed to show Proposition~\ref{prop:u}.

\begin{lemma}\label{lm:EY}
	The following estimates hold
	\begin{align*}
	\| \G_{A} (u_1,u_2,u_3) \|_{Y_1} &\les \|u_1\|_{Z_0} \|u_2\|_{Y_0} \|u_3\|_{Y_0} ,\\
	\| \G_{B} (u_1,u_2,u_3) \|_{Y_1} &\les \|u_1\|_{Z_0} \|u_2\|_{Z_0} \|u_3\|_{Y_0}.
	\end{align*}
\end{lemma}

\begin{proof}
	We will show that
	\begin{align*}
	\| \G_* (u_1,u_2,u_3) \|_{Y_1} & \les \|u_1\|_{Z_0} \|u_2\|_{Y_0} \|u_3\|_{Y_0},
	\end{align*}
	for $*\in\{A,B\}$, since the second estimate follows from the embedding $Z_0 \embeds Y_0$.
	Using Proposition~\ref{prop:KY}, namely \eqref{estimateKY_new}, and the changes of variables $\tau_j + n^3_j = \s_j  $, $j=1,2,3$, it follows that 
	\begin{align}
		&\| \G_* (u_1,u_2,u_3) \|_{Y_1}\nonumber\\
		& \!\!\!\les  \bigg\|  \sum_{\substack{n=n_1+n_2+n_3, \\ \nbar_{123} \in \X_*(n)}}  \intt_\R \frac{\jb{n}^\frac12 |n_1| }{\jb{\tau - \ld} \jb{\Phi(\nbar_{123})}^\frac12}  \intt_{\ld= \tau_1 + \tau_2 + \tau_3} \prod_{j=1}^3 |\ft{u}_j (\tau_j, n_j)| \ d\ld \bigg\|_{\l^p_n L^{r_1}_{\tau}} \nonumber\\
		 &\!\!\!\les  \bigg\|  \sum_{\substack{n=n_1+n_2+n_3, \\ \nbar_{123} \in \X_*(n)}}  \!\intt_{\s_1, \s_2 , \s_3} \frac{\jb{n}^\frac12 |n_1| }{\jb{\tau -n^3 - \bar{\s} + \Phi(\nbar_{123})} \jb{\Phi(\nbar_{123})}^\frac12}    \prod_{j=1}^3 |\ft{u}_j (\s_j + n_j^3, n_j)|  \bigg\|_{\l^p_n L^{r_1}_{\tau}},\label{change_variables}
	\end{align}
	where $\bar{\s} = \s_1 + \s_2 + \s_3$.
	Note that 
	$$\frac{\jb{n}^\frac12|n_1|}{\jb{\Phi(\nbar_{123})}^\frac12} \les \frac{\jb{n}^\frac12 |n_1|}{\max_{j=1,2,3}\jb{n_j}} \les \jb{n_1}^\frac12.$$
	Minkowski's inequality gives
	\begin{multline*}
	\| \G_* (u_1,u_2,u_3) \|_{Y_1} \les \intt_{\s_1, \s_2, \s_3} \bigg\| \sum_{\substack{n=n_1+n_2+n_3, \\ \nbar_{123} \in \X_*(n)}} \frac{\jb{n_1}^\frac12}{\jb{ \tau - n^3 - \bar{\s} + \Phi(\nbar_{123})}}  
	\prod_{j=1}^3 |\ft{u}_j (\s_j + n_j^3, n_j)|   \bigg\|_{\l^p_n L^{r_1}_{\tau}} .
	\end{multline*}
	Denoting the inner norm by $\I$, we can rewrite the sum as follows
	\begin{align*}
	\I 
	 & \les \Bigg\| \sum_{\mu} \frac{1}{\jb{\tau - n^3 - \bar{\s} + \mu}} \Bigg( \sum_{\substack{n=n_1+n_2+n_3, \\\nbar_{123} \in \X_*^\mu(n)}} \jb{n_1}^\frac12 \prod_{j=1}^3 |\ft{u}_j(\s_j + n_j^3, n_j)| \Bigg) \Bigg\|_{\l^{p}_n L^{r_1}_\tau},
	\end{align*}
	for $X^\mu_*(n)$ defined in \eqref{localizedX}.
	Since the following bounds hold uniformly in $\bar{\s}$ and $n$, for any $\tilde{r}>1$, 
	\begin{align*}
	\sum_\mu \frac{1}{\jb{\tau - n^3 - \bar{\s} + \mu}^{\tilde{r}}} \les 1, \quad 
	\intt_\R\frac{1}{\jb{\tau - n^3 - \bar{\s} + \mu}^{\tilde{r}}} d\tau \les 1,
	\end{align*}
	choosing $\tilde{r} = \frac{1}{1-\dl}$, we have $\frac{1}{r_1} + 1 = \frac{1}{\tilde{r}} + \frac{1}{r_2}$ and we can apply Schur's test\footnote{\textbf{Schur's test:} Let $X,Y$ be measurable spaces, $K:X\times Y \to \R$ a non-negative Schwartz kernel and $1\leq p,q,r\leq \infty$ such that $1 + \frac1q = \frac1p + \frac1r$. If for some $C>0$ we have \begin{align*}
		\sup_{x\in X}\intt_Y |K(x,y)|^r dy + \sup_{y\in Y}
		\intt_X |K(x,y)|^r dx  \leq C^r \quad \text{then} \quad \bigg\| \intt_Y K(x,y) f(y) \ dy \bigg\|_{L^q(X)} \leq C \|f\|_{L^p(Y)}.
		\end{align*}  } to obtain
	\begin{align}
	\I & \les \Bigg\| \sum_{\substack{n=n_1+n_2+n_3, \\\nbar_{123} \in \X^\mu_*(n)}}  \jb{n_1}^\frac12  \prod_{j=1}^3 |\ft{u}_j(\s_j + n_j^3,n_j)| \Bigg\|_{\l^{p}_n \l^{r_2}_{\mu}}. \label{EY_aux}
	\end{align}
	Let $\P_{N_j}$ denote the projection onto $\jb{n}\sim N_j$ and let $f_1(\s, n) = \jb{n_1}^\frac12 |\ft{u}_1(\s+n^3, n)|$, $f_j(\s,n) = | \ft{\P_{N_j} u}_j(\s+n^3, n)|$, $j=2,3$. Then, using Minkowski and H\"older's inequalities, we get
	\begin{align*}
	\I 
	& \les \sum_{N_2,N_3} \bigg\| |\X_*^\mu(n)|^{\frac{1}{r_2'}} \bigg(\sum_{\substack{n=n_1+n_2+n_3,\\\nbar_{123}\in\X_*^\mu(n)}}\prod_{j=1}^3 |f_j(\s_j, n_j)|^{r_2} \bigg)^{\frac{1}{r_2}} \bigg\|_{\l^{p}_n \l^{r_2}_{\mu}}.
	\end{align*}
	If $*=A$, we have $|\Phi(\nbar_{123})| \les |n_1|^3 \sim |n|^3$, so we use Lemma~\ref{lm:divisor} to count the divisors $d_2=n-n_2, d_3=n-n_3$ of $m$. Since
	\begin{align*}
	|d_2 - n| = |n_2| \leq N_2, \quad |d_3 - n| = |n_3| \leq N_3
	\end{align*}
	and $1 \leq |m|^\eps \leq |m|^\frac13 \les |n|$, for any $0<\eps\leq \frac13$, we conclude that there are at most $\mathcal{O}(N_j^\eps)$ choices for $d_j$, $j=2,3$. Since $n$ is fixed, this determines the choices of $n_2,n_3$ and consequently of $n_1$. 
	If $*=B$, then $|\Phi(\nbar_{123})| \les |n_2|^3$ and we can use the standard divisor counting estimate to conclude that there are at most $\mathcal{O}(N_2^\eps)$ choices for $n_2,n_3$. Consequently, $|\X_*^\mu(n)| \les (N_2N_3)^\eps$ and we have
	\begin{align*}
	\I & \les \sum_{N_2,N_3} (N_2 N_3)^\eps \bigg\|  \bigg( \sum_\mu  \sum_{\substack{n=n_1+n_2+n_3, \\\nbar_{123}\in\X^\mu_*(n)}}  \prod_{j=1}^3 |f_j(\s_j, n_j)|^{r_2} \bigg)^{\frac{1}{r_2}} \bigg\|_{\l^{p}_n }\\
	& \les \sum_{N_2,N_3} (N_2 N_3)^\eps \bigg\|  \bigg(  \sum_{\substack{n=n_1+n_2+n_3}} \bigg(\sum_\mu \1_{\Phi(\nbar_{123}) = \mu} \bigg) \prod_{j=1}^3 |f_j(\s_j, n_j)|^{r_2} \bigg)^{\frac{1}{r_2}} \bigg\|_{\l^{p}_n }\\
	& \les \sum_{N_2,N_3} (N_2N_3)^\eps \|f_1\|_{\l^p_n} \|f_2\|_{\l^{r_2}_n} \|f_3\|_{\l^{r_2}_n},
	\end{align*}
	where we use Minkowski's inequality and the fact that $r_2<2\leq p$ in the last inequality. Choosing $\eps< \dl$, we obtain 
	$$\I \les \|\jb{n}^\frac12 \ft{u}_1(\s_1+n^3,n)\|_{\l^p_n} \|\jb{n}^\dl \ft{u}_2(\s_2 + n^3, n)\|_{\l^{r_2}_n} \|\jb{n}^\dl \ft{u}_3(\s_3 + n^3, n)\|_{\l^{r_2}_n}.$$
	Applying this estimate to \eqref{EY_aux} gives
	\begin{align*}
	\| \G_* (u_1,u_2,u_3) \|_{Y_1} & \les  \|\jb{n}^\frac12 \ft{u}_1(\s+n^3,n)\|_{L^1_\s \l^p_n} \prod_{j=2}^3 \|\jb{n}^\dl \ft{u}_j(\s + n^3, n)\|_{L^1_\s \l^{r_2}_n} \\
		& \les \|u_1\|_{Z_0} \|u_2\|_{Y_0} \|u_3\|_{Y_0},
	\end{align*}
	by using H\"older and Minkowski's inequalities.
	
\end{proof}

\begin{remark}\rm
	
The change of variables from $\tau_j$ to the modulation $\s_j = \tau_j - n_j^3$, $j=1,2,3$, in \eqref{change_variables} is needed to guarantee that the quantity
\begin{align*}
\frac{1}{\jb{\tau -\ld}^{1-\dl}} 
&= \frac{1}{\jb{\tau-n^3 - \s_1 - \s_2 - \s_3 + \Phi(\nbar_{123})}^{1-\dl}}
\end{align*}
has an explicit dependence on the resonance relation $\Phi(\nbar_{123})$ and that when fixing its value, $\Phi(\nbar_{123}) = \mu$, there is no longer dependence on the variables $n_1, n_2,n_3$. Thus, one can consider the quantity inside the norm as a convolution operator in $\mu$, depending on $\tau$:
$$\sum_\mu \frac{1}{\jb{\tau - n^3 - \s_1 - \s_2 - \s_3 + \mu}^{1-\dl}} F(\mu, n_1,n_2,n_3).$$
This trick allows us to estimate the norm in $\tau$ and introduce a restriction on the value of the resonance relation. This strategy will be used throughout the paper.
\end{remark}

\section{Proof of Proposition~\ref{prop:w}}\ \label{sec:w}
In this section, we show the multilinear estimates needed to prove Proposition~\ref{prop:w} by a contraction mapping argument. In particular, we estimate the trilinear and quintilinear operators on the right-hand side of \eqref{eq:w}.

In Subsection~\ref{sec:cubic}, we focus on the cubic terms in \eqref{eq:w}, namely
\begin{equation}\label{cubic}
\begin{aligned}
\mathcal{IR}(u_1,u_2,u_3),  &&
\D  \mathcal{NR}_C(w_1,w_2,w_3) , && \D  \mathcal{NR}_D(w_1,w_2,w_3), \\
\B^{0}_A(w_1,u_2,u_3),  && \B^{1}_A(w_1,u_2,u_3), && \B^{2}_A(w_1,w_2,u_3),  &&\B^{3}_A(w_1,u_2,w_3),  \\
\B^{0}_B(w_1,w_2,u_3), &&
\B^{1}_B(w_1,w_2,u_3),  &&
\B^{2}_B(w_1,w_2,u_3), &&
\B^{3}_B(w_1,w_2,w_3),
\end{aligned}
\end{equation}
where $u_j \in \{u, \conj{u}\}$, $w_j\in\{w,\conj{w}\}$, $j=1,2,3$.

The quintic terms in \eqref{eq:w} arise from substituting a $u$ entry by an $\G_{\#}$-operator, for $\#\in\{A,B\}$. First, note that 
\begin{align*}
&\big|\Ft_{t,x} \big(\varphi_T \cdot \G_{\#} (u_1,u_2,u_3) \big) (\tau, n) \big|\\
&  \phantom{XXX}\les \sum_{\nbar_{123}\in\X_{\#}(n)} \intt_{\R} \frac{|n_1|}{\jb{\tau-\mu'}^{1-\theta} \jb{\Phi(\nbar_{123})}} \intt_{\mu' = \tau_1 + \tau_2 + \tau_3} \prod_{j=1}^3 |\ft{u}_j (\tau_j, n_j)| \ d\mu' , \\
&\big|\Ft_{t,x} \conj{\big(\varphi_T \cdot \G_{\#} (u_1,u_2,u_3) \big)} (\tau, n) \big| \\
&  \phantom{XXX} \les \sum_{\nbar_{123}\in\X_{\#}(n)} \intt_{\R} \frac{|n_1|}{\jb{\tau - \mu'}^{1-\theta} \jb{\Phi(\nbar_{123})}} \intt_{\mu' = \tau_1 + \tau_2 + \tau_3} \prod_{j=1}^3 |\ft{\conj{u}}_j (\tau_j, n_j)| \ d\mu',
\end{align*}
for any $0<\theta\ll1$. Since $\| \conj{u} \|_{X^{s,b}_{p,q}} = \| u\|_{X^{s,b}_{p,q}}$ for any choice of $s,b,p,q$, we will omit the contributions that depend on $\conj{\G_{\#}}$, as they can be estimated analogously. 

We start by calculating the space-time Fourier transform of the quintic contributions arising from $\mathcal{DN}$ terms. For example, for $*\in\{A,B,C,D\}$ and $\#\in\{A,B\}$, we have the following estimate
\begin{multline*}
\big|\Ft_{t,x} \big( \mathcal{DN}_* \big( \varphi_T \cdot \G_{\#}[u_1,u_2,u_3], u_4,u_5 \big) (\tau,n) \big| \\ 
\les \sum_{\substack{\nbar_{045} \in \X_*(n),\\ \nbar_{123} \in \X_{\#}(n_0)}} \intt_\R \intt_{\ld = \tau_1 + \ldots + \tau_5} \frac{|n_0n_1|}{\jb{\tau-\ld}^{1-\theta} \jb{\tau-n^3} \jb{\Phi(\nbar_{123})}} \prod_{j=1}^5 |\ft{u}_j(\tau_j, n_j)| \, d\ld, \label{DN_quintic}
\end{multline*}
for some $0<\theta<1$.
Similar estimates can be obtained for the contributions $\mathcal{DN}_*\big(u_1,\varphi_T \cdot \G_{\#}[u_2,u_3,u_4], u_5\big)$ and $\mathcal{DN}_*\big(u_1, u_2,\varphi_T \cdot \G_{\#}[u_3,u_4,u_5]\big)$, $*\in\{B,C,D\}$, $\#\in\{A,B\}$.
The main difficulty is controlling the spatial multiplier defined as follows
\begin{align*}\al (n, \nbar_{0\ldots 5}) =
\begin{cases}
\displaystyle \frac{|n_0n_1|}{|\Phi(\nbar_{123})|}  , &\text{if } \nbar_{045} \in \X_*(n), \ \nbar_{123} \in \X_{\#}(n_0) , \\
\displaystyle \frac{|n_1n_2|}{|\Phi(\nbar_{234})|}  , & \text{if }\nbar_{105} \in \X_*(n), \ \nbar_{234} \in \X_{\#}(n_0) , \\
\displaystyle \frac{|n_1n_3|}{|\Phi(\nbar_{345})|}  , & \text{if } \nbar_{120} \in \X_*(n), \ \nbar_{345} \in \X_{\#}(n_0) .\\
\end{cases} 
\end{align*}
We will refer to the frequencies in $\X_*(n)$ as the first generation of frequencies and those in $\X_{\#}(n_0)$ as the second when referring to the quintic terms.

In Subsection~\ref{sec:quintic}, we will estimate the contributions for which $\al(n, \nbar_{0\ldots5}) \les 1$, namely
\begin{equation}\label{quintic_terms}
\begin{aligned}
\D  \mathcal{NR}_{*} \big( \varphi_T \cdot \G_{A}[w_1,u_2,u_3],  u_4,u_5\big), && \D  \mathcal{NR}_{*} \big( \varphi_T \cdot \G_{B}[w_1,w_2,u_3],  u_4,u_5\big),  \\
\D  \mathcal{NR}_{\#} \big( w_1, \varphi_T \cdot \G_{A}[w_2,u_3,u_4],u_5\big), && \D  \mathcal{NR}_{\#} \big( w_1, \varphi_T \cdot \G_{B}[w_2,w_3,u_4],u_5\big), \\
\D  \mathcal{NR}_{D} \big( w_1, w_2, \varphi_T \cdot \G_{A}[w_3,u_4,u_5]\big), && \D  \mathcal{NR}_{D} \big( w_1, w_2, \varphi_T \cdot \G_{B}[w_3,w_4,u_5]\big),
\end{aligned}
\end{equation}
where $*\in\{A,B,C,D\}$, $\#\in\{B,C,D\}$ and $u_j\in\{u,\conj{u}\}$, $w_j \in \{u, \conj{w}\}$, $j=1, \ldots, 5$. 
The estimate for these contributions follows once we control $\mathcal{Q}(u_1,\ldots, u_5)$ defined by its space-time Fourier transform as follows
\begin{multline}
\Ft_{t,x}\big(\mathcal{Q}(u_1, \ldots, u_5)\big) (\tau,n) \\= \sum_{n=n_1+\ldots+n_5} \intt_\R \intt_{\ld = \tau_1 + \ldots + \tau_5} \frac{1}{\jb{\tau-\ld}^{1-\theta} \jb{\tau-n^3}} \prod_{j=1}^5 |\ft{u}_j(\tau_j, n_j)| \, d\ld .
\label{quintic}
\end{multline}
In Subsection~\ref{sec:quintic}, we establish an estimate for \eqref{quintic} under particular assumptions on the frequencies. Not all the contributions in \eqref{quintic_terms} will satisfy these additional assumptions, which will forces us to use the equation for $u$ once again, introducing new septic terms for which we also establish relevant estimates.

It remains to consider the following quintic contributions
\begin{equation}\label{quintic_bad}
\begin{aligned}
\D  \mathcal{NR}_C \big( w_1,w_2, \varphi_T \cdot \G_A[w_3,u_4,u_5] \big), && \D  \mathcal{NR}_C \big( w_1,w_2, \varphi_T \cdot \G_B[w_3,w_4,u_5] \big),\\
\B^{2}_A\big( w_1, \varphi_T \cdot \G_A[w_2,u_3,u_4], u_4 \big), && \B^{2}_A\big( w_1, \varphi_T \cdot \G_B[w_2,w_3,u_4], u_4 \big), \\
\B^{3}_A\big( w_1, u_2, \varphi_T \cdot \G_A[w_3,u_4,u_5]\big), && \B^{3}_A\big( w_1, u_2, \varphi_T \cdot \G_B[w_3,w_4,u_5]\big), \\
\B^{3}_B\big( w_1, w_2, \varphi_T \cdot \G_A[w_3,u_4,u_5]\big), &&
\B^{3}_B\big( w_1, w_2, \varphi_T \cdot \G_B[w_3,w_4,u_5]\big),
\end{aligned}
\end{equation}
where $u_j \in \{u, \conj{u}\}$, $w_j \in \{w, \conj{w}\}$, $j=1,\ldots, 5$. The $\mathcal{DN}_C$ contributions are not controlled by \eqref{quintic} and thus need a more refined approach. For the $\B^j_*$ contributions, not only does the $j$-th modulation play an important role, but also the largest modulation of the new functions in $\G_{\#}$. Thus these contributions require a more careful analysis, detailed in Subsection~\ref{sec:rem_quintic}.

\subsection{Cubic terms}\label{sec:cubic}
We start by estimating the cubic terms in \eqref{cubic}.

\begin{lemma}
	The following estimate holds
	\begin{align*}
	 \| \mathcal{I R} (u_1,u_2,u_3) \|_{Z_1} \les \|u_1\|_{Y_0} \|u_2\|_{Y_0} \|u_3\|_{Y_0}.
	\end{align*}
\end{lemma}
\begin{proof}
	Using the kernel estimate for $\D $ in \eqref{estimateK_new} and Young's inequality, we have
	\begin{align*}
	\| \mathcal{I R} (u_1,u_2,u_3) \|_{Z_1} 
	& \les \bigg\| \intt_\R \frac{1}{\jb{\tau-\ld}} \intt_{\ld = \tau_1 - \tau_2 + \tau_3}  \prod_{j=1}^3 \jb{n}^\frac12 |\ft{u}_j(\tau_j,n)| \, d\ld \bigg\|_{\l^p_n L^{q_0}_\tau}\\
	& \les \bigg\| \intt_{\tau= \tau_1 - \tau_2 + \tau_3} \prod_{j=1}^3 \jb{n}^\frac12 |\ft{u}_j (\tau_j,n)| \bigg\|_{\l^p_n L^{r_0}_\tau} ,
	\end{align*}
	for $\dl < \frac16$. 
	Applying H\"older's inequality gives
	\begin{align*}
	\| \mathcal{I R} (u_1,u_2,u_3) \|_{Z_1}
	& \les \sup_{\tau,n} J(\tau,n) 
	\bigg\| \prod_{j=1}^3 \| \jb{n}^\frac12 \jb{\tau-n^3}^\frac12 \ft{u}_j(\tau,n) \|_{L^{r_0}_\tau} \bigg\|_{\l^p_n},
	\end{align*}
	where 
	\begin{align*}
	J(\tau,n)^{r'_0} & =  \intt_{\R^2} \frac{d\tau_1 \, d\tau_2}{\jb{\tau_1 - n^3}^{\frac{r'_0}{2}} \jb{\tau_2 - n^3}^{\frac{r'_0}{2}} \jb{\tau  - \tau_1 +\tau_2 - n^3}^{\frac{r'_0}{2}}   } \les 1
	\end{align*}
	from Lemma~\ref{lm:convolution}.
	The result follows from H\"older's inequality.
\end{proof}

\begin{lemma}Let $*\in\{C,D\}$. Then, the following estimate holds
	\begin{align*}
	\| \D  \mathcal{NR}_*(u_1,u_2,u_3) \|_{Z_1} \les \|u_1\|_{Z_0} \|u_2\|_{Z_0} \|u_3\|_{Z_0}.
	\end{align*}
\end{lemma}
\begin{proof}
	Let $*\in\{C,D\}$, then $\nbar_{123}\in\X_*(n)$ implies that $\jb{n}^\frac12 |n_1| \les \jb{n_1}^\frac12 \jb{n_2}^\frac12 \jb{n_3}^\frac12$. Using \eqref{estimateK_new}, we have
	\begin{align*}
	&\| \D  \mathcal{NR}_*(u_1,u_2,u_3) \|_{Z_1} \\
	& \les \bigg\|  \intt_\R  \sum_{\substack{n=n_1+n_2+n_3,\\\nbar_{123}\in\X_*(n)}} \frac{1}{\jb{\tau-n^3}^\dl \jb{\tau-\ld}} \intt_{\ld = \tau_1 + \tau_2 + \tau_3} \prod_{j=1}^3 \jb{n_j}^\frac12 | \ft{u}_j(\tau_j, n_j)| d\ld \bigg\|_{\l^p_n L^{q_0}_\tau}.
	\end{align*}	
	Let $f_j(\s,n) = \jb{n}^\frac12 |\ft{u}_j (\s + n^3, n)|$, $j=1,2,3$, $\bar{s}=\s_1+\s_2+\s_3$ and proceed as in \eqref{change_variables}. Using Minkowski's and H\"older's inequalities gives
	\begin{align*}
	&\| \D  \mathcal{NR}_*(u_1,u_2,u_3) \|_{Z_1} \\
	&\les \intt_{\s_1,\s_2,\s_3} \bigg\|  \sum_{\substack{n=n_1+n_2+n_3,\\\nbar_{123}\in\X_*(n)}}  \frac{1}{\jb{\Phi(\nbar_{123})}^\dl \jb{\tau - n^3 -\bar{\s} + \Phi(\nbar_{123})}^{1-\dl}}  \prod_{j=1}^3 f_j(\s_j, n_j)  \bigg\|_{\l^p_n L^{q_0}_\tau} \\
	& \les \intt_{\s_1,\s_2,\s_3}\bigg\| \sum_{\mu} \frac{|\mu|^\eps}{\jb{\mu}^\dl \jb{\tau - n^3-\bar{\s} + \mu}^{1-\dl}}  \bigg( \sum_{\nbar_{123} \in \X_*^\mu (n)} \prod_{j=1}^3 |f_j(\s_j , n_j)|^p \bigg)^\frac1p \bigg\|_{\l^p_n L^{q_0}_\tau},
	\end{align*}
	since from the standard divisor counting estimate, we have that $|\X_*^\mu(n)| \les_\eps |\mu|^\eps$, for any $\eps>0$. Choosing $\eps \leq \dl$ and applying Schur's test with $1 + \frac{1}{q_0} = \frac1p + \frac1q$, we obtain
	\begin{align*}
	\| \D  \mathcal{NR}_*(u_1,u_2,u_3) \|_{Z_1} 
	& \les \intt_{\s_1,\s_2,\s_3}\bigg\| \bigg( \sum_{\nbar_{123} \in \X_*^\mu(n)} \prod_{j=1}^3 |f_j(\s_j , n_j)|^p \bigg)^\frac1p \bigg\|_{\l^p_n \l^{p}_\mu}  
	\les \prod_{j=1}^3 \|f_j(\s, n) \|_{L^1_\s,\l^p_n},
	\end{align*}
	for $\dl< \frac{1}{5p}$. 
	Consequently, using H\"older's and Minkowski's inequalities, it follows that
	\begin{align*}
	\| \D  \mathcal{NR}_*(u_1,u_2,u_3) \|_{Z_1} 
	& \les \prod_{j=1}^3 \| \jb{n}^\frac12 \jb{\s}^{1-4\dl+} \ft{u}_j(\s + n^3, n) \|_{L^{q_0}_\s \l^p_n} 
	 \les \prod_{j=1}^3 \|u_j\|_{Z_0},
	\end{align*}
	for $\dl < \frac{1}{4p}$. 
\end{proof}

\begin{remark}\rm \label{rm:w}
	(i) The terms $\D  \mathcal{NR}_A$, $\D  \mathcal{NR}_B$ cannot be estimated in a similar manner because $\jb{n}^\frac12 |n_1|$ is not controlled by $(\jb{n_1} \jb{n_2} \jb{n_3})^\frac12$. This motivates the splitting of the Duhamel operator for these contributions, since we require a convolution operator whose kernel has a negative power of the resonance relation $\Phi$, needed to control the loss of derivative from the nonlinearity, without using the largest modulation.
	
	\smallskip
	\noi(ii) Consider the estimate 
	$$\big\| \D  \mathcal{NR}_D(u_1,u_2,u_3) \big\|_{X^{\frac12, b}_{p,q}} \les \prod_{j=1}^3 \|u_j\|_{X^{\frac12, b}_{p,q}},$$
	for some $b\geq 0$, $2\leq q < \infty$. The region $\X_D(n)$ includes the case when $|n_1| \sim |n_2| \sim |n_3|$, $\max\limits_{j=1,2,3}|n_j| \les |\Phi(\nbar_{123})| \ll \max\limits_{j=1,2,3} |n_j|^2$. When attempting to show the above estimate under the nearly-resonant assumption, we must impose the conditions
	\begin{align*}
	\max\bigg(1 - \frac{1}{2q},  1 + \frac1q - \frac1p \bigg)&< b <1,
	\end{align*}
	which motivate our choice of $b=1-$ and $q=\infty-$ for the definition of the $Z_0$ space.
\end{remark}

\begin{lemma}
	Let $*\in\{A,B\}$. The following estimate holds
	\begin{align*}
	\| \B_*^{0}(u_1,u_2,u_3) \|_{Z_1} \les \|u_1\|_{Y_0} \|u_2\|_{Y_0} \|u_3\|_{Y_0}.
	\end{align*}
\end{lemma}
\begin{proof}
	Choosing $\al=4\dl$ in the kernel estimate \eqref{estimateK0}, gives
	\begin{align*}
	\| \B_*^{0}(u_1,u_2,u_3) \|_{Z_1} 
	& \les \bigg\|\sum_{\substack{n=n_1+n_2+n_3,\\ \nbar_{123}\in\X_*(n)}}   \intt_{\tau_1,\tau_2,\tau_3} \frac{\jb{n}^\frac12 |n_1| }{\jb{\tau-n^3}^{5 \dl} \jb{\Phi(\nbar_{123})}^{1-4\dl } } \prod_{j=1}^3 |\ft{u}_j (\tau_j, n_j)| \bigg\|_{\l^p_n L^{q_0}_{\tau}}\\
	&\les \bigg\|\sum_{\substack{n=n_1+n_2+n_3,\\ \nbar_{123}\in\X_*(n)}}   \frac{\jb{n}^\frac12 |n_1| }{ \jb{\Phi(\nbar_{123})}^{1-4\dl} } \prod_{j=1}^3 \|\ft{u}_j ( n_j)\|_{L^1_\tau} \bigg\|_{\l^p_n },
	\end{align*}
	by applying Minkowski's inequality in the last step and integrating in $\tau$. For $\nbar_{123}\in\X_*(n)$, we have $|\Phi(\nbar_{123})| \sim \max\limits_{j=1,2,3}|n_j|^2 \min\limits_{l=1,2,3}|n-n_l|$, which implies
	\begin{align*}
	\frac{\jb{n}^\frac12 |n_1|}{\jb{\Phi(\nbar_{123})}^{1-4\dl}} \les \frac{\jb{n_1}^\frac12}{\max\limits_{j=1,2,3}\jb{n_j}^{1 -8\dl } \min\limits_{l=1,2,3} \jb{n-n_l}^\frac12}.
	\end{align*}
	Applying H\"older and Minkowski's inequalities, it follows that
	\begin{align*}
	\| \B_*^{0}(u_1,u_2,u_3) \|_{Z_1} 
	&\les \big( \sup_n J(n)^{\frac{1}{p'}} \big) \| \jb{n}^\frac12 \ft{u}_1\|_{\l^p_n L^1_\tau} \prod_{j=2}^3 \|\ft{u}_j\|_{\l^p_n L^1_\tau},
	\end{align*}
	where $J(n)$ is defined as follows
	\begin{align*}
	J(n):= \sum_{n=n_1+n_2+n_3} \frac{1}{\max\limits_{j=1,2,3} \jb{n_j}^{(1-8\dl)p'} \min\limits_{l=1,2,3}\jb{n-n_l}^{(1-4\dl)p'}  }.
	\end{align*}
	Let $j,l\in\{1,2,3\}$ denote the indices at which the maximum and minimum in the definition of $J(n)$ are attained, respectively. If $j = l$, we can use the fact that $\jb{n_j} \ges \jb{n_i}$ for $i\in\{1,2,3\}\setminus\{j\}$ and sum in $n_i, n_j$. If $j\neq l$, we sum in $n_j, n_l$. Thus, $J(n) \les 1$ uniformly in $n$ for $\dl < \frac{1}{8p}$. The intended estimate, follows from H\"older's inequality.

\end{proof}

\begin{lemma}
	Let $*\in\{A,B\}$. Then, the following estimates hold for $j=1,2,3$
	\begin{align*}
	\| \B_*^{j}(u_1,u_2,u_3)\|_{Z_1} \les \|u_j\|_{Z_0} \prod_{\substack{k=1\\ k\neq j}}^3 \|u_k\|_{Y_0}.
	\end{align*}

\end{lemma}

\begin{proof}
	We will only show the estimate for $j=1$, as the remaining estimates follow an analogous proof. Let $*\in\{A,B\}$. From \eqref{KX+} with $1-\be= b_0 - \dl$ and for $\nbar_{123} \in \X_*(n)$ we have
	\begin{equation*}
	\jb{n}^\frac12 |n_1| \jb{\tau-n^3}^{b_1}\big|K_{+}\big(\tau-n^3, \ld - n^3, \Phi(\nbar_{123})\big)\big| \les \frac{\jb{n_1}^{\frac12} \jb{\tau_1 - n_1^3}^{1-3\dl}}{\jb{\tau-\ld} \jb{\Phi(\nbar_{123})}^{\frac12-3\dl} \jb{n-n_l}^{\frac12}},
	\end{equation*}
	where $|n-n_l| = \min_{j=1,2,3}|n-n_j|$.
	Let $f(\tau,n) = \jb{n}^\frac12 \jb{\tau_1 - n_1^3}^{b_0 - \dl}|\ft{u}_1(\tau,n)|$. Then, using Minkowski's and Young's inequalities, we have
	\begin{align*}
	&\big\| \B^{1}_* (u_1,u_2,u_3) \big\|_{Z_1}\\
	& \les \bigg\| \sum_{\substack{n=n_1+n_2+n_3,\\ \nbar_{123} \in \X_*(n)}} \intt_{\tau_1,\tau_2,\tau_3} \frac{1}{\jb{\Phi(\nbar_{123})}^{\frac12-3\dl} \jb{n-n_l}^{\frac12}} \bigg(\frac{1}{\jb{\cdot}} \ast f(\cdot, n_1) \ast_{j=2,3} |\ft{u}_j(\cdot, n_j)|\bigg)(\tau) \bigg\|_{\l^p_n L^{q_0}_\tau} \\
	& \les \bigg\| \sum_{\substack{n=n_1+n_2+n_3,\\ \nbar_{123} \in \X_*(n)}} \frac{1}{\jb{\Phi(\nbar_{123})}^{\frac12 - 3\dl}\jb{n-n_l}^{\frac12}} \|f(n_1)\|_{L^{q_1}_\tau} \prod_{j=2}^3 \|\ft{u}_j (n_j) \|_{L^1_\ld} \bigg\|_{\l^p_n}.
	\end{align*}
	Using H\"older's inequality, we obtain
	\begin{align*}
	\big\| \B^{1}_* (u_1,u_2,u_3) \big\|_{Z_1} \les \bigg(\sup_n J(n)\bigg)^{\frac{1}{p'}} \|f\|_{\l^p_n L^{q_1}_\tau} \prod_{j=2}^3 \|u_j\|_{\l^p_n L^1_\tau} ,
	\end{align*}
	where 
	\begin{align*}
	J(n) = \sum_{n=n_1+n_2+n_3} \frac{1}{\jb{n_{\max}}^{p'(1-6\dl)} \jb{n-n_l}^{p'(1-3\dl)}} \les \sum_{n_i, n_l} \frac{1}{\jb{n_i}^{p'(1-6\dl)} \jb{n-n_l}^{p'(1-3\dl)}} \les 1,
	\end{align*}
	for some distinct $n_i, n_l \in \{n_1,n_2,n_3\}$ and $\dl < \frac{1}{6p}$. The intended estimate follows from applying H\"older's inequality.

\end{proof}

\subsection{Standard quintic terms}\label{sec:quintic}
In this section, we focus on estimating the quintic terms in \eqref{quintic_terms}. 

We know how to control the spatial multiplier in the space-time Fourier transform of the terms in \eqref{quintic_terms}, but it remains to control the weight $\jb{n}^\frac12$ from the norm. Before doing so, we must take into account the new `resonances' introduced by using second iteration. In particular, in the sum $n=n_1+\ldots + n_5$, we can have $n_i+n_j = 0$ for distinct $i,j \in \{1,\ldots,5\}$ as long as they do not belong to the same generation of frequencies. For the estimates to hold, we need the largest frequency to not be in a pairing and correspond to a $w$ term. Otherwise, we will use the equation for $u$ \eqref{eq:u}, which introduces new septic terms.

Looking at $\mathcal{Q}$ in \eqref{quintic} in more detail, note that the sum in \eqref{quintic} over $n=n_1+\ldots + n_5$ does not exclude all resonances, i.e., we can have $n_i + n_j =0$ for distinct $i, j \in \{1,\ldots, 5\}$. If this holds, we say that $(i,j)$ is a pairing.

Since the space-time Fourier transform of the terms in \eqref{quintic_terms} is controlled by \eqref{quintic}, we want to take advantage of this common form and show a general estimate for \eqref{quintic}. This estimate (Proposition~\ref{standard_quintic}) requires one of the following conditions:
\begin{enumerate}
	\item There are no pairings in $(n_1, \ldots, n_5)$ and the largest frequency corresponds to a function in $Z_0$;
	\item There is one pairing $(i,j)$ and the largest frequency in $\{|n_k|: \ 1\leq k \leq 5, \ k\neq i,j \}$ corresponds to a function in $Z_0$;
	\item There are two pairings and the remaining frequency corresponds to a functions in $Z_0$.
\end{enumerate}
Note that if (1), (2) or (3) hold, we can always use the largest frequency which is not in a pairing to control the spatial weight from the norm $\jb{n}^\frac12$. If the contributions do not satisfy any of the above conditions, then the largest frequency that is not in a pairing corresponds to a function $u$ and we want to use the equation for $u$ again. This leads to one quintic term that satisfies the assumptions above and two septic terms, which are easily estimated.

To further clarify, let $\mathcal{Q}'(u_1, \ldots, u_5)$ denote a contribution in \eqref{quintic_terms}, $u_j \in\{u, \conj{u}, w, \conj{w}\}$. Let $n_j$ correspond to the spatial Fourier variable of $\ft{u}_j$, $j=1,\ldots,5$. If $n_1$ is the largest frequency that is not in a pairing and $u_1 \in \{w,\conj{w}\}$, then we keep the contribution as is. Otherwise, $u_1 \in \{u,\conj{u}\}$ and we will use the equation \eqref{eq:u} to replace the first entry in $\mathcal{Q}'$. For simplicity, assume that $u_1 = u$, then we have
\begin{align*}
\mathcal{Q}' (u_1, \ldots, u_5) &= \mathcal{Q}' (w, u_2, \ldots, u_5) \\
& +  \mathcal{Q}' \big(\varphi_T\cdot  \G_{A,\geq}[w, \conj{u}, u], u_2, \ldots, u_5 \big) +  \mathcal{Q}'\big(\varphi_T\cdot \G_{A,>}[w, u, \conj{u}], u_2 \ldots, u_5\big) \\
& + \mathcal{Q}' \big(\varphi_T\cdot \G_{B,\geq}[w, \conj{w}, u], u_2, \ldots, u_5\big) + \mathcal{Q}' \big(\varphi_T\cdot \G_{B,>}[w, w, \conj{u}], u_2, \ldots, u_5\big).
\end{align*}
By carefully examining the frequencies and pairings of the terms in \eqref{quintic_terms} and applying the above modification, we obtain the final equation for $w$. We have decided to not include this full equation for $w$ due to its length.

All the resulting quintic and septic terms arising from \eqref{quintic_terms}, can be estimated by the two following propositions.

\begin{proposition}\label{standard_quintic}
	Let $\mathcal{Q}$ as defined in \eqref{quintic} where the first factor has the largest spatial Fourier frequency which is not in a pairing.
	Then, the following estimate holds
	\begin{align*}
	\big\| \mathcal{Q}(u_1, \ldots, u_5) \big\|_{Z_1} \les \|u_1\|_{Z_0} \prod_{j=2}^5 \|u_j\|_{Y_0}.
	\end{align*}
\end{proposition}
\begin{proof}\
	
	\noi\underline{\textbf{Case 1: no pairing}}
	
	\noi Let $\P_{N_j}$ denote the Dirichlet projection onto $\jb{n_j} \sim N_j$, $j=1,\ldots, 5$. Since there is no pairing we have $|n| \les |n_1|$, therefore using Minkowski's inequality gives 
	\begin{multline*}
	\| \mathcal{Q}(u_1, \ldots, u_5) \|_{Z_1} 
	 \les \sum_{N_2, \ldots, N_5} \intt_{\tau_1, \ldots, \tau_5} \bigg\| \sum_{n=n_1 + \ldots + n_5}  \frac{1}{\jb{\tau - \tau_1 - \ldots - \tau_5}^{1-\theta}}\\\times  \jb{n_1}^\frac12 |\ft{u}_1 (\tau_1,n_1)| \prod_{j=2}^5 |\ft{\P_{N_j} u}(\tau_j,n_j)| \bigg\|_{\l^p_n L^{q_0}_\tau}.
	\end{multline*}
	Using the change of variables $\s_j = \tau_j  - n_j^3$, $j=1,\ldots, 5$, and Schur's test, we get
	\begin{align*}
	\| \mathcal{Q} (u_1, \ldots, u_5) \|_{Z_1} & \les \sum_{N_2, \ldots, N_5} \intt_{\s_1,\ldots, \s_5} \bigg\| \sum_{n=n_1 + \ldots + n_5}  \frac{1}{\jb{\tau - n^3 - \s_1 - \ldots - \s_5 + \Psi(n, \nbar_{1\ldots 5})}} \\
	&\phantom{XX} \times \jb{n_1}^\frac12 |\ft{u}_1 (\s_1 + n_1^3,n_1)| \prod_{j=2}^5 |\ft{\P_{N_j} u}(\s_j + n_j^3,n_j)| \bigg\|_{\l^p_n L^{q_0}_\tau} \ d\s_1 \cdots d\s_5 \\
	& \les \sum_{N_2, \ldots, N_5} \intt_{\s_1, \ldots, \s_5} \bigg\| \sum_{\substack{n=n_1 + \ldots + n_5,\\ \Psi(n,\nbar_{1\ldots 5})= \mu}} \prod_{j=1}^5 f_j(\s_j, n_j) \bigg\|_{\l^p_n \l^{q_1}_\mu},
	\end{align*}
	where $\Psi(n, \nbar_{1\ldots 5}) = n^3 - n_1^3 - \ldots - n_5^3$, $f_1(\s,n) = \jb{n}^\frac12 |\ft{u}_1(\s+n^3,n)|$ and $f_j(\s,n) = |\ft{P_{N_j} u_j}(\s+n^3,n)|$, $j=2,\ldots,5$. 
	Note that we can trivially restrict $\mu$ to the following region
	\begin{multline*}
	A(n,N_2,\ldots, N_5) = \big\{ \mu\in\Z: \ \mu = n - (n-n_2 - \ldots - n_5)^3 - n_2^3 - \ldots - n_5^3 , \\ |n_j|\sim N_j, \ j=2,\ldots,5 \big\},
	\end{multline*}
	which satisfies $|A(n,N_2,\ldots,N_5)| \les N^4_2$ for fixed $n$.
	Thus, it follows from H\"older's inequality in $\mu$ that 
	\begin{align}
	 \| \mathcal{Q} (u_1, \ldots, u_5)\|_{Z_1} 
	& \sim \sum_{N_2, \ldots, N_5} \intt_{\s_1, \ldots, \s_5} \bigg\| \1_{\mu \in A(n,N_2, \ldots, N_5)}\sum_{\substack{n=n_1 + \ldots + n_5,\\ \Psi(n,\nbar_{1\ldots 5})= \mu }}  \prod_{j=1}^5 f_j(\s_j,n_j) \bigg\|_{\l^p_n \l^{q_1}_\mu} \nonumber \\
	& \les  \sum_{N_2, \ldots, N_5} N_2^{\frac{4}{q_1}} \intt_{\s_1, \ldots, \s_5} \bigg\| \sum_{\substack{n=n_1 + \ldots + n_5,\\ \Psi(n,\nbar_{1\ldots 5})= \mu }}\prod_{j=1}^5 f_j(\s_j,n_j)  \bigg\|_{\l^p_n \l^\infty_\mu}. \label{aux_quintic}
	\end{align}
	
	Now, we consider two distinct cases depending on the size of the frequencies. 
	
	\vspace{2mm}
	\noi\underline{\textbf{Subcase 1.1: }$N_3 \geq N_2^{4\sqrt{\dl}}$}
	
	\noi Using Cauchy's inequality with $\al>0$, omitting the time dependence, we have
	\begin{align*}
	 \sum_{\substack{n=n_1 + \ldots + n_5,\\ \Psi(n,\nbar_{1\ldots 5})= \mu}} \prod_{j=1}^5 f_j (n_j) 
	& \les \sum_{\substack{n=n_1 + \ldots + n_5,\\ \Psi(n,\nbar_{1\ldots 5})= \mu}} f_1(n_1) \big(\al|f_2(n_2) f_3(n_3)|^2 + \al^{-1}|f_4(n_4)f_5(n_5)|^2 \big) \\
	&\les \sum_{n_2,n_3} \sum_{(n_1,n_4, n_5)\in B(n,n_2,n_3,\mu)} \al f_1(n_1) |f_2(n_2) f_3(n_3)|^2 \\
	& \phantom{XXXX}+ \sum_{n_4,n_5} \sum_{(n_1, n_2, n_3)\in B(n,n_4,n_5,\mu)} \al^{-1} f_1(n_1) |f_4(n_4) f_5(n_5)|^2,
	\end{align*}
	where
	\begin{multline*}
	B(k,k_1,k_2,\mu) := \big\{ (n_1,n_2,n_3)\in\Z^3: \ n_1 + n_2 + n_3 = k - k_1 - k_2 =: l , \\ 3(n_2+n_3) (l - n_2)(l - n_3) = \mu - k^3 + k_1^3 + k_2^3 + l^3 \big\}.
	\end{multline*}
	Taking a supremum in $n_1$, we obtain
	\begin{multline*}
	\sum_{\substack{n=n_1 + \ldots + n_5,\\ \Psi(n,\nbar_{1\ldots 5})= \mu}} \prod_{j=1}^5 f_j (n_j) \les \al \sup_{|n-n_1| \les N_2} f_1(n_1) \sum_{n_2,n_3} |B(n,n_2,n_3, \mu)| \cdot |f_2(n_2)f_3(n_3)|^2 \\
	+ \al^{-1} \sup_{|n-n_1| \les N_2} f_1(n_1) \sum_{n_4,n_5} |B(n,n_4,n_5, \mu)| \cdot |f_4(n_4)f_5(n_5)|^2.
	\end{multline*}
	In order to estimate $|B(n, n_2,n_3, \mu)|, |B(n, n_4,n_5, \mu)|$, we use Lemma~\ref{lm:divisor}. For the first one, to count the choices of $(n_1,n_4,n_5)$ it suffices to count the number of divisors $l-n_4, l-n_5$, where $l=n-n_2-n_3$, of $\tilde{\Psi}:=3(n_4+n+5)(n-n_2-n_3-n_4)(n-n_2-n_3-n_5) = 3(n_4+n_5)(l-n_4)(l-n_5)$. If $|n|\gg |n_2|$, then
	\begin{align*}
	|\tilde{\Psi}| \sim |(n_4 + n_5)(n-n_2-n_3-n_4)(n-n_2-n_3-n_5)| \les |n|^3 \implies |\tilde{\Psi}|^\eps \leq |\tilde{\Psi}|^\frac13 \les |n|,
	\end{align*}
	for any $\eps>0$.
	Otherwise, $|n| \les |n_2|$ and we have
	\begin{align*}
	|\tilde{\Psi}| \sim |(n_4 + n_5)(n-n_2-n_3-n_4)(n-n_2-n_3-n_5)| \les |n_2|^3 \implies |\tilde{\Psi}|^\eps \leq |\tilde{\Psi}|^\frac13 \les |n_2|,
	\end{align*}
	for any $\eps>0$. Applying the lemma, the number of divisors $d_4 = n-n_2-n_3-n_4 $, $d_5 = n-n_2-n_3-n_5$ satisfying 
	\begin{equation*}
	\begin{cases}
	|d_j - n| = |n_2+n_3+n_j| \les N_2 ,& \text{if } |n| \gg |n_2|,\\
	|d_j - n_2|  = |n - n_3 - n_j| \les N_2, & \text{if } |n| \les |n_2|,
	\end{cases}
	\end{equation*}
	is bounded by $N_2^\eps$, for $j=4,5$. Thus, $|B(n,n_2,n_3, \mu)| \les N_2^\eps$, for any $\eps>0$. An analogous approach gives $|B(n,n_4,n_5, \mu)| \les N_2^\eps$.
	Consequently, we have
	\begin{align*}
	\sum_{\substack{n=n_1 + \ldots + n_5,\\ \Psi(n,\nbar_{1\ldots 5})= \mu}} \prod_{j=1}^5 f_j (n_j) 
	& \les N_2^\eps \sup_{|n-n_1| \les N_2} f_1(n_1) \prod_{j=2}^5 \|f_j\|_{\l^2_n},
	\end{align*}
	by choosing $\al = (\|f_2\|_{\l^2_n} \|f_3\|_{\l^2_n})^{-1}\|f_4\|_{\l^2_n} \|f_5\|_{\l^2_n}$.
	Looking at \eqref{aux_quintic}, since $|n-n_1| \les N_2$, taking a supremum in $n_1$ gives
	\begin{align*}
	\|\mathcal{Q}(u_1, \ldots, u_5)\|_{Z_1} 
	& \les \sum_{N_2,\ldots, N_5} N_2^{\frac{4}{q_1} + \eps} \intt_{\s_1, \ldots, \s_5}\bigg\| \sup_{|n-n_1|\les N_2} f_1(\s_1,n_1) \bigg\|_{l^p_n} \prod_{j=2}^5 \|f_j(\s_j)\|_{ \l^2_n} \\
	& \les \sum_{N_2,\ldots, N_5} N_2^{\frac{4}{q_1} + \eps}  \intt_{\s_1, \ldots, \s_5} \bigg(\sum_{n_1} |f_1(\s_1,n_1)|^p \sum_{|n-n_1|\les N_2} 1 \bigg)^{\frac1p} \prod_{j=2}^5 \|f_j(\s_j)\|_{ \l^2_n} \\
	& \les \sum_{N_2, \ldots, N_5} N_2^{\frac{4}{q_1} + \frac1p + \eps} \|f_1\|_{L^1_\s \l^p_n} \prod_{j=2}^5 \|f_j\|_{L^1_\s \l^2_n}.
	\end{align*}
	Using H\"older's and Minkowki's inequalities, we have
	\begin{align*}
	\|\mathcal{Q}(u_1, \ldots, u_5)\|_{Z_1} 
	& \les \sum_{N_2,\ldots, N_5} N_2^{\frac{4}{q_1} + \frac1p+ \eps} (N_2 N_3 N_4 N_5)^{ \dl - \frac1p +} \| u_1 \|_{Z_0} \prod_{j=2}^5 \|u_j \|_{Y_0} .
	\end{align*}
	It only remains to sum in the dyadic numbers $N_j$. Using the fact that $N_3 \geq N_2^{4\sqrt{\dl}}$, we have
	\begin{align*}
	\sum_{N_2, \ldots, N_5} N_2^{\frac{4}{q_1} +  \dl + } ( N_3 N_4 N_5)^{\dl - \frac1p +} 
	& \les \sum_{N_2, \ldots, N_5} (N_2 N_3 N_4 N_5)^{-\dl} N_3^{2\dl+ 5\sqrt{\dl} - \frac1p +} (N_4 N_5 )^{2\dl - \frac1p+} \les 1
	\end{align*}
	for $\dl < \frac{1}{2p}$ and $2\dl + 5\sqrt{\dl} - \frac1p<0 \implies 0< \sqrt{\dl} < -\frac54 + \sqrt{\big(\frac{5}{4}\big)^2 + \frac{1}{2p}}$.
	
	\vspace{2mm}
	\noi\underline{\textbf{Subcase 1.2: }$N_3 \leq N_2^{4\sqrt{\dl}}$}
	
	\noi Using Cauchy-Schwarz inequality, we have
	\begin{align*}
	& \sum_{\substack{n=n_1 + \ldots + n_5,\\ \Psi(n,\nbar_{1\ldots 5})= \mu}} \prod_{j=1}^5 f_j (n_j) \\
	& = \sum_{n_3, n_4, n_5} \sum_{n_1 \in C(n, n_3,n_4,n_5, \mu)} f_1(n_1) f_2(n-n_1 - n_3 - n_4 - n_5) f_3 (n_3) f_4(n_4) f_5(n_5)\\
	& \les N_2^\eps\sum_{n_3, n_4, n_5} f_3 (n_3) f_4(n_4) f_5(n_5) \big(\sup_{n_1} f_1(n_1) f_2(n-n_1 - n_3 - n_4 - n_5) \big)
	\end{align*}
	where 
	\begin{multline*}
	C(n, n_3,n_4,n_5, \mu) : = \{n_1 \in \Z: l:= n_3 + n_4 + n_5, |l| \les |n-n_1-l| \les |n_1|, \\ |n-n_1 -l|\les N_2, \
	3(n-l)(n_1 + l)(n-n_1) = \mu - l^3 + n_3^3 + n_4^3 + n_5^3\},
	\end{multline*}
	for which $|C(n,n_3,n_4,n_5, \mu)| \les N_2^\eps$ for any $\eps>0$ from Lemma~\ref{lm:divisor}. Note that if $|n| \gg |n-n_1 - l|$, then $\tilde{\Psi}:= 3(n-l)(n_1 + l)(n-n_1)$ satisfies $|\tilde{\Psi}| \les |n|^3$ and $|\tilde{\Psi}|^\eps \leq |\tilde{\Psi}|^\frac13 \les |n|$ for any $0<\eps <\frac13$. Counting the number of $n_1$ is equivalent to counting the number of divisors $d = n_1 + l$. Since $|d - n| = |n-n_1-l| \les N_2$, from Lemma~\ref{lm:divisor}, there exist at most $N_2^\eps$ values for $n_1$. If $|n| \les |n-n_1-l|$, then $|\tilde{\Psi}| \les |n-n_1-l|^3\les N_2^3$, so by the standard divisor counting lemma, there are at most $N_2^\eps$ divisors $n_1+l$. Consequently, $|C(n,n_3,n_4,n_5, \mu)| \les N_2^\eps$, for any $0<\eps<\frac13$.
	
	Minkowski's inequality gives the following
	\begin{align*}
	\bigg\|  \sum_{\substack{n=n_1 + \ldots + n_5,\\ \Psi(n,\nbar_{1\ldots 5})= \mu }} \prod_{j=1}^5 f_j (n_j) \bigg\|_{\l^p_n \l^\infty_\mu} 
	& \les N_2^\eps \|f_1\|_{\l^p_n} \|f_2\|_{\l^p_n} \prod_{j=3}^5 \|f_j\|_{\l^1_n}.
	\end{align*}
	
	Applying the previous estimate to \eqref{aux_quintic} and H\"older's inequality give
	\begin{align*}
	\|\mathcal{Q}(u_1, \ldots, u_5)\|_{Z_1} & \les \sum_{N_2, \ldots, N_5} N_2^{\frac{4}{q_1} + \eps} \|f_1\|_{L^1_\s \l^p_n} \|f_2\|_{L^1_\s \l^p_n} \prod_{j=3}^5 \|f_j\|_{\l^1_n L^1_\ld} \\
	& \les \sum_{N_2, \ldots, N_5} N_2^{\frac{4}{q_1} + \frac{1}{r_0} - \frac1p - \frac12+} (N_3 N_4 N_5)^{\frac12 - \frac1p +} \|u_1\|_{Z_0} \prod_{j=2}^5 \|u_j\|_{Y_0}.
	\end{align*}
	It only remains to sum in the dyadics $N_j$:
	\begin{align*}
	\sum_{N_2, \ldots, N_5} N_2^{\frac{4}{q_1} + \dl - \frac1p+ } (N_3 N_4 N_5)^{\frac12 - \frac1p+} 
	& \les \sum_{N_2, \ldots, N_5} (N_3 N_4 N_5)^{-\dl } N_2^{\frac{4}{q_1} + \dl - \frac1p + 12\sqrt{\dl}(\frac12 - \frac1p + 3 \sqrt{\dl}) } \les 1
	\end{align*}
	if $\frac{4}{q_1} + \dl - \frac1p + 12\sqrt{\dl}(\frac12 - \frac1p + 3 \sqrt{\dl})<0 \implies 0< \sqrt{\dl} < - \frac{6}{21}\big(\frac12 - \frac1p\big) + \sqrt{\frac{6^2}{21^2}\big(\frac12 - \frac1p\big)^2 + \frac{21}{p}}$.
	
	\vspace{5mm}
	\noi\textbf{\underline{Case 2: one pairing $(4,5)$}}
	
	\noi In this case, we have $n=n_1 + n_2 + n_3$, $n_4 + n_5 =0$. Let $f_1(\s,n) = \jb{n}^\frac12 |\ft{u}_1(\s+n^3,n)|$ and $f_j(\s,n) = |\ft{\P_{N_j} u_j} (\s+n^3, n)|$, $j=2,3$.
	Using Cauchy-Schwarz inequality in $n_4$  and proceeding as in Case 1 gives the following
	\begin{align}
	& \|\mathcal{Q}(u_1, \ldots, u_5)\|_{Z_1}\nonumber\\
	& \les \sum_{N_2,N_3} \intt_{\s_1,\s_2,\s_3} \Bigg\| \sum_{\substack{n=n_1+n_2+n_3\\ \Phi(\nbar_{123}) = \mu}} \prod_{j=1}^3 f_j(\s_j,n_j)  \Bigg\|_{\l^p_n \l^{q_1}_\mu} \prod_{k=4}^5\| \ft{u}_k \|_{L^1_\tau \l^2_n} \nonumber\\
	&  \les \sum_{N_2,N_3} \intt_{\s_1, \s_2, \s_3}  \Bigg\| \1_{\mu \in A(n,N_2,N_3)} \sum_{\substack{n=n_1+n_2+n_3,\\ \Phi(\nbar_{123}) = \mu}}  \prod_{j=1}^3 f_j(\s_j, n_j) \Bigg\|_{\l^p_n }
	 \prod_{k=4}^5\| \ft{u}_k \|_{L^1_\tau \l^2_n}  \nonumber\\
	 & \les  \sum_{N_2,N_3} N_2^{\frac{2}{q_1}} \intt_{\s_1,\s_2,\s_3}  \Bigg\| \sum_{\substack{n=n_1+n_2+n_3,\\ \Phi(\nbar_{123}) = \mu}} \prod_{j=1}^3 f_j(\s_j, n_j) \Bigg\|_{\l^p_n \l^\infty_\mu}
	 \prod_{k=4}^5\| \ft{u}_k \|_{L^1_\tau \l^2_n}, \label{aux2_quintic}
	\end{align}
	where 
	\begin{multline*}
	A(n,N_2,N_3) := \big\{ \mu \in \Z: \ \mu = n^3 - (n-n_2-n_3)^3 - n_2^3 - n_3^3, \ |n_j| \sim N_j, j=2,3 \big\},
	\end{multline*}
	which satisfies $|A(n,N_2,N_3)| \les N_2^2$ uniformly in $n$.
	
	As before, we consider two subcases.
	
	\vspace{2mm}
	\noi\underline{\textbf{Subcase 2.1: }$N_2^{4\sqrt{\dl}} \leq N_3$}\\
	\noi Focusing on the inner sum, we apply Cauchy's inequality, with $\al>0$, to obtain the following
	\begin{multline*}
	 \sum_{\substack{n=n_1+n_2+n_3,\\ \Phi(\nbar_{123}) = \mu}}  f_1( n_1) f_2( n_2) f_3(n_3)\\
	\les \sum_{n_2} \sum_{n_1 \in B(n,n_2,\mu)} \al f_1(n_1) |f_2(n_2)|^2 +  \sum_{n_3} \sum_{n_1 \in B(n,n_3,\mu)} \al^{-1} f_1(n_1) |f_3( n_3)|^2 ,
	\end{multline*}
	where 
	\begin{equation*}
	B(n,n_j, \mu) = \big\{ n_1\in \Z : \ 3(n-n_1)(n-n_j)(n_1+n_j) = \mu \big\}, \ j=2,3.
	\end{equation*}
	Note that $|B(n,n_j,\mu)| \leq 2$ because the given equation is quadratic in $n_1$. Thus, taking a supremum in $n_1$ and using the fact that $|n-n_1|\les N_2$, we get
	\begin{align*}
	 \sum_{\substack{n=n_1+n_2+n_3,\\ \Phi(\nbar_{123}) = \mu}}  \prod_{j=1}^3 f_j(\s_j, n_j) 
	& \les  \sup_{|n-n_1|\les N_2} f_1(\s_1,n_1) \big(\al \| f_2(\s_2)\|_{\l^2_n}^2 + \al^{-1} \|f_3(\s_3)\|_{\l^2_n}^2 \big) \\
	& \les \sup_{|n-n_1|\les N_2} f_1(\s_1,n_1) \prod_{j=2}^3 \|f_j(\s_j)\|_{\l^2_n},
	\end{align*}
	by choosing $\al = \|f_2(\s_2)\|_{\l^2_n}^{-1} \|f_3(\s_3) \|_{\l^2_n}$.
	Using this estimate on $\mathcal{Q}$ gives
	\begin{align*}
	\|\mathcal{Q}(u_1, \ldots, u_5)\|_{Z_1}
	& \les \sum_{N_2, N_3} N_2^{\frac{2}{q_1}+\frac1p} \|f_1\|_{L^1_{\s}\l^p_n} \bigg(\prod_{j=2}^3 \|f_j\|_{L^1_\s\l^2_n} \bigg) \bigg(\prod_{k=4}^5 \|\ft{u}_k\|_{L^1_\tau \l^2_n}\bigg) \\
	& \les \sum_{N_2,N_3} N_2^{\frac{2}{q_1} + \frac1p} (N_2 N_3)^{\frac{1}{r_0} - \frac1p - \frac12 +} \|u_1\|_{Z_0} \prod_{j=2}^5 \|u_j\|_{Y_0}.
	\end{align*}
	The estimate follows from summing in the dyadics.

	\vspace{2mm}
	\noi\underline{\textbf{Subcase 2.2: }$N_2^{4\sqrt{\dl}} \geq N_3$}\\
	\noi Focusing on the spatial norm on \eqref{aux2_quintic}, we have
	\begin{align*}
	&\Bigg\| \sum_{\substack{n=n_1+n_2+n_3,\\ \Phi(\nbar_{123}) = \mu}}  \prod_{j=1}^3 f_j(\s_j,n_j) \Bigg\|_{\l^p_n \l^{\infty}_\mu} \\
	& \les \bigg\| \sum_{n_3} \bigg( \sum_{n_1 \in C(n,n_3,\mu)} f_1(\s_1,n_1) f_2(\s_2, n-n_1-n_3)  \bigg) f_3(\s_3,n_3) \bigg\|_{\l^p_n \l^\infty_\mu}\\
	& \les \sum_{n_3} f_3(\s_3,n_3) \big\| \sup_{n_1} f_1(\s_1,n_1) f_2(\s_2,n-n_1-n_3) \big\|_{\l^p_n},
	\end{align*}
	where the set $C$ is defined as follows
	$$C(n,n_3,\mu) = \big\{ n_1\in\Z : 3(n-n_1) (n-n+3) (n_1+n_3) = \mu \big\}$$
	and satisfies $|C(n,n_3,\mu)| \leq 2$.
	Substituting this estimate in \eqref{aux2_quintic} and using H\"older's inequality gives
	\begin{align*}
	\|\mathcal{Q}(u_1, \ldots, u_5)\|_{Z_1} 
	& \les \sum_{N_2,N_3} N_2^{\frac{2}{q_1} + \frac{1}{r_0} - \frac1p - \frac12 + } N_3^{\frac12 - \frac1p+} \|u_1\|_{Z_0} \prod_{j=2}^5 \|u_j\|_{Y_0}.
	\end{align*}
	The estimate follows from summing in the dyadics.

	\vspace{2mm}
	\noi\textbf{\underline{Case 3: two pairings $(2,3), (4,5)$}}\\
	\noi Using Minkowski's and Cauchy-Schwarz inequalities, we get the following
	\begin{align*}
	&\|\mathcal{Q}(u_1, \ldots, u_5)\|_{Z_1} \\
	& \les \bigg\| \|\jb{n}^\frac12 \ft{u}_1(n)\|_{L^1_\tau}  \sum_{n_2} \|\ft{u}_2(n_2)\|_{L^1_\tau} \|\ft{u}_3( -n_2)\|_{L^1_\tau}  \sum_{n_4} \|\ft{u}_4(n_4)\|_{L^1_\tau} \|\ft{u}_5 (- n_4)\|_{L^1_\tau} \bigg\|_{\l^p_n }\\
	& \les \| \jb{n}^\frac12 \ft{u}_1 \|_{\l^p_n L^1_\tau} \prod_{j=2}^5 \| \ft{u}_j \|_{\l^2_n L^1_\tau}.
	\end{align*}
	The result follows from H\"older's inequality.

\end{proof}

\begin{remark}\rm
	Note that the above estimate still holds if we include a factor of $\jb{n_j}^\eps$ in the multiplier, for some $j\in\{2,\ldots, 5\}$ and a small $0<\eps \ll 1$.
\end{remark}

\begin{proposition}\label{prop:septic}
	Let the first factor in $\mathcal{Q}$ have the highest frequency which is not associated to a pairing and $\#\in\{A,B\}$. Then, the following estimates hold
	\begin{align}
	\big\| \mathcal{Q} \big( \varphi_T \cdot \G_{\#}[u_1,u_2,u_3], u_4,\ldots, u_7\big) \big\|_{Z_1} & \les \|u_1 \|_{Z_0} \prod_{j=2}^7 \|u_j\|_{Y_0}, \label{septic}\\
	\big\|  \mathcal{Q} \big( \conj{\varphi_T \cdot \G_{\#}[u_1,u_2,u_3]}, u_4,\ldots, u_7\big)\big\|_{Z_1} & \les \|u_1 \|_{Z_0} \prod_{j=2}^7 \|u_j\|_{Y_0}.\label{septic2}
	\end{align}
\end{proposition}
\begin{proof}
	Note that the left-hand side of \eqref{septic} is controlled by the following quantity
	\begin{equation}
	\bigg\|  \sum_{n=n_1 + \ldots  + n_7}  \intt_{\R} \frac{\jb{n}^\frac12 |n_1|}{\jb{\Phi(\nbar_{123})}\jb{\tau  - \ld}^{1-\theta} } 
	\intt_{\ld = \tau_1 + \ldots + \tau_7} \prod_{j=1}^7 |\ft{u}_j(\tau_j, n_j)| \  d\ld \bigg\|_{\l^p_n L^{q_0}_\tau} \label{aux_septic}.
	\end{equation}
	Similarly, the left-hand side of \eqref{septic2} is controlled by \eqref{aux_septic} with $\ft{u}_j$ substituted by $\ft{\conj{u}}_j$, $j=1,2,3$. Thus, we will only focus on estimating \eqref{aux_septic}. 
	
	Since $(n_1,n_2,n_3)\in \X_*(n_0)$ and $|n_0| = \max(|n_0|, |n_4|, \ldots, |n_7|)$, we have 
	$$\frac{\jb{n}^\frac12 |n_1|}{\jb{\Phi(\nbar_{123})}} \les \frac{1}{\max\limits_{j=1,\ldots,7}\jb{n_j}^\frac12}.$$
	Consider the change of variables $\s_j = \tau_j - n_j^3$, $j=1,\ldots, 7$ and let $f_j(\s, n) = | \ft{\P_{N_j} u}_j(\s + n^3, n)|$, $j=1, \ldots, 7$.

	\vspace{2mm}
	\noi\underline{\textbf{Case 1:} no pairings } \\
	\noi Assume that $*=A$. Then, $|n_1| \geq |n_j|$, $j=2,\ldots,7$. Using Minkowski's inequality and Schur's test, we have
	\begin{align*}
	\eqref{aux_septic}  &\les  \sum_{N_1, \ldots, N_7} N_1^{-\frac12} \intt_{\R^7} \bigg\| \sum_{n=n_1 + \ldots + n_7} \frac{1}{\jb{\tau - \tau_1 - \ldots - \tau_7 + \Psi(n, \nbar_{1\ldots 7})}^{1-\theta}} \prod_{j=1}^7 f_j(\s_j, n_j) \bigg\|_{\l^p_n L^{q_0}_\tau}\\
	 & \les  \sum_{N_1, \ldots, N_7} \intt_{\R^7} N_1^{-\frac12} \Bigg\| \sum_{\substack{n=n_1+ \ldots+ n_7,\\ \Psi(n,\nbar_{1\ldots7} ) = \mu }} \prod_{j=1}^7 f_j(\s_j, n_j) \Bigg\|_{\l^p_n \l^{q_1}_{\mu}},
	\end{align*}
	where $\Psi(n,\nbar_{1\ldots 7}) = n^3 -n_1^3 - \ldots - n_7^3$. Using H\"older's inequality, it follows that
	\begin{align*}
	\eqref{aux_septic}  &\les   \sum_{N_1, \ldots, N_7} \intt_{\R^7} N_1^{-\frac12} \Bigg\| \1_{\mu \in A(n,N_1, \ldots, N_7)}\sum_{\substack{n=n_1+ \ldots+ n_7,\\ \Psi(n,\nbar_{1\ldots7} ) = \mu  }} \prod_{j=1}^7 f_j(\s_j, n_j) \Bigg\|_{\l^p_n \l^{q_1}_{\mu}} \\
	& \les \sum_{N_1, \ldots, N_7} \intt_{\R^7} N_1^{-\frac12} (N_2 \cdots N_7)^{\frac{1}{q_1}} \bigg\| \sum_{\substack{n=n_1+ \ldots+ n_7,\\ \Psi(n,\nbar_{1\ldots7} ) = \mu  }} \prod_{j=1}^7 f_j(\s_j, n_j) \Bigg\|_{\l^p_n \l^{\infty}_{\mu}},
	\end{align*}
	where
	\begin{multline*}
	A(n, N_1, \ldots,N_7) = \big\{ \mu\in\Z : \ \mu = n^3 - (n-n_2-\ldots-n_7)^3 - n_2^3 - \ldots - n_7^3, \\  |n_j|\sim N_j, j=2, \ldots, 7 \big\}
	\end{multline*}
	which satisfies $|A(n,N_2, \ldots, N_7)| \les N_2 \cdots N_7$, uniformly in $n$.
	Focusing on the inner sum and omitting the time dependence, we have for $\al>0$ 
	\begin{align*}
	\sum_{\substack{n=n_1+ \ldots+ n_7,\\ \Psi(n,\nbar_{1\ldots7} ) = \mu }} \prod_{j=1}^7 f_j(n_j)  & \les  \sum_{(n_2, n_3, n_4)} \sum_{n_1} f_1(n_1)\sum_{\substack{(n_5,n_6,n_7) \\\in B(n,n_1,n_2,n_3, n_4, \mu)}} \al|f_2(n_2) f_3(n_3) f_4(n_4)|^2 \\
	& + \sum_{(n_5,n_6,n_7)} \sum_{n_1} f_1(n_1) \sum_{\substack{(n_2, n_3, n_4)\\\in B(n,n_1,n_5,n_6,n_7, \mu)}} \al^{-1}|f_5(n_5) f_6(n_6) f_7(n_7)|^2  ,
	\end{align*}
	where 
	\begin{multline*}
	B(n,n_1,n_2,n_3,n_4, \mu)= \big\{ (n_5,n_6,n_7)\in\Z^3: \ n_5 + n_6 + n_7 = n-n_1 - n_2 - n_3 - n_4, \\
	n_5^3 + n_6^3 + n_7^3 = n^3 - n_1^3 - n_2^3 - n_3^3 - n_4^3 - \mu \big\}.
	\end{multline*}
	Using Lemma~\ref{lm:divisor} (2) we have that $|B(n,n_1, \ldots, n_4, \mu)|, |B(n, n_1, n_5, n_6, n_7, \mu)| \les N_{2+}^\eps$, for any $\eps>0$ and $N_{2+} = \max(N_2, \ldots, N_7)$. In addition, we know that $|n-n_1| \les N_{2+}$. Thus, it follows that
	\begin{align*}
	\sum_{\substack{n=n_1+ \ldots+ n_7,\\ \Psi(n,\nbar_{1\ldots7} ) = \mu }} \prod_{j=1}^7 f_j(n_j)  & \les N_{2+}^\eps \bigg( \sum_{|n-n_1| \les N_{2+}} f_1(n_1) \bigg) \prod_{j=2}^7 \|f_j\|_{\l^2_n}, 
	\end{align*}
	by choosing $\al = (\|f_2\|_{\l^2_n} \|f_3\|_{\l^2_n}\|f_4\|_{\l^2_n})^{-1} \|f_5\|_{\l^2_n} \|f_6\|_{\l^2_n}\|f_7\|_{\l^2_n} $.
	Consequently, using H\"older's and Minkowski's inequality gives the following
	\begin{align*}
	\eqref{aux_septic}
	& \les \sum_{N_1, \ldots, N_7} N_1^{\frac12 -\frac1p + } N_{2+}^{\frac1p + \eps} (N_2 \cdots N_7)^{\frac{1}{q_1}} \|f_1\|_{L^1_\tau \l^p_n} \prod_{j=1}^7 \|f_j\|_{L^1_\s \l^2_n}\\
	& \les \sum_{N_1, \ldots, N_7}  N_1^{ - \frac1p +}  N_{2+}^{ \frac1p + \eps} (N_2 \cdots N_7)^{\frac{1}{q_1} + \frac{1}{r_0} - \frac1p - \frac12 +}  \| u_1 \|_{Z_0} \prod_{j=2}^7 \|u_j\|_{Y_0}.
	\end{align*}
	The estimate follows from summing in the dyadics. 
	If $*=B$, then $|n_2|$ is the largest frequency and we can proceed as for $*=A$ by swapping the roles of $u_1$ and $u_2$.
	
	\vspace{2mm}
	\noi\underline{\textbf{Case 2:} at least one pairing} \\
	If there is at least one pairing, we can show a stronger estimate, where every function belongs in $Y_0$, due to the negative power of the largest frequency which is not in a pairing. This case will follow a similar strategy to the proof of Proposition~\ref{standard_quintic}, thus we will omit the details.

\end{proof}

\subsection{Remaining quintic terms}\label{sec:rem_quintic}
It remains to estimate the terms in \eqref{quintic_bad}. These terms cannot be written as \eqref{quintic} and thus require a finer analysis.

For the $\B^j_*$ terms, we need to use the modulations to help establishing the estimate. For example, calculating the space-time Fourier transform of the $\B_*^3$ terms in \eqref{quintic_bad}, $*\in\{A,B\}$, we have
\begin{align*}
&\big| \Ft_{t,x} \B^{3}_{*}\big(u_1, u_2, \varphi_T \cdot \G_{\#} [u_3, u_4,u_5]\big) (\tau, n) \big| \\
& \les  \sum_{\substack{\nbar_{120}\in\X_*(n), \\\nbar_{345} \in \X_{\#}(n_0)}} \intt_{\R^3} \intt_{\substack{\ld = \tau_1 + \tau_2 + \tau_0,\\\s = \tau_3 + \tau_4 + \tau_5}} |n_1n_3| \big| K_{+} (\tau-n^3, \ld - n^3, \Phi(\nbar_{120})) \big|\\
& \phantom{XX} \times \frac{ \1_{|\tau_0 - n_0^3| \ges |\ld - n^3 + \Phi(\nbar_{120})|}}{\jb{\tau_0 - \mu}  \jb{\mu - \s}  } \min\bigg(\frac{1}{\jb{\Phi(\nbar_{345})}}, \frac{1}{\jb{\tau_0 - n_0^3}}\bigg)  \prod_{j=1}^5 |\ft{u}_j(\tau_j, n_j)| \ d \s \ d \mu \ d \ld,
\end{align*}
using \eqref{estimateKY_new}. In order to control the multiplier, we must consider two cases depending on the modulations of the second generation: 
\begin{align}
|\tau_0 - n_0^3| &\gg |\s - n_0^3| \label{good_modulation},\\
|\tau_0 - n_0^3| &\les |\s - n_0^3| \label{bad_modulation}.
\end{align}
If \eqref{good_modulation} holds, then $|\tau_0 - \s| \sim |\tau_0 - n_0^3| \ges |\ld - n^3 + \Phi(\nbar_{105})|$ and we can obtain powers of the resonance relation of the first and the second generations:
\begin{multline*}
\1_{|\tau_0 - n_0^3| \gg |\s - n_0^3|} \big| \Ft_{t,x} \B^{3}_{*}\big(u_1, u_2, \varphi_T \cdot \G_{\#} [ u_3,u_4, u_5] \big) (\tau, n) \big|  \\
 \les  \sum_{\substack{\nbar_{120}\in\X_*(n), \\\nbar_{345} \in \X_{\#}(n_0)}}  \intt_\R \frac{\max\limits_{j=1, \ldots, 5} \jb{n_j}^{9\theta} |n_1n_3|}{\jb{\Phi(\nbar_{120})} \jb{\Phi(\nbar_{345})} \jb{\tau-n^3}^{1+\theta}} \intt_{\s' = \tau_1 + \ldots + \tau_5 } \prod_{j=1}^5 |\ft{u}_j(\tau_j, n_j)| \, d\s'.
\end{multline*}
If \eqref{bad_modulation} holds, we can only gain a power of the resonance relation of the first generation
\begin{multline*}
\1_{|\tau_0 - n_0^3| \les |\s - n_0^3|} \big| \Ft_{t,x} \B^{3}_{*}\big(u_1, u_2,\varphi_T \cdot \G_{\#} [ u_3,u_4, u_5] \big) (\tau, n) \big|  \\
\les  \sum_{\substack{\nbar_{120}\in\X_*(n), \\\nbar_{345} \in \X_{\#}(n_0)}}  \intt_\R \frac{\max\limits_{j=1, \ldots, 5} \jb{n_j}^{9\theta}|n_1n_3|}{\jb{\Phi(\nbar_{120})} \jb{\tau-n^3}^{1+\theta} \jb{\tau-\s'}^{1-\theta}} \intt_{\s' = \tau_1 + \ldots + \tau_5 } \prod_{j=1}^5 |\ft{u}_j(\tau_j, n_j)| \, d\s'.
\end{multline*}
Analogous estimates hold for the $\B^2_A$ contributions in \eqref{quintic_bad}, where the sums are over $\nbar_{105}\in\X_A(n)$ and $\nbar_{234}\in\X_{\#}(n_0)$, $\#\in\{A,B\}$, and we consider the same case separation depending on $\tau_0 - n_0^3$.

For the $\B^j_*$ contributions restricted to the region where \eqref{bad_modulation} holds and the $\mathcal{DN}_C$ terms in \eqref{quintic_bad}, we start by looking at the spatial frequencies in more detail. In the frequency regions where the spatial multipliers can be bounded, the contributions are controlled by \eqref{quintic} and we can use Proposition~\ref{standard_quintic}. Otherwise, we can apply the following result.

\begin{proposition}\label{new_quintic_3}
	Assume that the frequencies are ordered as follows $|n_1| \geq \ldots \geq |n_5|$. If $|n_1| \sim |n_2| \gg |n_3|\ges |n|$, $\al \les \frac{|n_1| \max_j|n_j|^{3\theta}}{|n_3|}$ and $(1,2)$ not a pairing, then the following estimate holds
	\begin{multline}
	\bigg\| \jb{n}^\frac12 \jb{\tau-n^3}^{b_1} \sum_{n=n_1+\ldots +n_5} \intt_\R \frac{\al(n, n_1, \ldots, n_5)}{\jb{\tau-n^3}^{1-\theta} \jb{\tau - \s'}^{1-\theta}} \intt_{\s' = \tau_1 + \ldots + \tau_5} \prod_{j=1}^5 |\ft{u}_j (\tau_j, n_j)| \ d\s' \bigg\|_{\l^p_n L^{q_0}_\tau} \\ \les \bigg(\prod_{j=1}^3\|u_j\|_{Z_0}\bigg)\|u_4\|_{Y_0} \|u_5\|_{Y_0}. \label{eq_new_quintic_3}
	\end{multline}
\end{proposition}

The $\B^j_*$ terms in \eqref{quintic_bad} localized to \eqref{good_modulation} can be estimated by the following proposition.
\begin{proposition} \label{prop:quintic_good}
	Let $\mathcal{Q}'(u_1,\ldots,u_5)$ be such that 
	\begin{multline*}
	 \big| \Ft_{t,x}\mathcal{Q}'(u_1, \ldots, u_5) (\tau,n)\big| \\\les \sum_{n=n_1+n_2+n_0} \sum_{n_0=n_3+n_4+n_5} \intt_{\R^5} \frac{\max_j \jb{n_j}^{9\theta}}{\jb{n_{\max} } \jb{n'_{\max}} \jb{\tau-n^3}^{1+\theta}} \prod_{j=1}^5 | \ft{u}_j(\tau_j, n_j) | \ d\tau_1 \cdots d\tau_5,
	\end{multline*}
	where $n_{\max} = \max(|n_1|, |n_2|) \geq |n_0|$ and $n'_{\max} = \max_{j=3,4,5} |n_j|$. Then, the following estimate holds
	\begin{align*}
	\big\| \mathcal{Q}' (u_1, \ldots, u_5) \|_{Z_1} & \les \prod_{j=1}^5 \|u_j\|_{Y_0}.
	\end{align*}
\end{proposition}

It only remains to show the previously stated propositions.
\begin{proof}[Proof of Proposition~\ref{new_quintic_3}]
	
	Due to the $\theta$ loss in the largest frequency when estimating $\al$, we will distinguish two cases: when $|n_1|^{\frac12}\les |n_3|$ and when $|n_1|^{\frac12}\gg |n_3|$. 
	
	\vspace{2mm}
	\noi\underline{\textbf{Case 1:} $|n_1|^{\frac12}\les |n_3|$}
	
	\noi Using the notation $\Psi(n, \nbar_{1\ldots5}) = n^3 - n_1^3 - \ldots - n^3_5$ and the change of variables $\s_j = \tau_j - n_j^3$, $j=1,\ldots,5$, we start by applying Minkowski's inequality and Schur's test to obtain the following
	\begin{align*}
	& \text{LHS of }\eqref{eq_new_quintic_3} \\
	& \les \intt_{\s_1, \ldots, \s_5} \bigg\| \sum_{\mu} \frac{1}{\jb{\tau-n^3 - \bar{\s} + \mu}^{1-\theta}}\bigg( \sum_{\substack{n=n_1+\ldots+n_5,\\\Psi(n, \nbar_{1\ldots5}) = \mu}} \frac{|n_1| \jb{n_1}^{3\theta} }{ \max(\jb{n_3}, \jb{n})^\frac12 }  \prod_{j=1}^5 |\ft{u}_j(\s_j+n_j^3, n_j)| \bigg) \bigg\|_{\l^p_n L^{q_0}_\tau} \\
	& \les \intt_{\s_1, \ldots, \s_5} \bigg\| \sum_{\substack{n=n_1+\ldots+n_5,\\\Psi(n, \nbar_{1\ldots5}) = \mu}} \frac{|n_1| \jb{n_1}^{3\theta} }{ \max(\jb{n_3}, \jb{n})^\frac12}  \prod_{j=1}^5 |\ft{u}_j(\s_j + n_j^3, n_j)| \bigg\|_{\l^p_n \l^{q_1}_\mu},
	\end{align*}
	where $\bar{\s} = \s_1 + \ldots + \s_5$ and $\theta< \frac\dl2$.
	
	Let $f_j(\s,n) = \jb{n}^\frac12 | \ft{\P_{N_j} u}_j(\s  + n^3, n)|$, $j=1,2$, and $f_k(\s,n) = |\ft{\P_{N_k} u}_k(\s + n^3,n)|$, $k=3,4,5$, with $N_1\sim N_2 \gg N_3 \geq N_4 \geq N_5$ dyadic numbers with $N_1 \les N_3^2$. Omitting the time dependence, using H\"older's inequality and the standard divisor counting estimate, we have
	\begin{align*}
	& N_1^{3\theta} N_3^{-\frac12} \bigg\| \sum_{\substack{n=n_1+\ldots+n_5,\\\Psi(n, \nbar_{1\ldots5}) = \mu}}  \prod_{j=1}^5 f_j(n_j) \bigg\|_{\l^p_n \l^{q_1}_\mu} \\
	&  \les N_1^{3\theta + \eps} N_3^{-\frac12}  \bigg\| \sum_{(n_4, n_5)} f_4(n_4) f_5(n_5) \bigg(\sum_{\substack{n_1 + n_2 + n_3 = n - n_4 - n_5 ,\\ n_1^3 + n_2^3 + n_3^3 = n^3 - n_4^3 - n_5^3 - \mu}} \prod_{j=1}^3 |f_j(n_j)|^p \bigg)^{\frac1p} \bigg\|_{\l^p_n \l^{q_1}_\mu} \\
	& \les N_1^{3\theta + \eps} N_3^{-\frac12}  \| f_1\|_{\l^p_n} \| f_2\|_{\l^p_n} \| f_3\|_{\l^p_n} \|f_4\|_{\l^1_n} \|f_5\|_{\l^1_n}.
	\end{align*}
Applying the previous estimate, we obtain
	\begin{align}
	\text{LHS of }\eqref{eq_new_quintic_3} & \les \sum_{N_1, \ldots, N_5} N_1^{3\theta+\eps} N_3^{-\frac12} \bigg(\prod_{j=1}^3 \|f_j\|_{L^1_\s \l^p_n} \bigg) \bigg( \prod_{k=4}^5 \|f_k\|_{L^1_\s \l^1_n}  \bigg) \nonumber\\
	& \les \sum_{N_1, \ldots, N_5} N_1^{3\theta+\eps} N_3^{-1} (N_4 N_5)^{1-\frac1p - \frac12 +} \bigg(\prod_{j=1}^3\|u_j\|_{Z_0} \bigg) \bigg( \prod_{k=4}^5 \| u_k\|_{ Y_0} \bigg). \label{aux_quintic_three}
	\end{align}
	Using the fact that $N_1 \sim N_2 \les N_3^2$, for $\eps, \theta<\frac{\dl}{2}$ and $\dl$ small enough, the estimate follows from summing in the dyadic numbers.

	\vspace{2mm}
	\noi\underline{\textbf{Case 2:} $|n_1|^{\frac12}\gg |n_3|$}
	
	\noi In this case, we need a different approach to control the small power of $N_1$ in the multiplier as well as the $\eps$-loss from using the divisor counting estimate. Note that $\Psi(n, \nbar_{1\ldots5}) = 3(n_1+n_2)(n_1 + n_3 + n_4 + n_5)(n_2 + n_3 + n_4 + n_5) + 3(n_3+n_4)(n_3+n_5)(n_4+n_5)$. Since
	\begin{align*}
	|(n_3+n_4)(n_3+n_5)(n_4+n_5)| \les |n_3|^3 \ll |n_1|^{\frac32},\\
	|(n_1+n_2)(n_1 + n_3 + n_4 + n_5)(n_2 + n_3 + n_4 + n_5)| \ges |n_1|^2,
	\end{align*}
	then $|\Psi(n, \nbar_{1\ldots5})| \ges |n_1|^2$. Following the previous strategy, we have
	\begin{align*}
	& \text{LHS of } \eqref{new_quintic_3} \les \intt_{\s_1, \ldots, \s_5} \bigg\| \sum_{n=n_1+ \ldots +n_5} \frac{|n_1| \jb{n_1}^{3\theta}}{\max(\jb{n}, \jb{n_3})^\frac12} \\
	& \phantom{XXX} \times \frac{1}{ \jb{\tau - n^3 - \bar{\s} + \Psi(n, \nbar_{1\ldots5})}^{1-\theta} \jb{\tau-n^3}^{\dl - \theta}}\prod_{j=1}^5 |\ft{u}_j(\s_j + n_j^3, n_j)| \bigg\|_{\l^p_n L^{q_0}_\tau}
	\end{align*}
	Thus, in order to control the small powers of $\jb{n_1}$, we use the following fact
	\begin{equation*}
	|n_1|^2 \les \jb{\Psi(n, \nbar_{1\ldots5}) } \les \jb{\tau - n^3 - \bar{\s} + \Psi(n, \nbar_{1\ldots5}) } \jb{\tau - n^3} \jb{\s_1} \cdots \jb{\s_5}. \label{quintic3_aux1}
	\end{equation*}
	To gain a power of $\jb{\tau-n^3}$, we impose $\theta \leq \frac\dl2$. For $\jb{\tau-n^3 - \bar{\s}+ \Psi(n, \nbar_{1\ldots5})}$, when applying Schur's estimate, we want to keep $\tfrac\dl4$ of this quantity. Thus, with $1 + \frac{1}{q_0} = \frac{1}{q_1} + \frac1r$, we need 
	$$1-\theta - \frac{\dl}{4}  >  \frac1r = 1 - \frac{\dl}{2} \implies \theta < \frac\dl2.$$
	We can obtain a power of $\jb{\s_k}^{\al}$ for $k=4,5$, given that
	\begin{align*}
	\| \jb{\s}^\al f_k\|_{\l^1_n L^1_\s} \les \| \jb{\s}^{\al + \frac12 - \dl +} f_k \|_{\l^1_n L^{r_0}_\s} \les \| \jb{\s}^{\frac12} f_k \|_{\l^1_n L^{r_0}_\s},
	\end{align*}
	given that $\al + \frac12 - \dl < \frac12 \implies \al < \dl$, thus we can choose $\al = \frac{\dl}{4}$. Similarly, for $\jb{\s_j}^\be$, $j=1,2,3$, we have
	\begin{align*}
	\| \jb{\s}^\be f_j \|_{L^1_\s \l^p_n} \les \| \jb{\s}^{\be + 1 - 4\dl +} f_j \|_{L^{q_0}_\s \l^p_n} \les \| \jb{\s}^{1- 2\dl } f_j \|_{L^{q_0}_\s \l^p_n},
	\end{align*}
	given that $$\be + 1 - 4\dl < 1 - 2\dl \implies \be < 2\dl, $$ so we can choose $\be = \frac{\dl}{4}$, as well.
	Combining all of these powers, we get $\jb{\Psi(n, \nbar_{1\ldots5})}^{-\frac{\dl}{4}}\les N_1^{-\frac{\dl}{2}}$ which we use in \eqref{aux_quintic_three} instead of the condition $N_1\les N_3^2$.

\end{proof}

\begin{proof}[Proof of Proposition~\ref{prop:quintic_good}]
	We have the following estimate
	\begin{align*}
	&\| \mathcal{Q}'(u_1, \ldots, u_5) \|_{Z_1} \\
	& \les \bigg\| \sum_{n=n_1+n_2+n_0} \sum_{n_0=n_3+n_4+n_5} \intt_{\R^5} \frac{\max_j \jb{n_j}^{9\theta} \jb{n}^\frac12}{\jb{n_{\max} } \jb{n'_{\max}} \jb{\tau-n^3}^{\theta + \dl}} \prod_{j=1}^5 | \ft{u}_j(\tau_j, n_j) | \ d\tau_1 \cdots d\tau_5 \bigg\|_{\l^p_n L^{q_0}_\tau} \\
	& \les \bigg\| \sum_{n=n_1+n_2+n_0} \sum_{n_0=n_3+n_4+n_5}  \frac{\max_j \jb{n_j}^{9\theta} \jb{n}^\frac12}{\jb{n_{\max} } \jb{n'_{\max}}} \prod_{j=1}^5 \| \ft{u}_j( n_j) \|_{L^1_\tau} \bigg\|_{\l^p_n }, 
	\end{align*}
	given that $\theta + \dl > 4\dl \implies \theta > 3 \dl$. Let $f_j(\tau,n) = \jb{n}^\frac12 | \ft{\P_{N_j} u} (\tau,n)|$, for dyadic numbers $N_j$, $j=1, \ldots, 5$. By symmetry, we can assume without loss of generality that $N_1 \geq N_2$, $N_3 \geq N_4 \geq N_5$. We will consider two cases: $|n_1| \geq |n_3|$ or $|n_1| < |n_3|$. Since the strategies follow a similar approach, we will only focus on the former.
		
	\smallskip
	Assume that $|n_1| \geq |n_3|$.
	Using Young's and H\"older's inequality, we get
	\begin{align*}
	\|\mathcal{Q}' (u_1, \ldots, u_5)\|_{Z_1} & \les \sum_{N_1, \ldots, N_5} N_1^{9\theta - 1} (N_2 N_4 N_5)^{-\frac12 } N_3^{-\frac32} \bigg\| \sum_{n=n_1+ \ldots + n_5} \|f_j(n_j)\|_{L^1_\tau} \bigg\|_{\l^p_n} \\
	& \les \sum_{N_1, \ldots, N_5} N_1^{9\theta - 1} (N_2 N_4 N_5)^{-\frac12 } N_3^{-\frac32} \|f_1\|_{\l^p_n L^1_\tau} \prod_{j=2}^5 \|f_j\|_{\l^1_n L^1_\tau}\\
	& \les \sum_{N_1, \ldots, N_5} N_1^{9\theta - 1} (N_2 N_4 N_5)^{\frac12 - \frac1p + } N_3^{-\frac12 - \frac1p +}  \prod_{j=1}^5 \|u_j\|_{Y_0}.
	\end{align*}
	It only remains to sum in the dyadics
	\begin{align*}
	\sum_{N_1, \ldots, N_5} N_1^{9\theta - 1} (N_2 N_4 N_5)^{\frac12 - \frac1p + } N_3^{-\frac12 - \frac1p +} & \les \sum_{N_1, \ldots, N_5} (N_1 N_2)^{-\theta} (N_3N_4N_5)^{\frac16 - \frac1p + (-\frac12 - \frac1p + 11\theta )/3 +} 
	\end{align*}
	and the estimate follows if 
	$$\frac16 - \frac1p - \frac16 - \frac{1}{3p} + \frac{11}{3}\theta < 0 \implies 3\dl < \theta < \frac{4}{11p}. $$

\end{proof}

\appendix

	\section{Equation for $w$}\label{app:w}
	Here we present the equation for $w$ after using second iteration once, following the strategy described in Section~\ref{sec:equations}. 
	\begin{align}
	w & = \varphi(t) S(t) u_0 + \varphi_T\cdot \mathcal{IR}(u, u, u) \nonumber\\
	& + \varphi_T\cdot (\D  \mathcal{NR}_{C,\geq} + \D  \mathcal{NR}_{D,\geq})(w,\conj{w},w) + \varphi_T\cdot (\D  \mathcal{NR}_{C,>} + \D  \mathcal{NR}_{D,>})(w,w,\conj{w}) \nonumber\\
	& + \varphi_T \big( \B^{ 0}_{A,\geq} (w, \conj{u},u) + \B^{ 1}_{A,\geq}(w, \conj{u},u)  + \B^{ 2}_{A,\geq} (w, \conj{w},u)  + \B^{ 3}_{A,\geq}(w, \conj{u},w)  \big) \nonumber\\
	& + \varphi_T \big( \B^{ 0}_{A,>} (w, u,\conj{u}) + \B^{ 1}_{A,>}(w, u,\conj{u})  + \B^{ 2}_{A,>} (w, w,\conj{u})  + \B^{ 3}_{A,>}(w, u,\conj{w})  \big) \nonumber\\
	& +   \varphi_T \big( \B^{ 0}_{B,\geq} (w, \conj{w},u) + \B^{ 1}_{B,\geq} (w, \conj{w},u) + \B^{ 2}_{B, \geq} (w, \conj{w},u) + \B^{ 3}_{B, \geq} (w, \conj{w},w) \big) \nonumber\\
	& +   \varphi_T \big( \B^{ 0}_{B,>} (w, w,\conj{u}) + \B^{ 1}_{B,>} (w, w,\conj{u}) + \B^{ 2}_{B, >} (w, w,\conj{u}) + \B^{ 3}_{B,>} (w, w,\conj{w}) \big) \nonumber\\
	&+   \varphi_T(\D  \mathcal{NR}_{A, \geq} + \D  \mathcal{NR}_{B, \geq}+ \D  \mathcal{NR}_{C, \geq} + \D  \mathcal{NR}_{D, \geq})\big(\varphi_T\cdot \G_{A, \geq}[w,\conj{u}, u],\conj{u}, u\big) \nonumber\\
	&+   \varphi_T(\D  \mathcal{NR}_{A, \geq} + \D  \mathcal{NR}_{B, \geq}+ \D  \mathcal{NR}_{C, \geq} + \D  \mathcal{NR}_{D, \geq})\big(\varphi_T\cdot \G_{A, >}[w,u,\conj{u}],\conj{u}, u\big) \nonumber\\
	&+   \varphi_T(\D  \mathcal{NR}_{A, \geq} + \D  \mathcal{NR}_{B, \geq}+ \D  \mathcal{NR}_{C, \geq} + \D  \mathcal{NR}_{D, \geq})\big(\varphi_T\cdot \G_{B, \geq}[w,\conj{w},u],\conj{u}, u\big) \nonumber\\
	&+   \varphi_T(\D  \mathcal{NR}_{A, \geq} + \D  \mathcal{NR}_{B, \geq}+ \D  \mathcal{NR}_{C, \geq} + \D  \mathcal{NR}_{D, \geq})\big(\varphi_T\cdot \G_{B, >}[w,w,\conj{u}],\conj{u}, u\big) \nonumber\\
	&+   \varphi_T(\D  \mathcal{NR}_{A, >} + \D  \mathcal{NR}_{B, >}+ \D  \mathcal{NR}_{C, >} + \D  \mathcal{NR}_{D,>})\big(\varphi_T\cdot \G_{A, \geq}[w,\conj{u}, u],u,\conj{u}\big) \nonumber\\
	&+   \varphi_T(\D  \mathcal{NR}_{A, >} + \D  \mathcal{NR}_{B,>}+ \D  \mathcal{NR}_{C, >} + \D  \mathcal{NR}_{D, >})\big(\varphi_T\cdot \G_{A, >}[w,u,\conj{u}],u,\conj{u}\big) \nonumber\\
	&+   \varphi_T(\D  \mathcal{NR}_{A, >} + \D  \mathcal{NR}_{B, >}+ \D  \mathcal{NR}_{C, >} + \D  \mathcal{NR}_{D, >})\big(\varphi_T\cdot \G_{B, \geq}[w,\conj{w},u],u,\conj{u}\big) \nonumber\\
	&+   \varphi_T(\D  \mathcal{NR}_{A, >} + \D  \mathcal{NR}_{B, >}+ \D  \mathcal{NR}_{C,>} + \D  \mathcal{NR}_{D, >})\big(\varphi_T\cdot \G_{B, >}[w,w,\conj{u}],u,\conj{u}\big) \nonumber\\
	& +   \varphi_T(\D  \mathcal{NR}_{B, \geq}+\D  \mathcal{NR}_{C, \geq} + \D  \mathcal{NR}_{D, \geq})\big(w, \conj{\varphi_T\cdot \G_{A, \geq}[w,\conj{u}, u] },  u\big) \nonumber\\
	& +   \varphi_T(\D  \mathcal{NR}_{B, \geq}+\D  \mathcal{NR}_{C, \geq} + \D  \mathcal{NR}_{D, \geq})\big(w, \conj{\varphi_T\cdot \G_{A, >}[w,u, \conj{u}] },  u\big) \nonumber\\
	& +   \varphi_T(\D  \mathcal{NR}_{B, \geq}+\D  \mathcal{NR}_{C, \geq} + \D  \mathcal{NR}_{D, \geq})\big(w, \conj{\varphi_T\cdot \G_{B, \geq}[w,\conj{w},u] },  u\big) \nonumber\\
	& +   \varphi_T(\D  \mathcal{NR}_{B, \geq}+\D  \mathcal{NR}_{C, \geq} + \D  \mathcal{NR}_{D, \geq})\big(w, \conj{\varphi_T\cdot \G_{B, >}[w,w, \conj{u}] },  u\big) \nonumber\\
	& +   \varphi_T(\D  \mathcal{NR}_{B, >}+\D  \mathcal{NR}_{C, >} + \D  \mathcal{NR}_{D,>})\big(w, \varphi_T\cdot \G_{A, \geq}[w,\conj{u}, u] ,  \conj{u}\big) \nonumber\\
	& +   \varphi_T(\D  \mathcal{NR}_{B, >}+\D  \mathcal{NR}_{C, >} + \D  \mathcal{NR}_{D, >})\big(w, \varphi_T\cdot \G_{A, >}[w,u, \conj{u}] ,  \conj{u}\big) \nonumber\\
	& +   \varphi_T(\D  \mathcal{NR}_{B, >}+\D  \mathcal{NR}_{C, >} + \D  \mathcal{NR}_{D, >})\big(w, \varphi_T\cdot \G_{B, \geq}[w,\conj{w},u] , \conj{u}\big) \nonumber\\
	& +   \varphi_T(\D  \mathcal{NR}_{B, >}+\D  \mathcal{NR}_{C, >} + \D  \mathcal{NR}_{D, >})\big(w, \varphi_T\cdot \G_{B, >}[w,w, \conj{u}] ,  \conj{u}\big) \nonumber\\
	& +   \varphi_T(\D  \mathcal{NR}_{C,\geq} + \D  \mathcal{NR}_{D, \geq})\big(w, \conj{w},\varphi_T\cdot \G_{A, \geq}[w,\conj{u}, u]\big) \nonumber\\
	& +   \varphi_T(\D  \mathcal{NR}_{C,\geq} + \D  \mathcal{NR}_{D, \geq})\big(w, \conj{w},\varphi_T\cdot \G_{A, >}[w,u,\conj{u}]\big) \nonumber\\
	& +   \varphi_T(\D  \mathcal{NR}_{C,\geq} + \D  \mathcal{NR}_{D, \geq})\big(w, \conj{w},\varphi_T\cdot \G_{B, \geq}[w,\conj{w}, u]\big) \nonumber\\
	& +   \varphi_T(\D  \mathcal{NR}_{C,\geq} + \D  \mathcal{NR}_{D, \geq})\big(w, \conj{w},\varphi_T\cdot \G_{B,>}[w,w,\conj{u}]\big) \nonumber\\
	& +   \varphi_T(\D  \mathcal{NR}_{C,>} + \D  \mathcal{NR}_{D, >})\big(w,w, \conj{\varphi_T\cdot \G_{A, \geq}[w,\conj{u}, u]} \big) \nonumber\\
	& +   \varphi_T(\D  \mathcal{NR}_{C,>} + \D  \mathcal{NR}_{D, >})\big(w, w,\conj{\varphi_T\cdot \G_{A, >}[w,u,\conj{u}] }\big) \nonumber\\
	& +   \varphi_T(\D  \mathcal{NR}_{C,>} + \D  \mathcal{NR}_{D, >})\big(w, w, \conj{\varphi_T\cdot \G_{B, \geq}[w,\conj{w}, u]} \big) \nonumber\\
	& +   \varphi_T(\D  \mathcal{NR}_{C,>} + \D  \mathcal{NR}_{D,>})\big(w, w, \conj{\varphi_T\cdot \G_{B,>}[w,w,\conj{u}] }\big) \nonumber\\
	&+   \varphi_T \Big( \B^{ 2}_{A, \geq} \big(w, \conj{\varphi_T \cdot \G_{A, \geq}[w,\conj{u}, u]}, u \big) + \B^{ 2}_{A, \geq} \big(w, \conj{\varphi_T \cdot \G_{A, >}[w,u,\conj{u}]}, u \big) \Big) \nonumber\\
	&+\varphi_T \Big( \B^{ 2}_{A, \geq} \big(w, \conj{\varphi_T \cdot \G_{B, \geq}[w,\conj{w}, u]}, u \big) + \B^{ 2}_{A, \geq} \big(w, \conj{\varphi_T \cdot \G_{B,>}[w,w,\conj{u}]}, u \big)  \Big) \nonumber\\
	&+   \varphi_T \Big( \B^{ 2}_{A, >} \big(w, \varphi_T \cdot \G_{A, \geq}[w,\conj{u}, u], \conj{ u} \big) + \B^{ 2}_{A, >} \big(w, \varphi_T \cdot \G_{A, >}[w,u,\conj{u}], \conj{u} \big) \Big) \nonumber\\
	&+\varphi_T \Big( \B^{ 2}_{A, >} \big(w, \varphi_T \cdot \G_{B, \geq}[w,\conj{w}, u], \conj{u} \big) + \B^{ 2}_{A, >} \big(w, \varphi_T \cdot \G_{B,>}[w,w,\conj{u}], \conj{u} \big)  \Big) \nonumber\\
	& +   \varphi_T \Big( \B^{ 3}_{A, \geq} \big(w, \conj{u}, \varphi_T\cdot \G_{A, \geq}[w, \conj{u}, u] \big) + \B^{ 3}_{A, \geq} \big(w, \conj{u}, \varphi_T\cdot \G_{A, >}[w,u, \conj{u}] \big) \Big) \nonumber\\
	& + \varphi_T \Big( \B^{ 3}_{A, \geq} \big(w, \conj{u}, \varphi_T\cdot \G_{B, \geq}[w,\conj{w}, u] \big) + \B^{ 3}_{A, \geq} \big(w, \conj{u}, \varphi_T\cdot \G_{B, >}[w,w, \conj{u}] \big)  \Big) \nonumber\\
	& +   \varphi_T \Big( \B^{ 3}_{A, >} \big(w, u, \conj{\varphi_T\cdot \G_{A, \geq}[w, \conj{u}, u]} \big) + \B^{ 3}_{A, >} \big(w, u, \conj{\varphi_T\cdot \G_{A, >}[w,u, \conj{u}]} \big) \Big) \nonumber\\
	& + \varphi_T \Big( \B^{ 3}_{A, >} \big(w, u, \conj{\varphi_T\cdot \G_{B, \geq}[w,\conj{w}, u]} \big) + \B^{ 3}_{A,>} \big(w, u, \conj{\varphi_T\cdot \G_{B, >}[w,w, \conj{u}]} \big)  \Big) \nonumber\\
	& +   \varphi_T\Big( \B^{ 3}_{B, \geq} \big(w,\conj{w}, \varphi_T\cdot \G_{A, \geq}[w, \conj{u}, u] \big) + \B^{ 3}_{B, \geq} \big(w,\conj{w}, \varphi_T\cdot \G_{A, >}[w, u, \conj{u}] \big) \Big) \nonumber\\
	& +\varphi_T \Big(\B^{ 3}_{B, \geq} \big(w,\conj{w}, \varphi_T\cdot \G_{B, \geq}[w, \conj{w}, u] \big) + \B^{ 3}_{B, \geq} \big(w,\conj{w}, \varphi_T\cdot \G_{B, >}[w, w,\conj{u}] \big) \Big) \nonumber\\
	& +   \varphi_T\Big( \B^{ 3}_{B, >} \big(w,w,\conj{ \varphi_T\cdot \G_{A, \geq}[w, \conj{u}, u] } \big) + \B^{ 3}_{B, >} \big(w,w, \conj{ \varphi_T\cdot \G_{A, >}[w, u, \conj{u}]} \big) \Big)\nonumber \\
	& +\varphi_T \Big(\B^{ 3}_{B, >} \big(w,w,\conj{ \varphi_T\cdot \G_{B, \geq}[w, \conj{w}, u]} \big) + \B^{ 3}_{B, >} \big(w,w,\conj{ \varphi_T\cdot \G_{B, >}[w, w,\conj{u}]} \big) \Big) . \label{eq:w}
	\end{align}

	% ---------------------------------------BIBLIOGRAPHY-----------------------------------
	
	\begin{ack}\rm 
		The author would like to thank her advisor, Tadahiro Oh, for suggesting the problem and for his continuous support throughout the work. A.C.~acknowledges support from Tadahiro Oh’s ERC Starting Grant (no.
		637995 ProbDynDispEq) and the Maxwell Institute Graduate School in Analysis and its Applications, a Centre for Doctoral Training funded by the UK Engineering and Physical Sciences Research Council (grant EP/L016508/01), the Scottish Funding Council, Heriot-Watt University and the University of Edinburgh. The author would like to thank Younes Zine for careful proofreading, and Yuzhao Wang and Justin Forlano for helpful and interesting discussions.
		
	\end{ack}

\end{document}